\documentclass[11pt]{article}

\usepackage[mathscr]{euscript} 

\usepackage[top=1in, bottom=1.25in, left=0.75in, right=0.75in]{geometry}
\usepackage{hyperref}
\usepackage{amssymb}
\usepackage[linesnumbered,lined,ruled]{algorithm2e}
\usepackage{amsmath,amsthm,amsfonts,dsfont,mathtools,bm}
\usepackage{url}
\usepackage{rotating}
\usepackage{multirow}
\usepackage{color}
\usepackage{xcolor}
\usepackage{enumitem}
\DeclareSymbolFont{rsfs}{U}{rsfs}{m}{n}
\DeclareSymbolFontAlphabet{\mathscrsfs}{rsfs}

\hypersetup{
  colorlinks   = true, 
  urlcolor     = blue, 
  linkcolor    = blue, 
  citecolor   = blue 
}

\newtheorem{theorem}{Theorem}
\newtheorem{definition}[theorem]{Definition}

\newtheoremstyle{myremark} 
    {\topsep}                    
    {\topsep}                    
    {\rm}                        
    {}                           
    {\bf}                        
    {.}                          
    {.5em}                       
    {}  

\newtheorem{lemma}[theorem]{Lemma}

\newtheorem{corollary}[theorem]{Corollary}

\newtheorem{proposition}[theorem]{Proposition}

\usepackage{accents}
\newlength{\dhatheight}
\newcommand{\doublehat}[1]{%
    \settoheight{\dhatheight}{\ensuremath{\hat{#1}}}%
    \addtolength{\dhatheight}{-0.35ex}%
    \hat{\vphantom{\rule{1pt}{\dhatheight}}%
    \smash{\hat{#1}}}}

\theoremstyle{myremark}
\newtheorem{remark}{Remark}[section]

\renewcommand{\P}{\operatorname{\mathbb{P}}}
\newcommand{\Q}{\operatorname{\mathbb{Q}}}

\newcommand{\E}{\operatorname{\mathbb{E}}}

\newcommand{\Z}{\mathbb{Z}}
\newcommand{\R}{\mathbb{R}}

\newcommand{\T}{T}

\newcommand{\D}{\boldsymbol{D}}
\newcommand{\Dm}{\D(\m)}
\newcommand{\F}{\mathcal{F}}

\def\sH{{\mathscr{H}}}

\newcommand{\rmd}{\mathrm{d}}

\DeclareMathOperator{\cov}{\mathrm{cov}}

\renewcommand{\hat}{\widehat}

\DeclareMathOperator*{\argmin}{arg\,min}

\newcommand{\op}{\mbox{\tiny\rm op}}

\newcommand{\bA}{\bm{A}}
\newcommand{\bB}{\bm{B}}

\newcommand{\bI}{\bm{I}}
\newcommand{\bY}{\bm{Y}}

\newcommand{\bW}{\bm{W}}

\newcommand{\normal}{\mathcal N}

\newcommand{\cF}{{\mathcal F}}

\newcommand{\wtSigma}{\widetilde{\Sigma}}

\newcommand{\wtW}{\widetilde{W}}
\newcommand{\wtbW}{\widetilde{\boldsymbol W}}

\newcommand{\wtsigma}{\widetilde{\sigma}}
\newcommand{\wtgamma}{\widetilde{\gamma}}
\newcommand{\wtbeta}{\widetilde{\beta}}

\newcommand{\wtbz}{\widetilde{\bz}}
\newcommand{\wtbu}{\widetilde{\bu}}
\newcommand{\wtbv}{\widetilde{\bv}}
\newcommand{\wtm}{\widetilde{\m}}

\newcommand{\Unif}{{\sf Unif}}
\newcommand{\onu}{\overline{\nu}}

\newcommand{\spher}{\mbox{\rm\tiny spher}}
\newcommand{\salg}{\mbox{\rm\tiny alg}}

\newcommand{\sdyn}{\mbox{\rm\tiny dyn}}

\newcommand{\sTAP}{\mbox{\rm\tiny TAP}}
\newcommand{\ons}{{\rm ONS}}
\newcommand{\eps}{\varepsilon}
\newcommand{\n}{\boldsymbol{n}}
\newcommand{\kalg}{k_{\salg}}
\DeclareMathOperator{\atanh}{arctanh}

\newcommand{\Dh}{D_{-\bh}}

\newcommand{\pl}{{\mbox{\rm\tiny pl}}}
\newcommand{\rd}{{\mbox{\rm\tiny rd}}}

\def\hq{\hat{q}}

\def\GOE{{\sf GOE}}

\def\info{{\sf I}}
\def\fS{{\mathfrak{S}}}

\def\bc{{\boldsymbol c}}
\def\RS{{\rm RS}}

\def\bgamma{{\boldsymbol \gamma}}
\def\bG{{\boldsymbol{G}}}
\def\obG{{\overline{\boldsymbol{G}}}}
\def\bg{{\boldsymbol{G}}}
\def\bSigma{{\boldsymbol{\Sigma}}}
\def\bD{{\boldsymbol{D}}}
\def\bu{{\boldsymbol{u}}}
\def\bv{{\boldsymbol{v}}}
\def\bw{{\boldsymbol{w}}}
\def\bx{{\boldsymbol{x}}}

\def\by{{\boldsymbol{y}}}
\def\hby{\hat{\boldsymbol y}}
\def\bz{{\boldsymbol{z}}}
\def\hbz{\hat{\boldsymbol{z}}}
\def\bh{h}

\def\m{{\boldsymbol{m}}}
\def\g{{\boldsymbol{g}}}
\def\hm{\hat{\boldsymbol{m}}}
\def\b0{{\boldsymbol{0}}}
\def\bfone{{\boldsymbol{1}}}
\def\normal{{\sf N}}
\def\AMP{{\sf AMP}}
\def\NGD{{\sf NGD}}
\def\tG{\widetilde{G}}
\def\tF{\widetilde{F}}

\def\bfe{{\boldsymbol e}}

\def\sNGD{\mbox{\tiny \sf NGD}}
\def\sAMP{\mbox{\tiny \sf AMP}}
\def\ALG{{\sf ALG}}
\def\sNGD{\mbox{\tiny \sf NGD}}

\def\sABXY{\mbox{\tiny \rm ABXY}}
\def\sEA{\mbox{\tiny \rm EA}}
\def\ssh{\mbox{\tiny \rm sh}}

\def\Lip{{\rm Lip}}

\def\MSE{{\sf MSE}}
\def\Par{{\sf P}}
\def\ed{\stackrel{{\rm d}}{=}}

\def\cuP{\mathscrsfs{P}}
\def\cuF{\mathscrsfs{F}}
\def\cuG{\mathscrsfs{G}}

\def\cL{{\mathcal L}}

\def\obeta{\bar{\beta}}
\def\oxi{\hat{\xi}}
\def\Treg{{\Gamma}}

\def\hcuF{\mathscrsfs{F}}

\def\Cov{{\rm Cov}}
\def\oH{\bar{H}}
\def\<{{\langle}}
\def\>{{\rangle}}
\def\de{{\rm d}}
\def\cC{{\mathcal C}}
\def\bX{{\boldsymbol{X}}}
\def\bY{{\boldsymbol{Y}}}
\def\sb{{\sf b}}
\DeclareMathOperator*{\plim}{p-lim}

\def\id{{\boldsymbol I}}

\def\Ball{{\sf B}}
\def\sT{{\sf T}}

\def\bfzero{{\boldsymbol 0}}

\def\lt{\left}
\def\rt{\right}

\def\la{\langle}
\def\ra{\rangle}

\def\sH{{\mathscr{H}}}

\newcommand{\Gp}[1]{{\boldsymbol G}^{(#1)}}

\def\bbE{{\mathbb{E}}}

\def\bbR{{\mathbb{R}}}

\def\diam{{\rm diam}}
\def\cC{{\mathcal C}}
\def\sep{\mathrm{sep}}

\numberwithin{equation}{section}
\numberwithin{theorem}{section}

\begin{document}

\title{Sampling from Mean-Field Gibbs Measures via\\ Diffusion Processes}

\author{Ahmed El Alaoui\thanks{Department of Statistics and Data Science, Cornell University}, 
\;\; Andrea Montanari\thanks{Department of Statistics and Department of Mathematics, 
Stanford University}, \;\; Mark Sellke\thanks{Department of Statistics, Harvard University}}

\date{}
\maketitle

\begin{abstract}
We consider Ising mixed $p$-spin glasses at high-temperature and without external field, and study the problem of sampling from the Gibbs distribution 
$\mu$ in polynomial time. We develop a new sampling 
algorithm with complexity of the same order as evaluating the gradient of the Hamiltonian
and, in particular, at most linear in the input size.
We prove that, at sufficiently high-temperature, it produces samples from a distribution $\mu^{\salg}$
which is close in normalized Wasserstein distance to $\mu$.
Namely, there exists a coupling of $\mu$ and  $\mu^{\salg}$ such that
if $(\bx,\bx^{\salg})\in\{-1,+1\}^n\times \{-1,+1\}^n$ is a pair drawn from this coupling,
then $n^{-1}\E\{\|\bx-\bx^{\salg}\|_2^2\}=o_n(1)$. For the case of 
the Sherrington-Kirkpatrick model, our algorithm succeeds in the full replica-symmetric phase. 
Previously, \cite{adhikari2022spectral,anari2023universality} showed that  Glauber dynamics 
succeeds in sampling under a stronger assumption on the temperature. (However, these works
prove sampling in total variation distance.)

We complement this result with a negative one for sampling algorithms satisfying a certain 
`stability' property, which is verified by many standard techniques.
 No stable algorithm can approximately sample at temperatures below the onset of 
shattering, even under the normalized Wasserstein metric.
Further, no algorithm can sample at temperatures below the onset of replica symmetry breaking.

Our sampling method implements a discretized version of a diffusion process
that has become recently popular in machine learning under the name of `denoising diffusion.'
We derive the same process from the general construction of stochastic localization. 
Implementing the diffusion process requires to efficiently
approximate the mean of  the tilted measure. To this end, we use an approximate message passing 
algorithm that, as we prove, achieves sufficiently accurate mean estimation.  
\end{abstract}

\tableofcontents

\section{Introduction}

\subsection{Background: Sampling from spin glass Gibbs measures}

Fix an integer $P \ge 2$ and a sequence $(c_p)_{2\leq p\leq P}$ 
 with $c_p\ge 0$. For each $2\leq p\leq P$, let
 $\Gp{p} \in \lt(\bbR^n\rt)^{\otimes p}$ be an independent $p$-tensor with i.i.d. 
 $\normal(0,1)$ entries and $\bg=(\Gp{p})_{2\leq p\leq P}$. 
 The mixed $p$-spin Hamiltonian $H_n$ is the function $H_n:\R^n\to\R$
 defined by
\begin{equation}
  \label{eq:def-hamiltonian}
  H_n(\bx) =\sum_{2\leq p\leq P} \frac{c_p}{n^{(p-1)/2}} \,\la \Gp{p}, \bx^{\otimes p} \ra\, .
\end{equation}
Setting
\[
  \xi(t)=\sum_{p=2}^P c_p^2 t^p\, ,
\]
$H_n$ is equivalently characterized as the centered Gaussian process on $\bbR^n$ with covariance
\[
  \bbE\big[H_n(\bx)H_n(\bx')\big] 
  =
  n\xi(\langle \bx,\bx'\rangle/n)\, .
\]
The associated (Ising) Gibbs measure 
is the probability distribution over $\{-1,+1\}^n$ given by 
\begin{align}
\label{eq:sk}
  \mu_{\bg}(\bx) = \frac{1}{Z(\beta,\bG)}\, e^{\beta H_n(\bx)} \, ,\;\;\;\;
  \bx\in\{+1,-1\}^n\, ,
\end{align}  
where $\beta \ge  0$ is referred to as the  inverse temperature. The parameter $\beta$ 
is fixed and we will leave implicit the dependence of $\mu$ upon $\beta$,
 unless mentioned otherwise. We let $\sH_n$ denote the space of tensor sequences 
 $\bg=(\Gp{p})_{2\leq p\leq P}$, $\Gp{p}\in(\bbR^n)^{\otimes p}$.
 The case $c_2\neq 0$, $c_p = 0$ for all $p>2$ is also known as
  Sherrington-Kirkpatrick (SK) model, and has played a special historical
  role in the theory of mean-field spin glasses \cite{sherrington1975solvable}.

The variant of model \eqref{eq:sk}, in which the uniform measure over
$\{-1,+1\}^n$ is replaced by the uniform measure over the sphere of radius $\sqrt{n}$, 
is also known as the `spherical $p$-spin model.' Our treatment will be focused on the Ising
case here, generalizing it to the spherical one is straightforward.
While model \eqref{eq:def-hamiltonian} might appear somewhat exotic to
the reader encountering it for the first time, it is the prototype of large
class of high-dimensional probability measures with a complex structure,
including examples from optimization, high dimensional statistics, machine learning 
\cite{MezardMontanari,montanari2022short}.

In this paper, we consider the problem of efficiently sampling from the measure \eqref{eq:sk}. 
We seek a 
randomized algorithm that accepts as input the sequence $\bg=(\Gp{p})_{2\leq p\leq P}$ and 
generates $\bx^{\salg}\sim \mu^{\salg}_{\bg}$,
such that: 
\begin{enumerate}
  \item The algorithm runs in polynomial time for any $\bg$.
  \item The distribution  $\mu^{\salg}_{\bg}$ is close to $\mu_{\bg}$ for 
typical realizations of $\bG$, in the sense that 
\[
  \E\big[{\sf dist}(\mu_{\bg},\mu^{\salg}_{\bg})\big]=o_n(1) \, ,
\] 
for some choice of distance ${\sf dist}$ on probability measures.
\end{enumerate}

Gibbs sampling, also known in this context as Glauber dynamics, provides
an algorithm to approximately sample from $\mu_{\bg}$. However, standard techniques to 
bound its mixing time (e.g., Dobrushin condition \cite{aizenman1987rapid})
only imply polynomial mixing for a vanishing interval of temperatures
$\beta=O(n^{-1/2})$.
By contrast, for more than forty years physicists \cite{sompolinsky1981dynamic,SpinGlass,kirkpatrick1987connections,
cugliandolo1993analytical,bouchaud1998out} have conjectured
fast convergence to equilibrium for a non-vanishing interval of temperatures.
Namely, at least for certain observables, they predict convergence to 
within a $\eps$-error from equilibrium values 
(Gibbs averages) in a physical time of order one (i.e., $O(n)$ Glauber updates) for all 
$\beta<\beta_{\sdyn}(\xi)$.
The inverse temperature  $\beta_{\sdyn}(\xi)$ is strictly positive and known as 
the inverse critical temperature for the dynamical phase transition. 

The location of $\beta_{\sdyn}(\xi)$ depends on the nature of the so-called 
replica symmetry breaking (RSB) phase transition, occurring at $\beta_c(\xi)$
(see \cite{panchenko2013sherrington} for a rigorous introduction to RSB; the reader unfamiliar
with this notion can safely skip the next paragraph).
Two scenarios are considered in the physics literature:
\begin{itemize}
\item In the first scenario, the phase transition at $\beta_c(\xi)$ is `continuous.'
Namely, denoting by $q_{\sEA}(\beta)$ the supremum of the overlap distribution, we have 
$q_{\sEA}(\beta)\downarrow 0$ as $\beta\downarrow \beta_c(\xi)$. 
In this case, it is predicted that $\beta_{\sdyn}(\xi) = \beta_c(\xi)$.
This is expected to be the case for the SK model and small perturbations of it.
In this case $\beta_{\sdyn}(\xi) = \beta_c(\xi) = 1/\xi''(0)$. 
\item In the  second scenario, the phase transition at $\beta_c(\xi)$ is `discontinuous.'
Namely, we have 
$\lim_{\beta\downarrow\beta_c}q_{\sEA}(\beta)> 0$ strictly. 
In this case, it is predicted that $\beta_{\sdyn}(\xi) <\beta_c(\xi)$ strictly,
and that $\beta_{\sdyn}(\xi)$ coincides with the temperature
for onset of shattering $\beta_{\ssh}(\xi)$ (see \cite{alaoui2023shattering,gamarnik2023shattering} for
recent references on the latter).

This scenario is expected to hold for `pure' $p$-spin models, i.e., if $\xi(t)= \xi_p(t) := t^p$,
$p\ge 3$. 
We provide the explicit (conjectural) formula for $\beta_{\sdyn}(\xi)$
in Appendix \ref{app:prelim}.
For large $p$, this formula yields 
\begin{align}
\beta_{\sdyn}(\xi_p) =\sqrt{\frac{2\log p}{p}}\cdot \big(1+o_p(1)\big)\, ,
\end{align}
while $\beta_c(\xi_p) = \sqrt{2\log 2} \cdot (1+o_p(1))$, a fact first conjectured in \cite{gross1984simplest}. 
\end{itemize}
A quite general argument shows that mixing of Glauber dynamics (or Langevin 
dynamics for the spherical model) must be slow beyond the shattering phase transition 
$\beta>\beta_{\ssh}(\xi)$
(see, e.g., \cite[Appendix D]{montanari2006rigorous} or \cite{ben2018spectral}).

Significant progress on the sampling question was achieved only recently.
For the SK model,  Bauerschmidt and Bodineau \cite{bauerschmidt2019very} showed that, 
for $\beta<1/4$, the measure 
$\mu_{\bg}$ can be decomposed into a log-concave mixture of product measures.
They use this decomposition to prove that $\mu_{\bg}$ satisfies a log-Sobolev inequality, although
not for the Dirichlet form of Glauber 
dynamics\footnote{We note in passing that their result immediately suggests a sampling algorithm:
 sample from the log-concave mixture using Langevin dynamics, and then sample from the corresponding
 component using the product form.}.
Eldan, Koehler, Zeitouni \cite{eldan2021spectral} prove that, in the same region $\beta<1/4$,
 $\mu_{\bg}$ satisfies a Poincar\'e inequality for the Dirichlet form of Glauber dynamics. 
 As a consequence,  Glauber dynamics mixes in $O(n^2)$ spin flips in total variation distance. 
 This mixing time estimate was improved to $O(n\log n)$ by \cite{anari2021entropic} 
 using a modified log Sobolev inequality, see also \cite[Corollary 51]{chen2022localization}.
 The aforementioned results apply
 deterministically to any matrix $\bG^{(2)}$ satisfying $\beta(\lambda_{\max}(\bG^{(2)})-\lambda_{\min}(\bG^{(2)}))
 \le 1-\eps$ but are specific to quadratic Hamiltonians. 
 
 Even more recently, \cite{adhikari2022spectral} showed that a general $p$-spin models 
  also obey a Poincar\'e inequality for Glauber dynamics for $\beta\le \beta_{\sABXY}(\xi)$,
 by an induction over $n$. 
 However, the constant $\beta_{\sABXY}(\xi)$ is significantly smaller than the predicted threshold
 $\beta_{\sdyn}(\xi)$. In particular, for the pure $p$-spin model and large $p$, we have
 \begin{align}
 \beta_{\sABXY}(\xi_p) \asymp \frac{1}{\sqrt{p^3\log p}}\, .
 \end{align}
 A modified log-Sobolev inequality in the same regime of temperatures was proven by different techniques in
 \cite{anari2023universality}.
 
For \emph{spherical} spin glasses, it is shown in \cite{gheissari2019spectral} that Langevin dynamics
 has a polynomial spectral gap at high temperature. Meanwhile
  \cite{ben2018spectral} proves that at sufficiently low temperature, the mixing times of 
  Glauber and Langevin dynamics are exponentially large in Ising and spherical spin glasses, 
  respectively.

\subsection{Background: Sampling via diffusion processes}

In this paper we develop a different approach which is not based on a 
Monte Carlo Markov Chain strategy. We build on the well known remark that 
approximate sampling can be reduced to approximate computation of expectations of 
the measure $\mu_{\bg}$, and of a family of measures obtained from $\mu_{\bg}$.
One well known method to achieve this reduction is via sequential sampling 
\cite{jerrum1986random,chen2005sequential,blitzstein2011sequential}. 
A sequential sampling approach to $\mu_{\bg}$ would proceed as follows.
Order the variables $x_1,\dots, x_n\in \{-1,+1\}$ arbitrarily. At step $i$ compute
the marginal distribution of $x_i$, conditional to 
$x_1,\dots,x_{i-1}$ taking the previously chosen values:
$p^{(i)}_s := \mu_{\bg}(x_i=s|x_1,\dots,x_{i-1})$, $s\in\{-1,+1\}$. 
Fix $x_i=+1$ with probability $p^{(i)}_{+1}$ and  $x_i=-1$ with probability $p^{(i)}_{-1}$. 

We follow a different route, which is similar in spirit, but 
more convenient technically.
Our reduction was originally motivated by the stochastic localization technique of
\cite{eldan2020taming} but is in fact equivalent to the denoising diffusion method
of  \cite{sohl2015deep,song2019generative,ho2020denoising,song2021score}. 
Given any probability measure $\mu$ on $\R^n$ with finite second moment, positive time $t>0$,
and vector $\by \in \R^n$, define the tilted measure
\begin{equation}\label{eq:sktilted}
\mu_{\by,t}(\de \bx) :=\frac{1}{Z(\by)} e^{\<\by,\bx\>-\frac{t}{2}\|\bx\|_2^2}\, \mu(\de\bx )\, ,
\end{equation}  
and let its mean vector be
\begin{equation}\label{eq:meanvec}
\m(\by,t) := \int_{\R^n}  \bx \, \mu_{\by,t}(\de\bx)\,.
\end{equation}  
Consider the stochastic differential 
equation\footnote{If $\mu$ is has finite variance, then $\by \to \m(\by,t)$ is Lipschitz 
and so this  SDE is well posed with unique strong solution.} (SDE)
\begin{equation}\label{eq:GeneralSDE}
\rmd \by(t) = \m(\by(t),t) \rmd t + \rmd \bB(t),~~~~ \by(0)=0 \, ,
\end{equation}  
where $(\bB(t))_{t \ge 0}$ is a standard Brownian motion in $\R^n$. 
Then,  the measure-valued process 
$(\mu_{\by(t),t})_{t \ge 0}$ is a martingale
and (almost surely) $\mu_{\by(t),t}\Rightarrow \delta_{\bx^\star}$ as $t\to\infty$,
for some random $\bx^{\star}$
(i.e., the measure localizes). As a consequence of the martingale property, 
$\E[\int \varphi(\bx)\mu_{\by(t),t}(\de\bx)]$  is a constant for any 
bounded continuous function $\varphi$, whence
$\E[\varphi(\bx^{\star})] = \int \varphi(\bx)\mu(\de\bx)$. In other words, $\bx^{\star}$
is a sample from $\mu$.

An equivalent definition of this process is obtained by noting that
there exists a $\bx\sim \mu$, independent of a (different) Brownian motion $(\bB'(t))_{t\ge 0}$
such that
\begin{align}
\by(t) = t\bx+\bB'(t)\, ,
\end{align}
and $\mu_{\by,t}$ is nothing but the conditional distribution of $\bx$ given $\by(t)=\by$
\cite{el2022information}.
 For further information on this process, we refer to Section~\ref{sec:stochloc}.

In order to use this process as an algorithm to sample from the Gibbs measure
$\mu=\mu_{\bg}$, we need to overcome two problems:
\begin{itemize}
\item \emph{Mean computation.} We need to be able to compute the mean vector 
$\m(\by,t)$ efficiently. To this end, we use an approximate message passing (AMP) algorithm
for which we can   establish that
$\|\m(\by)-\hm_{\AMP}(\by)\|^2_2/n=o_n(1)$ along the algorithm trajectory.
(Note that the Gibbs measure is supported on vectors with $\|\bx\|_2^2=n$, and
hence the quadratic component of the tilt in Eq.~\eqref{eq:sktilted} drops out.
We will therefore write $\m(\by)$ or $\m(\bG,\by)$ instead of $\m(\by,t)$ for the
 mean of the Gibbs measure, and similarly for the mean approximation $\hm(\by)$.)
\item \emph{Discretization.} We need to discretize the SDE \eqref{eq:GeneralSDE} in time,
and still guarantee that the discretization closely tracks the original process.  
This is of course possible only if the approximate mean estimation 
map $\by\mapsto \hm(\by)$ is sufficiently regular. In fact, we will prove 
that the approximation we construct is Lipschitz continuous, with Lipschitz constant 
bounded as $n\to\infty$.
\end{itemize}
\begin{remark}
To the best of our knowledge, 
the idea of an algorithmic implementation of the stochastic localization process
was first presented in the conference version of this paper \cite{alaoui2022sampling}.
The present paper generalizes the approach of \cite{alaoui2022sampling}
to other mean-field spin glasses beyond the Sherrington-Kirkpatrick model.

After the conference publication \cite{alaoui2022sampling}, it became apparent 
that the latter (in its most standard form
 used here) is equivalent to the `denoising diffusions' method in machine learning.
 We refer to \cite{montanari2023sampling} for a discussion of the connection and 
 some generalizations.
 
A recent paper by
 Nam, Sly and Zhang~\cite{nam2022ising} uses the same process
  to show that the Ising measure on the infinite $k$-regular tree is a factor of IID,
  for $\beta\le C/\sqrt{k}$. This temperature threshold is particularly significant because
  it is within a constant factor from the Kesten--Stigum, or ``reconstruction", 
   threshold. Their construction can easily be transformed into a sampling algorithm.       
\end{remark}

\subsection{Summary of results}

In order to state our results, we define the normalized 2-Wasserstein 
distance between two probability 
measures $\mu , \nu$ on $\R^n$ with finite second moments as
\begin{equation}  
W_{2,n}(\mu,\nu)^2 =  \inf_{\pi \in \cC(\mu,\nu)} \frac{1}{n} 
\E_{\pi} \Big[\big\|\bX - \bY\big\|_2^2\Big] \, , 
\end{equation}   
where the infimum is over all couplings $(\bX,\bY) \sim \pi$ with marginals
 $\bX \sim \mu$ and 
$\bY \sim \nu$.  

In this paper, we establish two main results.
\begin{description}
\item[Sampling algorithm for $\beta<\bar\beta(\xi)$.] We prove that the strategy outlined above 
yields an algorithm which makes $O(1)$ queries to $\nabla H_n$ and samples from a distribution $\mu_{\bg}^{\salg}$ such that 
$W_{2,n}(\mu_{\bg}^{\salg},\mu_{\bg}) = o_{\mathbb P,n}(1)$. The complexity of 
computing $\nabla H_n$ is $O(n^P)$, i.e. linear in the size of the input. 

In the case of the SK model, we have $\bar\beta(\xi)=1/(2\xi''(0))$, which is a factor $2$
away from the conjectured critical temperature. For our other running example, the pure
$p$-spin model, we get $\bar\beta(\xi_p)\asymp 1/p$, which improves over state of the art
 \cite{adhikari2022spectral,anari2023universality}.
(However, the $W_{2,n}$ sampling guarantee is weaker than the guarantees of 
 \cite{adhikari2022spectral,anari2023universality}.)
 More interestingly, our approach is based on a very different algorithmic technique.
\item[Hardness for stable algorithms for $\beta>\beta_{\sdyn}(\xi)$.] We prove that no algorithm 
satisfying a certain \emph{stability}
property can sample from the Gibbs measure (under the
same criterion $W_{2,n}(\mu_{\bg}^{\salg},\mu_{\bg}) = o_{\mathbb P,n}(1)$) under
either of the following conditions: $(i)$~When replica symmetry is broken,
i.e., for $\beta>\beta_c(\xi)$; $(ii)$~In the shattering phase, which is conjectured to hold for 
discontinuous models when  $\beta\in (\beta_{\sdyn}(\xi),\beta_c(\xi))$.

 Roughly speaking, stability formalizes
the notion that the algorithm output 
behaves continuously with respect to the disorder $\bG$. 
\end{description}
Our hardness results  are proven using the notion of 
disorder chaos. This is analogous to the use of the \emph{overlap gap property} for
random optimization, estimation, and constraint satisfaction 
problems~\cite{gamarnik2014limits, rahman2017independent, gamarnik2017performance, 
chen2019suboptimality, gamarnik2019overlap, gamarnik2020optimization, wein2020independent,
 gamarnik2021partitioning, bresler2021ksat, gamarnik2021circuit, huang2021tight}. 
 While the overlap gap property has been used to rule out stable algorithms for this class of 
 problems, and variants have been used to rule out efficient sampling by specific Markov chain 
 algorithms, to the best of our knowledge this is the first argument ruling out stable sampling algorithms 
 using  disorder chaos. 
 
 In sampling there is no hidden solution or set of solutions to be found, 
 and therefore overlap gap arguments are not directly applicable. Instead, we argue directly 
 that the distribution to be sampled from is unstable in a $W_{2,n}$ sense at low temperature,
  and hence cannot be approximated by any stable algorithm.

\paragraph{Conference version.} A conference version of this paper was presented at
FOCS 2022 \cite{el2022sampling}, and ony only treated the case of the SK model. This version extends
\cite{el2022sampling} by treating general mixed $p$-spin models.

\subsection{Notations}

We use the standard big-Oh notation $O_n(\,\cdot\,)$, $o_n(\,\cdot\,)$ and so on, where
the subscript indicates the variable $n\to\infty$.
 We use $o_{n,\P}(1)$ 
for a quantity tending to $0$ in probability. If $X$ is a random variable, 
then $\mathcal L(X)$ indicates its law. The quantity $C(\beta)$ refers to a constant depending 
on $\beta$. For $\bx\in\R^n$ and $\rho\in\R_{\geq 0}$, we denote the open ball 
of center $\bx$ and radius $\rho$ by 
$\Ball^n(\bx;\rho):=\{\by\in \R^n: \|\by-\bx\|_2<\rho\}$. 
When the center is the origin, we will use the shorthand 
$\Ball^n(\rho) = \Ball^n(\bfzero;\rho)$.
The set of probability distributions over a measurable space $(\Omega,\cF)$
is denoted by $\cuP(\Omega)$. 
For a sequence of random variables $(X_n)$, we write $\plim_{n \to \infty} X_n = c$ if $X_n$ converges to the constant $c$ in probability. 


\section{Main results: I. Sampling algorithm for $\beta$ small enough}
\label{sec:Sampling} 

In this section we describe the sampling algorithm and formally
state the result of our analysis. As pointed out in the introduction, 
a main component of the algorithm is a procedure to approximate
the mean of the tilted Gibbs measure:
\begin{align}
\label{eq:mu_tilted}
  \mu_{\bG,\by}(\bx)
  :=
  \frac{1}{Z(\bG,\by)} e^{\beta H_n(\bx)+\langle \by,\bx\rangle},
  \quad \bx\in\{-1,+1\}^n\, .
\end{align}
We describe the algorithm to approximate this mean in Section 
\ref{sec:MeanApprox}, the overall sampling procedure 
(which uses this estimator as a subroutine) in Section \ref{sec:SamplingAlg},
and our Wasserstein-distance guarantee in  Section \ref{sec:SamplingTheorem}.

\subsection{Approximating the mean of the Gibbs measure}
\label{sec:MeanApprox}

\begin{algorithm}
\label{alg:Mean}
\DontPrintSemicolon 
\KwIn{Data $\bG\in\sH_n$, $\by\in \mathbb R^n$, parameters $\beta,\eta,\Treg>0$, $q\in (0,1)$, iteration numbers 
$K_{\sAMP}$, $K_{\sNGD}$.}
$\hm^{-1} = \bz^{0}= 0$,\\
\For{$k = 0,\cdots,K_{\sAMP}-1$} { 
$\hm^{k} = \tanh(\bz^{k} ) \, , ~~~~~~~ \hat{q}^k = \frac{1}{n}\sum_{i=1}^n \tanh^2(z^{k}_i) \, ,~~~~~~~\sb_{k}= \beta^2 
(1-\hat{q}^k) \xi''( \hat{q}^k)$\, ,\\
$\bz^{k+1} = \beta \nabla H_n
    \big(\hm^{k}\big) + \by - \sb_{k} \hm^{k-1}$\, , \label{eq:AMP-main-step}\\
}
$\bu^0 = \bz^{K_{\sAMP}}$, $\hm^{+,0} = \hm^{K_{\sAMP}}$, \label{alg:NGD-begin}\\
\For{$k = 0,\cdots,K_{\sNGD}-1$} { 
$\bu^{k+1} = \bu^k - \eta \cdot\nabla \widehat{\cuF}_{\sTAP}(\hm^{+,k};\by,q)$,   \\
$\hm^{+,k+1} = \tanh(\bu^{k+1})$,
}
\Return{$\hm^{+,K_{\sNGD}}$}\label{alg:NGD-end}\\
\caption{{\sc Mean of the tilted Gibbs measure}}
\end{algorithm}

We will denote our approximation of the mean of the Gibbs measure 
$\mu_{\bG,\by}$ by $\hm(\bG,\by)$, while the actual mean will be $\m(\bG,\by)$.

 The algorithm to compute
$\hm(\bG,\by)$ is given in Algorithm \ref{alg:Mean}, and is composed of two phases: 
\begin{enumerate}
\item An Approximate Message Passing (AMP) 
algorithm is run for $K_{\sAMP}$ iterations and constructs a first estimate  of
the mean. 
We denote by $\AMP(\bG,\by ; k)$ the estimate produced after $k$
AMP iterations
\begin{equation} 
\label{eq:AMP}
\AMP(\bG,\by ; k) := \hm^{k}\,.
\end{equation}
\item Natural gradient descent (NGD) is run for $K_{\sNGD}$ iterations
with initialization given by vector computed at the end of the first phase.
This phase attempts to minimize a modified version of the Thouless--Anderson--Palmer (TAP) free energy.
 The usual TAP free energy~\cite{thouless1977solution,talagrand2011mean1}  takes the form 
\begin{align}\label{eq:TAP}
\cuF_{\sTAP}(\m ; \by) &:= 
  -\beta H_n(\m) - \langle \by , \m \rangle - \sum_{i=1}^n h(m_i) - \ons\big(Q(\m)\big) \, ,\\
 \ons\big(Q\big) &:= \frac{\beta^2n}{2}\Big( \xi(1) - \xi(Q) - (1-Q)\xi'(Q) \Big)\, ,\\
  Q(\m) 
  &= 
  \frac{1}{n} \|\m\|^2 \, ,~~~\mbox{and}~~~ h(m) = -\frac{1+m}{2}\log \left(\frac{1+m}{2}\right) - \frac{1-m}{2}\log \left(\frac{1-m}{2}\right)  \, .
\end{align}
Our modified version replaces $\ons\big(Q(\m)\big)$ by a quadratic function
of $Q(\m)$, namely $\ons(q) + \ons'(q)(Q(\m)-q)+
n\Treg\beta(Q(\m)-q)^2/8$, where $q$, $\Treg$ are
 appropriately chosen constants (see line~\ref{line:q} for Algorithm~\ref{alg:Sampling}):
\begin{equation}\label{eq:TAP_reg}
 \widehat{\cuF}_{\sTAP}(\m ; \by, q) :=  -\beta H_n(\m) - \langle \by , \m \rangle - \sum_{i=1}^n h(m_i) - \ons(q) - \ons'(q)(Q(\m)-q) 
 +\frac{n\Treg\beta}{8}\big(Q(\m)-q\big)^2\, .
\end{equation}
\end{enumerate}
This procedure is motivated by the objective to find a solution
of  the so-called TAP equations. Indeed  the Gibbs mean of a high-temperature spin glass is known
to be an approximate solution of the TAP equations
 \cite{SpinGlass,talagrand2011mean1}. By construction, 
both fixed points of AMP and stationary points of the 
 function $\widehat{\cuF}_{\sTAP}(\m ; \by, q)$ solve the TAP equations
 (albeit of slightly different versions of these equations).
 
 We will show in Lemma~\ref{lem:TAP-stationary} that the first stage above constructs 
 an approximate stationary point $\hm^{K_{\sAMP}}$ for $\widehat{\cuF}_{\sTAP}(\m ; \by, q)$. 
 Next, we prove that $\widehat{\cuF}_{\sTAP}(\,\cdot\, ; \by, q)$  is
 suitably convex in a neighborhood of $\hm^{K_{\sAMP}}$, 
 cf. Lemma \ref{lem:local-convex} .
As a consequence, the second stage quickly converges to a stationary point (in fact, a local minimum)
of  $\widehat{\cuF}_{\sTAP}(\,\cdot\, ; \by, q)$.
On the other hand, because of the local strong convexity, and the fact that
$\widehat{\cuF}_{\sTAP}(\,\cdot\, ; \by, q)$ depends on $\by$ through a linear function
 of $\m$, the stationary point is Lipschitz continuous in $\by$.
 Hence, the second stage does not reduce significantly the estimation 
 error but allows us to control the error incurred by
 discretizing time in line~\ref{step:DiscreteSDE} of Algorithm, and 
 the accumulation of error across iterations.

Let us emphasize that we use this two-stage algorithms for technical reasons. 
Indeed a simpler algorithm, that runs AMP for a larger number of 
iterations and does not run NGD at all, is expected to work.
However, proving Lipschitz continuity of the resulting estimator is currently 
an open problem.  
The hybrid algorithm above allows us to exploit known properties of AMP 
(precise analysis via state evolution)
and of $\widehat{\cuF}_{\sTAP}(\m ; \by, q)$ (Lipschitz continuity of the minimizer in $\by$). 

The modification of the Onsager term $\ons$ introduced in Eq.~\eqref{eq:TAP_reg} is also 
introduced for technical reasons. Indeed it allows for a simplified analysis of the Hessian
 of the TAP functional, which in turn guarantees the previously mentioned Lipschitz property.
  The free parameter $q$ will be set close to the theoretical prediction
  for
  \begin{align}
  q_*(\beta,t) = \lim_{k\to\infty} \plim_{n\to\infty} Q(\hm^k) \, .\label{eq:LimitQstar}
  \end{align}
  (These limits will be shown to exist as a consequence of the theory developed below.)
Therefore replacing $\cuF_{\sTAP}(\m ; \by)$ by $\widehat{\cuF}_{\sTAP}(\,\cdot\, ; \by, q)$
is expected to have  a small effect on the location of the minimum.

\subsection{Sampling via stochastic localization}
\label{sec:SamplingAlg}

\begin{algorithm}
\label{alg:Sampling}
\DontPrintSemicolon 
\KwIn{Data $\bG\in\sH_n$, parameters $(\beta,\eta,K_{\sAMP},K_{\sNGD},L,\delta)$}
$\hby_0=0$,\\
\For{$\ell = 0,\cdots,L-1$} {
Draw $\bw_{\ell+1}\sim\normal(0,\bI_n)$ independent of everything so far;\\
Set $q= q_{*}(\beta,t=\ell\delta)$; \label{line:q}\\
Set $\hm(\bG,\hby_{\ell})$ the output of Algorithm \ref{alg:Mean}, 
with parameters $(\beta,\eta,q,K_{\sAMP},K_{\sNGD})$;\\
Update $\hby_{\ell+1} = \hby_{\ell} + \hm(\bG,\hby_{\ell}) \, \delta + \sqrt{\delta} \, \bw_{\ell+1}$
\label{step:DiscreteSDE}
}
Set $\hm(\bG,\hby_{L})$ the output of Algorithm \ref{alg:Mean}, 
with parameters $(\eta,q,K_{\sAMP},K_{\sNGD})$;\\
Draw $\{x_i^{\salg}\}_{i\le n}$  conditionally independent
with 
$\E[x_i^{\salg}|\by,\{\bw_{\ell}\}] = \widehat{m}_i(\bG,\hby_{L})$\\
\Return{$\bx^{\salg}$}\;
\caption{{\sc Approximate sampling from the Gibbs measure}}
\end{algorithm}

Our sampling algorithm is presented as Algorithm~\ref{alg:Sampling}.
The algorithm makes uses of constants $q_*:=q_*(\beta,t)$
which have the interpretation given in Eq.~\eqref{eq:LimitQstar}.
As explained in Section \ref{sec:AMP}, $q_*(\beta,t)$ can be computed by a simple
recursive characterization, known as `state evolution.' 
For $W\sim \normal(0,1)$ a standard Gaussian,  and $k,\beta,t\ge 0$,
define
\begin{equation}\label{eq:q_k}
  q_{k+1} = \E\Big[\tanh\Big(\sqrt{\beta^2\xi'(q_k)+t}\,W + \beta^2\xi'(q_k)+t \Big)^2\Big]\, ,
  \quad q_0=0\, ,
  \quad
  q_* = \lim_{k\to\infty}q_k\, .
\end{equation}
This iteration can be implemented via a one-dimensional integral,
and the limit $q_*$ is approached exponentially fast in $k$ 
for the values of $\beta$ of interest here (see Lemma \ref{lem:properties} below). 
The values
$q_*(\beta,t=\ell\delta)$ for $\ell\in\{0,\dots,L\}$ can be precomputed
and are independent of the input $\bG$. For the sake of simplicity, we will neglect errors
in this calculation. 

The core of the sampling procedure is step \ref{step:DiscreteSDE},
which is a standard Euler discretization of the SDE \eqref{eq:GeneralSDE},
with step size $\delta$, over the time interval $[0,T]$, $T=L\delta$. The
mean of the Gibbs measure $\m(\bG,\by)$ is replaced by the output 
of Algorithm \ref{alg:Mean} which we recall is denoted by $\hm(\bG,\by)$.
We reproduce the Euler iteration here for future reference
\begin{equation}\label{eq:approx}
\hby_{\ell+1} = \hby_{\ell} + \hm(\bG,\hby_{\ell}) \, \delta + \sqrt{\delta} \, \bw_{\ell+1}\, .
\end{equation}

The output of the iteration is $\hm(\bG,\hby_{L})$, which should be thought of as an approximation of 
$\m(\bG,\by(T))$, $T=L\delta$, which is the mean of $\mu_{\bG,\by(T)}$. 
According to the discussion in the introduction,
for large $T$, $\mu_{\bG,\by(T)}$ concentrates around $\bx^{\star}\sim \mu_{\bg}$. In other words,
$\m(\bG,\by(T))$ is close to the corner $\bx^{\star}$ of the hypercube. We
round its coordinates independently to produce the output $\bx^{\salg}$.

\subsection{Theoretical guarantee}
\label{sec:SamplingTheorem}

Our main positive result is the following.
\begin{theorem}
\label{thm:main}
Define the functions $h(x):=-((1+x)/2)\log((1+x)/2) - ((1+x)/2)\log((1+x)/2)$,
$\psi(\gamma) = \E[\tanh(\gamma+\sqrt{\gamma} \, G)]$ (with expectation with respect to 
$G\sim\normal(0,1)$), and $\phi=\psi^{-1}$ its inverse.
Let $\bar{\beta}(\xi):=\min(\beta_1(\xi),\beta_2(\xi),\beta_3(\xi))>0$, where 
(for a sufficiently small numerical constant $C_0$)
\begin{align}
 \beta_1(\xi)  &:=  \inf_{q\in (0,1)}   \sqrt{\frac{\phi'(q)}{\xi''(q)}} \, ,
 \label{eq:Beta0Def}\\ 
\beta_2(\xi)&:= \sup\Big \{\beta>0: \; \beta^2\xi(q)+h(q)-\log(2)<0 \;\;\forall q\in (0,1)\Big\}\, ,
\label{eq:Beta2Def}\\
\beta_3(\xi)&:= \begin{dcases}
\frac{1}{2\sqrt{\xi''(0)}} & \mbox{ if $\xi(t) = c_2^2t^2$,}\\
\frac{C_0}{\sqrt{\xi''(1)\log\oxi^{(8)}(1)}} & \mbox{ otherwise,}
\end{dcases}
\label{eq:Beta3Def}
\end{align}
 depend uniquely on the mixture polynomial $\xi$ (here
 $\oxi^{(\ell)}(1)  :=\sum_{p=2}^Pc_p^2p^\ell$). 
 Then, for any $\eps>0$ and $\beta< \bar{\beta}(\xi)$,
  there exist 
$\eta,K_{\sAMP},K_{\sNGD},\Treg,\delta$ independent of $n$, so that the following holds.
The sampling algorithm~ takes as input $\bG$  and
parameters  $(\eta,K_{\sAMP},K_{\sNGD},\Treg,\delta)$
and outputs a random point
 $\bx^{\salg} \in \{-1,+1\}^n$ with law $\mu_{\bg}^{\salg}$ such that with probability $1-o_n(1)$ 
 over $\bG$,
\begin{equation}
\label{eq:main}
    W_{2,n}( \mu_{\bg}^{\salg} , \mu_{\bg}^{\phantom{\salg}} ) \leq \eps \, .
\end{equation}    
The total complexity of this algorithm is $O(n^2)$.
\end{theorem}

\begin{remark}\label{rmk:mainThm}
Each of the three inverse temperatures $\beta_1(\xi),\beta_2(\xi),\beta_3(\xi)$
originates in a different part of the analysis. Here we briefly comment on these conditions:        
\begin{itemize}
\item
 $\beta_1$ is the inverse temperature below which the iteration defining $q_k(\beta,t)$,
  Eq.~\eqref{eq:q_k} has a unique fixed point for all $t>0$, i.e., the large $k$ limit $q_*(\beta,t)$ 
  does not depend on the initial value $q_0$; see Lemma~\ref{lem:contraction}. 
  In other words, we have the equivalent characterization:
\begin{equation}
\label{eq:beta1}
\beta_1(\xi):= \sup\big\{\beta:\; q = \psi(t+\tilde\beta^2\xi'(q)) 
\mbox{ has a unique solution for all }(\tilde\beta,t)\in [0,\beta)\times\R_{\ge 0} 
\big\}
   \, .
\end{equation}
The uniqueness of this fixed point guarantees that the two-step Algorithm~\ref{alg:Mean} is able to accurately approximate the mean $\m(\bG,\by_{t})$ of the tilted Gibbs measure $\mu_{\bG,\by(t)}$ for all $t>0$.      
\item $\beta_2$ is a sufficient upper bound on $\beta$ for contiguity between the
 joint distribution of $(\bG,(\by(t))_{t \in [0,T]})$ and a certain planted model 
 defined in Section~\ref{sec:planted} to hold, see Theorem~\ref{thm:fluctuations_Z}. 
 In fact the contiguity holds for all $\beta<\beta_c$ as shown very recently in \cite[Corollary 2.4]{dey2022hypergraph} using cluster expansion techniques.
 
 By Taylor expanding the function $\beta^2\xi(q)+h(q)-\log(2)$ around $q=0$,
 we observe that: $(1)$~$\beta_2\le 1/\sqrt{\xi''(0)}$; $(2)$~$\beta_2= 1/\sqrt{\xi''(0)}$
 for the SK model ($\xi(t) = c_2^2 t^2$) and small perturbations of it. 
 \item Finally, the condition $\beta<\beta_3$
 is sufficient for  for the (modified) TAP free energy functional $\widehat{\cuF}_{\sTAP}$
 to be globally convex (with respect with a suitable metric that is relevant for NGD). 
 This is required in the analysis of the second stage of Algorithm~\ref{alg:Mean}. 
\end{itemize}
\end{remark}

\begin{remark}[Condition $\beta<\beta_1$]\label{rmk:Beta0}
It turns out that the condition $\beta<\beta_1$ is necessary for the algorithm presented here to be
effective, unlike conditions $\beta<\beta_2$ (which can be weakened to
$\beta<\beta_c$ using \cite{dey2022hypergraph}) and $\beta<\beta_3$ (which is a proof 
artifact).
 Indeed, as proven in Section \ref{sec:ConvergenceToFP} (see Lemma \ref{lem:Beta0}),
for $\beta>\beta_1$, the state evolution iteration \eqref{eq:q_k} has multiple fixed points
for some $t\in \R_{>0}$.
This signals a phase transition in the tilted measure $\mu_t$.
In the case of sequential sampling, this phase transition was first studied heuristically in
\cite{montanari2007solving,ricci2009cavity}.
See~\cite{ghio2023sampling} for a recent heuristic study 
of the phase diagram case of the $3$-spin model, under the stochastic localization
process considered here. 

A possible solution to this limitation could be to use a different stochastic localization mechanism,
for instance by evolving the randomness $\bG$ with $t$. See \cite{montanari2023posterior,montanari2023sampling}  
for proposals of this type. However as discussed in Remark \ref{rmk:BarLargeP}
below, this is not the bottleneck to the present proof.
\end{remark}

\begin{remark}[SK model]
In the case of the SK model $\xi(t) =t^2/2$,
a direct calculation with the formulas above yields $\beta_1 = \beta_2  = 1$
and  $\beta_3 = 1/2$.
 Indeed, the conference version of this paper~\cite{alaoui2022sampling} was limited to the
 SK case and assumed the condition $\beta<1/2$, which we recover here. 
 In subsequent work, Celentano~\cite{celentano2022sudakov} removed this spurious condition thereby 
 showing that the above algorithm succeeds in the sense of Theorem~\ref{thm:main} for all $\beta<1$.   
 The temperature $\beta_c=1$ is the critical temperature for replica symmetry breaking
 and, as discussed in the next section, a large class of algorithms are known to fail for $\beta>\beta_c$.
\end{remark}

\begin{remark}[Pure $p$-spin]\label{rmk:BarLargeP}
For the pure $p$-spin model $\xi(t) = \xi_p(t):= t^p$, and $p$ large, it is possible to
show that the most constraining condition is given by $\beta_3$.
Namely we have $\beta_1(\xi_p) \asymp 1/\sqrt{p}$, $\beta_2(\xi_p)\asymp 1$ and 
$\beta_3 =  (\overline{C}/p\sqrt{\log p})(1+o_p(1))$,
implying 
\begin{align}
\obeta(\xi_p) = \frac{\overline{C}}{p\sqrt{\log p}}\cdot\big(1+o_p(1)\big)\, .
\end{align}
As anticipated, we cover a larger range of temperatures than in the 
recent works  \cite{adhikari2022spectral,anari2023universality}.
The bottleneck of our analysis is in proving local convexity of the
TAP free energy, an obstacle that can probably overcome using the
 approach of \cite{celentano2022sudakov}. 
\end{remark}

\section{Main results: II. Hardness of stable sampling from RSB or shattering}

Our sampling algorithm enjoys stability properties
with respect to changes in the inverse temperature $\beta$ and the disorder $\bG$
which are shared by many natural efficient algorithms. We will 
use the fact that the actual Gibbs measure does not enjoy this stability property
$\beta>\beta_c$ to conclude that sampling is hard for
all stable algorithms. Here $\beta_c=\beta_c(\xi)$ is the critical temperature for 
replica symmetry breaking. We refer to Appendix \ref{app:prelim} 
for a review of the analytic characterization of 
$\beta_c(\xi)$. 
We will finally show that the same conclusions on hardness follow from the existence of a shattered decomposition
 for $\mu_{\bG,\beta}$, which was recently proven to exist for pure models in a non-trivial 
 portion of the replica symmetric phase in \cite{alaoui2023shattering,gamarnik2023shattering}.

Throughout this section, we denote the
Gibbs and algorithmic output distributions by $\mu_{\bG,\beta}$ and $\mu_{\bG,\beta}^{\salg}$ 
respectively to emphasize the dependence on $\beta$.
\begin{definition}\label{def:Stable}
Let $\{\ALG_n\}_{n\geq 1}$ be a family of randomized sampling algorithms, i.e., measurable 
maps  
\[
    \ALG_n:(\bG,\beta,\omega) \mapsto \ALG_n(\bG,\beta,\omega)\in [-1,1]^n \, ,
\]
where $\omega$ is a random seed (a point in a probability space $(\Omega,\cF,\P)$.) 
Let $\bG_1$ and $\bG_0$ be independent copies of the disorder, and consider perturbations 
\[
  \bG_s=\sqrt{1-s^2} \bG_0 + s \bG_1
\]
for $s \in [0,1]$, i.e., $\bG_s^{(p)}=\sqrt{1-s^2} \bG_0^{(p)} + s \bG_1 ^{(p)}$ for $2\leq p\leq P$.
Finally, denote by $\mu_{\bG_s,\beta}^{\salg}$ the law of the algorithm output,
i.e., the distribution of $\ALG_n(\bG_s,\beta,\omega)$ for $\omega\sim \P$ independent 
of $\bG_s,\beta$ which are fixed.

We say $\ALG_n$ is \emph{stable with respect to disorder}, at inverse temperature $\beta$, if
\begin{align}
\lim_{s\to 0} \, \plim_{n\to\infty} \, W_{2,n}(\mu_{\bG,\beta}^{\salg},\mu_{\bG_s,\beta}^{\salg})=0\, .
\end{align}

We say $\ALG_n$ is \emph{stable with respect to temperature} at inverse temperature $\beta$, if
\begin{align}
\lim_{\beta'\to\beta}\, \plim_{n\to\infty} \, W_{2,n}(\mu_{\bG,\beta}^{\salg},\mu_{\bG,\beta'}^{\salg})=0\, .
\end{align}
\end{definition}

We begin by establishing the stability of the proposed sampling algorithm. 
\begin{theorem}[Stability of the sampling Algorithm~\ref{alg:Sampling}]
\label{thm:stable}
For any $\beta\in (0,\infty)$ and fixed parameters 
$(\eta,$ $K_{\sAMP},$ $K_{\sNGD},$ $\Treg,$ $\delta)$, Algorithm~\ref{alg:Sampling} is
stable with respect to disorder and with respect to temperature. 
\end{theorem}
As a consequence, the Gibbs measures $\mu_{\bG,\beta}$ enjoy similar stability 
properties for $\beta<\beta_1$, which amount (as discussed below) to the absence of chaos in 
both temperature and disorder:
\begin{corollary}
\label{cor:stable}
For any $\beta<\obeta(\xi)$, the following properties hold for the Gibbs measure 
$\mu_{\bG,\beta}$, cf. Eq.~\eqref{eq:sk}:
\begin{enumerate}
    \item \label{it:disorder-stability} $\lim_{s\to 0}\plim_{n\to\infty} W_{2,n}(\mu_{\bG,\beta},\mu_{\bG_s,\beta})=0$.
    \item $\lim_{\beta'\to\beta}\plim_{n\to\infty} W_{2,n}(\mu_{\bG,\beta},\mu_{\bG,\beta'})=0$.
\end{enumerate}
\end{corollary}

\begin{proof}
Take $\eps>0$ and choose  parameters 
$(\eta,K_{\sAMP},K_{\sNGD},\Treg,\delta)$ of Algorithm~\ref{alg:Sampling} with the desired tolerance $\eps$ so that Theorem~\ref{thm:main} holds. 
Combining with Theorem~\ref{thm:stable} using the same parameters $(\eta,K_{\sAMP},K_{\sNGD},\Treg,\delta)$
implies the result since $\eps$ can be arbitrarily small. 
(Recall that $(\eta,K_{\sAMP},K_{\sNGD},\Treg,\delta)$ can be chosen independent of $\beta$
for $\beta< \beta_1$.)
\end{proof}

\begin{remark}
We emphasize that Corollary \ref{cor:stable} makes no reference to the sampling algorithm, and is instead a purely structural property of
the Gibbs measure. The sampling algorithm, however, is the key tool of our proof.
\end{remark}

Stability is related to the opposite notion of chaos, which is a well studied and 
important property of spin glasses, see e.g.
\cite{chatterjee2009disorder,chen2013disorder,chatterjee2014superconcentration,chen2015disorder,
chen2018disorder}. In particular, ``disorder chaos" refers to 
the following phenomenon. Draw $\bx^0\sim \mu_{\bG,\beta}$ independently of  
$\bx^s\sim \mu_{\bG_s,\beta}$, and denote by $\mu^{(0,s)}_{\bG,\beta}:=\mu_{\bG,\beta}\otimes
\mu_{\bG^s,\beta}$ their joint distribution. Disorder chaos holds at inverse temperature
$\beta$ if 
\begin{align}
\lim_{s\to 0}\lim_{n\to\infty}\E\mu^{(0,s)}_{\bG,\beta} \Big\{\Big(\frac{1}{n}\<\bx^0,\bx^s\>\Big)^2\Big\}= 0
\, .\label{eq:FirstDisorderChaos}
\end{align}
Note that disorder chaos is not necessarily a surprising property.
For instance when $\beta=0$, the distribution $\mu_{\bG_s,\beta}$ is simply the uniform measure over the hypercube $\{-1,+1\}^n$ for all $s$, and this example exhibits disorder chaos
in the sense of Eq.~\eqref{eq:FirstDisorderChaos}. 
In fact, the Gibbs measure exhibits disorder chaos at all $\beta\in [0,\infty)$ 
\cite{chatterjee2009disorder}. However, for $\beta>\beta_c$, 
Eq.~\eqref{eq:FirstDisorderChaos} leads to a stronger conclusion where disorder chaos holds
 in the optimal transport sense.
In the results below, we will assume $\xi$ is even, i.e. $c_p=0$ for all odd $p$. 
\begin{theorem}[Transport disorder chaos in $W_{2,n}$ distance] 
\label{thm:disorder_chaos_sk}
For even $\xi$ and all $\beta>\beta_c$, 
\begin{equation}
\label{eq:transport_disorder_chaos}
   \inf_{s\in (0,1)} \, \liminf_{n\to\infty} \, \E\big[W_{2,n}(\mu_{\bG,\beta},\mu_{\bG_s,\beta})\big]>0 \, .
\end{equation}
\end{theorem}

While \eqref{eq:transport_disorder_chaos} is implied by the onset of replica-symmetry breaking, it also holds in a larger range of temperatures. Following \cite{alaoui2023shattering}, we say $\mu_{\bG,\beta}$ is \emph{shattered} if it admits a decomposition as in Definition~\ref{def:Shattering} below, for some parameters $(c,r,s)$ which are independent of $n$.
Informally, shattering means $\{-1,+1\}^n$ admits a $\bG$-dependent partition into 
exponentially many clusters which are separated by linear Hamming distance, each 
cluster having exponentially small weight for $\mu_{\bG,\beta}$, yet the clusters 
collectively cover nearly all the mass of $\mu_{\bG,\beta}$.

\begin{definition}\label{def:Shattering}
We say the sets (clusters) $\{\cC_1,\dots,\cC_M\}$ with $\cC_m\subseteq\{-1,+1\}^n$ are a 
\emph{symmetric shattering decomposition} of the probability measure $\mu$, 
with parameters $(c,r,s)$, if the following hold:
 \begin{enumerate}[label={\sf S\arabic*},ref={\sf S\arabic*}]
    \item 
        \label{it:clusters-cover}
        The clustering is origin symmetric: for each $\cC_i$, there exists $j\neq i$ such that $\cC_j=-\cC_i$.
        \item
        \label{it:small-prob}
        Each cluster has exponentially small probability:
        \[
        \max_{1\leq m\leq M}\,\, \mu(\cC_m)\leq e^{-cn}\, .
        \]
        \item 
        \label{it:small-diam}
        Each cluster is geometrically small:
        \[
        \diam(\cC_m)\leq r \sqrt{n} \quad\forall~m\in [M]\, .
        \]
        \item 
        \label{it:clusters-separate}
        Distinct clusters are well-separated: 
        \[
        \min_{1\leq m_1<m_2\leq M} d\big(\cC_{m_1},\cC_{m_2}\big)\geq s \sqrt{n} \, .
        \]
        \item 
        \label{it:clusters-cover}
        Collectively, the clusters carry nearly all of the Gibbs mass:
        \[
        \mu\lt(
        \bigcup\limits_{m=1}^M \cC_m
        \rt)
        \geq 1-e^{-cn}\, .
        \] 
        \item 
        \label{it:3r-sep}
          Most cluster pairs are $3r$-separated:
          defining for $\Delta>0$ the event 
          \[
          \sep_{m_1,m_2,\Delta}
          =
          \lt\{
          d\big(\cC_{m_1},\cC_{m_2}\big)\geq \Delta \sqrt{n}
          \rt\},
          \]
          we have (for $r$ as in \ref{it:small-diam}):
          \[
          \sum_{1\leq m_1<m_2\leq M}
          \mu(\cC_{m_1})
          \mu(\cC_{m_2})
          (1-1_{\sep_{m_1,m_2,3r}})
          \leq 0.01 \, .
          \]
        \end{enumerate} 
\end{definition}

\begin{remark}
Alternative definitions of shattering are possible in which 
some of the above  conditions are omitted,
in particular condition {\sf S3} and {\sf S6}. The specific definition given here  was introduced
in \cite{alaoui2023shattering}, and is motivated by two considerations:
$(i)$~The additional conditions allow to establish disorder chaos; 
$(ii)$~They hold for $p$-spin model in a large sub-interval of 
$(\beta_{\sdyn},\beta_c)$ (and potentially in the whole interval).
\end{remark}

\begin{proposition}[{\cite[Theorem 5.1]{alaoui2023shattering}}]
Let $(c,r,s)$ be positive constants. Suppose $\xi$ is even and $(\xi,\beta)$ is such that $\mu_{\bG,\beta}$ is $(c,r,s)$-shattered with uniformly positive probability. 
  Then Eq.~\eqref{eq:transport_disorder_chaos} holds.
\end{proposition}
(Technically \cite[Theorem 5.1]{alaoui2023shattering} is stated for pure models only, but the above follows by the exact same proof.)

Finally, we obtain the desired hardness results by reversing the implication 
in Corollary \ref{cor:stable}.  In the presence of replica-symmetry breaking or shattering, 
no stable algorithm  can approximately sample from the measure 
$\mu_{\bG,\beta}$ in $W_{2,n}$ sense.
\begin{theorem}[Hardness for stable sampling]  
\label{thm:disorder-stable-LB}
Let $\{\ALG_n\}_{n\geq 1}$ be a family of randomized algorithms  
which is stable with respect to disorder as per Definition \ref{def:Stable}
at inverse temperature $\beta$. 
Assume $\xi$ is even and let $\mu^{\salg}_{\bG,\beta}$ be the law of the output 
$\ALG_n(\bG,\beta,\omega)$ conditional on $\bG$. Then
\[
   \liminf_{n\to\infty} \E \big[W_{2,n}(\mu^{\salg}_{\bG,\beta},~\mu_{\bG,\beta}) \big] >0 \, ,
\]
under either one of the following conditions:
\begin{enumerate}[label={\sf C\arabic*},ref={\sf C\arabic*}]
  \item 
  \label{it:shattering-implies-hardness-Ising}
  $\beta>\beta_c$.
  \item 
  \label{it:shattering-implies-hardness-sphere}
  $\mu_{\bG,\beta}$ is shattered with probability at least $1/2$.
\end{enumerate}
\end{theorem}

In the case of pure models, the latter condition~\ref{it:shattering-implies-hardness-sphere} 
implies  concrete conclusions about the failure of stable algorithms 
 using the results of \cite{alaoui2023shattering,gamarnik2023shattering} (which show 
 that shattering holds in the two cases below). We state such conclusions
 both for the case of Ising spin glasses (treated in the rest of this paper and in
 \cite{gamarnik2023shattering})
 and of spherical spin glasses (treated in \cite{alaoui2023shattering}).
 We recall that spherical spin glasses are defined as in \eqref{eq:sk}, but with reference measure 
 the uniform measure over the sphere of radius 
 $\sqrt{n}$ (which we denote by $\mu^{\spher}_0$), instead of the uniform measure over $\{-1,+1\}^n$. Namely
 \begin{align}
\label{eq:spherical}
  \mu^{\spher}_{\bG,\beta}(\de \bx) = \frac{1}{Z(\beta,\bG)}\, e^{\beta H_n(\bx)} \, \mu^{\spher}_0(\de\bx)\, ,
\end{align}  

\begin{corollary}
\label{cor:concrete-shattering-hardness}
In the setting of Theorem~\ref{thm:disorder-stable-LB}, let $\xi(t)=t^{p}$ and suppose that either:
  \begin{enumerate}[label={(\arabic*)},ref={(\arabic*)}]
  \item 
  \label{it:shattering-Ising}
  For the Ising spin glass model of  
  Eq.~\eqref{eq:sk}: $\beta\in (\sqrt{\log 2},\sqrt{2\log 2})$, and $p\geq P(\beta)$. 
  \item 
  \label{it:shattering-spherical}
  For the spherical spin glass model of Eq.~\eqref{eq:spherical}:
  $p,\beta\geq C$ for $C$ an absolute constant and 
  $\beta\neq\beta_c(\xi)$. 
\end{enumerate}
  Then 
  \[
   \liminf_{n\to\infty} \E \big[W_{2,n}(\mu^{\salg}_{\bG,\beta},~\mu_{\bG,\beta}) \big] >0 \, .
  \]
\end{corollary}
The results of this section are proved in Section~\ref{sec:disorder}. 
We expect that shattering and hence impossibility of stable sampling holds for all 
$\beta>\beta_{\sdyn}$ as defined in the introduction, but proving this
remains an open problem. (See also Appendix \label{app:prelim} for the explicit
expression of $\beta_{\sdyn}(\xi)$ as predicted in the physics literature).

\begin{remark}
\label{rem:no-degeneration-near-criticality}
  Although Corollary~\ref{cor:concrete-shattering-hardness} part~\ref{it:shattering-spherical} does not directly address the critical temperature $\beta=\beta_c$, the proof shows that $\liminf_{\beta\uparrow\beta_c}\liminf_{n\to\infty} \E \big[W_{2,n}(\mu^{\salg}_{\bG,\beta},~\mu_{\bG,\beta}) \big]>0$.
  Equivalently in Corollary~\ref{cor:concrete-shattering-hardness} part~\ref{it:shattering-spherical}, the values $(\beta_n)_{n\geq 1}$ may converge up to $\beta_c$, as long as the convergence is sufficiently slow.
\end{remark}

\section{Properties of stochastic localization}
\label{sec:stochloc}

We collect in this section the main properties of the stochastic localization 
process needed for our analysis. To be definite, we will focus
on the stochastic localization process for the Gibbs measure \eqref{eq:sk},
although most of what we will say generalizes to other probability measures in $\R^n$,
 under suitable tail conditions. Throughout this section, the disorder $\bG$ is viewed as fixed.
 
Recalling the tilted measure $\mu_{\bG,\by}$ of Eq.~\eqref{eq:sktilted},
and the SDE of Eq.~\eqref{eq:GeneralSDE}, we introduce the shorthand
\[
 \mu_t = \mu_{\bG,\by(t)}\, .
 \]

 The following properties are well known. See for instance~\cite[Propositions 9, 10]{eldan2022log} 
 or~\cite{eldan2020taming}. We provide proofs for the reader's convenience.
 \begin{lemma}\label{prop:stochloc1}
For all $t \ge 0$ and all $\bx \in \{-1,+1\}^n$, 
\begin{equation}\label{eq:L}
 \rmd \mu_t(\bx) = \mu_t(\bx) \langle \bx - \m_{\bG,\by(t)} , \rmd \bB(t)\rangle  \, .
 \end{equation}
As a consequence, for any function $\varphi : \R^n \to \R^m$, the process 
$\big(\E_{\bx \sim \mu_t}\big[\varphi(\bx)\big]\big)_{t \ge 0}$ is a martingale.  
 \end{lemma}
 \begin{proof}
 Let us evaluate the differential of $\log \mu_t$. By writing $Z_t$
 for the normalization constant $Z(\by(t))$ of Eq.~\eqref{eq:sktilted}, we get
\begin{equation}\label{eq:logL}
\rmd \log \mu_t (\bx) =  \langle \rmd \by(t) , \bx\rangle - \rmd \log Z_t\, .
\end{equation}
 Using It\^{o}'s formula for $Z_t$ we have
\begin{align*}
    \rmd Z_t 
    &= 
    \rmd \sum_{\bx \in \{-1,+1\}^n} e^{\beta H_n(\bx) + \langle \by(t) ,\bx\rangle}
    \\
    &=  
    \sum_{\bx \in \{-1,+1\}^n}  \big(\langle \rmd  \by(t) , \bx\rangle + \frac{1}{2}\|\bx\|_2^2 \rmd t\big) 
    e^{\beta H_n(\bx) + \langle  \by(t) , \bx\rangle} \, .
\end{align*}
Therefore, denoting by $[Z]_t$ the quadratic variation process associated to $Z_t$,
\begin{align*}
\rmd \log Z_t &= \frac{\rmd Z_t}{Z_t} - \frac{1}{2} \frac{\rmd [Z]_t}{Z_t^2} \\
&= \langle \rmd  \by(t) , \m_{\bG,\by(t)}\rangle  +  \frac{1}{2} \E_{\mu_t}[\|\bx\|^2] \rmd t -  \frac{1}{2} \|\m_{\bG,\by(t)}\|^ 2 \rmd t \, \\
&= \langle \rmd  \by(t) , \m_{\bG,\by(t)}\rangle  +  \frac{n}{2}  -  \frac{1}{2} \|\m_{\bG,\by(t)}\|^ 2 \rmd t \, .
\end{align*}
Substituting in~\eqref{eq:logL} we obtain
\begin{align*}
\rmd \log \mu_t (\bx) &= \langle \rmd  \by(t) , \bx - \m_{\bG,\by(t)}\rangle -  \frac{n}{2} \rmd t +  \frac{1}{2} \|\m_{\bG,\by(t)}\|^ 2 \rmd t \\
&= \langle \rmd \bB_t , \bx - \m_{\bG,\by(t)}\rangle - \frac{1}{2} \|\bx - \m_{\bG,\by(t)}\|^ 2 \rmd t \, .
\end{align*}
Applying It\^{o}'s formula to $e^{\log \mu_t(\bx)}$ yields the desired result. 

Finally, Eq.~\eqref{eq:L} implies that $\mu_t(\bx)$ is a martingale for every
$\bx\in\{-1,+1\}^n$. Since $\E_{\bx \sim \mu_t}\big[\varphi(\bx)\big]$ is
a linear combination of martingales, it is itself a martingale.
\end{proof}

 \begin{lemma}[\cite{eldan2020taming}]\label{lem:covbound}
For all $t >0$,
\begin{equation}\label{eq:covbound}
\E \cov(\mu_t) \preceq \frac{1}{ t } \bI_n \, .
 \end{equation}
 \end{lemma}

\begin{lemma}\label{lem:W2bound}
For all $t>0$, 
\begin{equation}\label{eq:W2bound}
W_{2,n}\big(\mu_{\bG}, \cL(\m_{\bG,\by(t)})\big)^2 \le \frac{1}{t}\, .
 \end{equation}
In particular, the mean vector $\m_{\bG,\by(t)}$ converges in distribution to a random vector $\bx^\star \sim \mu_{\bG}$ as $t\to \infty$. 
\end{lemma}
\begin{proof}
 By Lemma~\ref{lem:covbound},
 \[\E \big[\E_{\bx \sim \mu_t} [\|\bx - \m_{\bG,\by(t)}\|^2]  \big]\le \frac{n}{t} \, ,\]  
 therefore 
\[\E \Big[W_{2,n}\big(\mu_t, \delta_{\m_{\bG,\by(t)}}\big)^2\Big] \le \frac{1}{t} \, .\] 
Notice that $(\mu,\nu)\mapsto W_{2,n}^2(\mu,\nu)$ is jointly convex.
Since $\mu_{\bG} = \E [\mu_t] $, this implies
\[
    W_{2,n}\big(\mu_{\bG}, \cL(\m_{\bG,\by(t)})\big)^2 \le \E \Big[W_{2,n}\big(\mu_t, \delta_{\m_{\bG,\by(t)}}\big)^2\Big] \le \frac{1}{t}\, .
\]
\end{proof}

%

\section{Analysis of the sampling algorithm and proof of Theorem \ref{thm:main}}
\label{sec:analysis}

 This section is devoted to the analysis of Algorithm~\ref{alg:Sampling} described in the previous section.
 An important simplification is obtained by 
 working under a corresponding \emph{planted} model. This approach has two advantages:
 $(i)$~The joint distribution of the disorder $\bG$ and the process $(\by(t))_{t\ge 0}$ 
 in~\eqref{eq:GeneralSDE} is significantly simpler in the planted model;
 $(ii)$~Analysis in the planted model can be cast as a statistical estimation problem.
 In the latter, Bayes-optimality considerations can be exploited to relate the output of 
 the AMP algorithm $\AMP(\bG,\by;k)$ to the true mean vector 
 $\m(\bG,\by)$.       

 This section is organized as follows. Section \ref{sec:planted} introduces the 
 planted model and its relation to the original model. We then analyze the AMP
 component of our algorithm in Section \ref{sec:AMP}, and the NGD component in
 Section \ref{sec:ngd}. Finally, Section \ref{sec:proofMain} puts the various elements
 together and proves Theorem \ref{thm:main}.

\subsection{The planted model and contiguity}
\label{sec:planted}

In what follows, will use the notation $\mu_{\rd}(\rmd \bx, \rmd \bG)$ for the joint distribution
of $\bx$, $\bG$ introduced above.  Namely $\bG=(\Gp{p})_{2\leq p\leq P}$
has i.i.d.  $\normal(0,1)$ entries, and $\mu_{\rd}(\de\bx|\bG)=\mu_{\bG}(\de\bx)$
is the Gibbs measure \eqref{eq:sk}.

We next introduce a second `planted' measure, as follows.
Let $\onu$ be the uniform distribution over $\{-1,+1\}^n$ and consider pairs $(\bx,\bG) \in \{-1,+1\}^n \times \sH_n$ with law
\begin{equation}
    \mu_{\pl}(\rmd \bx, \rmd \bG) 
    = 
    \frac{1}{Z_{\pl}}\, 
    \exp\Big(-\frac{1}{2}\sum_{p=2}^P
    \lt\|\bG^{(p)} - \frac{\beta c_p \bx^{\otimes p}}{n^{(p-1)/2}} \rt\|_{F}^2 \, 
    \Big)  \, \onu(\rmd \bx) \, \rmd \bG \, ,
\end{equation}
where $\rmd \bG$ is Lebesgue measure on $\sH_n$, and the normalizing constant
\begin{align}\label{eq:Zpl}
    Z_{\pl} := \int\! \exp\Big(- \frac{1}{2}\sum_{p=2}^P
    \lt\|\bG^{(p)} - \frac{\beta c_p\bx^{\otimes p}}{n^{(p-1)/2}} \rt\|_{F}^2 \, 
    \Big)\, \rmd \bG 
\end{align}
is independent of $\bx \in \{-1,+1\}^n$. 
It is easy to see  that the marginal distribution of $\bx$ under $\mu_{\pl}$ is $\onu$. Meanwhile, under the conditional law $\mu_{\pl}( \, \cdot \, | \bx)$ the tensors $\bG^{(2)},\dots,\bG^{(P)}$ are independent with rank-one spikes $\frac{\beta c_p}{ n^{\frac{p-1}{2}}} \bx^{\otimes p}$. Namely, under $\mu_{\pl}( \, \cdot \, | \bx)$, we have
 \begin{align}\label{eq:G_planted}
     \bG^{(p)} = \frac{\beta c_p}{n^{(p-1)/2}}  \bx^{\otimes p}
     +
     \bW^{(p)}\, ,~~~~~~ \bW^{(p)}_{i_1,\cdots, i_p} \stackrel{i.i.d.}{\sim} \normal(0,1)\, ,~~ 1 \le i_1,\cdots,i_p \le n\, ,~~ 1 \le p \le P \, .
 \end{align}
On the other hand, the conditional law $\mu_{\pl}( \, \cdot \, | \bG)$ of $\bx$ given 
$\bG$ is the Gibbs measure $\mu_{\bG}$ (since $\|\bx\|_2^2$ is constant for $\bx\in \{-1,+1\}^n$).  
For this reason, we will omit the subscript ${\rm pl}$ or  ${\rm rd}$ from this conditional distribution. 

Note that the marginal of $\bG$ under $\mu_{\pl}$ takes the form
\begin{equation}\label{eq:lr}
    \mu_{\pl}(\rmd \bG) 
    =
    \mu_{\rd}(\rmd \bG)\, Z_{\xi}(\bG)
    \, ,
\end{equation}
with $\mu_{\rd}$  the standard Gaussian measure on $\sH_n$, and $Z_{\xi}(\bG)$ is the (rescaled)
 partition function for the corresponding mixed $p$-spin model,
\begin{equation}
    Z_{\xi}(\bG) = \frac{1}{2^n} \sum_{\bx\in \{-1,+1\}^n} e^{\beta H_n(\bx) - n\beta^2\xi(1)/2} \, .
\end{equation}
(Of course, since $\mu(\de\bx|\bG)$ is the same under the two distributions, we also have
$\mu_{\pl}(\de\bx,\rmd \bG)  = Z_{\xi}(\bG) \,\mu_{\rd}(\de\bx,\rmd \bG)$.)

In order to show that $\mu_{\pl}$ and $\mu_{\rd}$ are close (in a sense to be made precise shortly), 
we need to control the fluctuations of their ratio, which is given by the rescaled partition 
function $Z_{\xi}(\bG)$. This will impose a constraint on how large $\beta$ can be. 
\begin{theorem}\label{thm:fluctuations_Z}
Let $\bG \sim \mu_{\rd}$, $h(x):=-((1+x)/2)\log((1+x)/2) - ((1+x)/2)\log((1+x)/2)$
and define $\beta_2$ as per Eq.~\eqref{eq:Beta2Def}.
If $\beta<\beta_2$ (which, in particular, implies $\beta^2\xi''(0)<1$), 
then
\begin{equation}
Z_{\xi}(\bG) \xrightarrow[n \to \infty]{\rmd} \exp (W) \, , \label{eq:ConvergenceZ}
\end{equation}
where $\sigma^2 = \frac{1}{4} (-\log(1-\beta^2\xi''(0)) - \beta^2\xi''(0))$ and 
$W\sim \normal \big(- \sigma^2, 2\sigma^2\big)$. 

Consequently,  $ \mu_{\pl}(\de\bG)$ and $\mu_{\rd}(\de\bG)$ are mutually contiguous for all
 $\beta<\beta_2$. Namely, for any sequence of events $E_n$ we have
  $\mu_{\rd}(E_n) \to 0$ if and only if $\mu_{\pl}(E_n) \to 0$.   
\end{theorem}
\begin{proof}

Let $\oH_n(\bx)=H_n(\bx)-n\xi(1)/2$,
$\oH^{(2)}_n(\bx) := c_2\<\bx,\bG^{(2)}\bx\>/\sqrt{n}-n\xi''(0)/4$ and 
$\oH^{(>2)}_n(\bx)= \oH_n(\bx)-\oH^{(2)}_n(\bx)$. Note that
$\oH^{(2)}_n(\, \cdot \, )$ and $\oH^{(>2)}_n(\,\cdot\,)$ are two independent Gaussian processes with
\begin{align}
\E \oH^{(2)}_n(\bx) = -n\xi''(0)/4\, ,&~~~~~ \E \oH^{(>2)}_n(\bx) = -n\xi_{>2}(1)/2\, ,\\
\Cov(\oH^{(2)}_n(\bx_1);\oH^{(2)}_n(\bx_2)) = n\xi''(0)Q_{12}^2/2\, ,&~~~~~
\Cov(\oH^{(>2)}_n(\bx_1);\oH^{(>2)}_n(\bx_2)) = n\xi_{>2}(Q_{12})\, ,
\end{align}
where $\xi_{>2}(q) :=\xi(q)-\xi''(0)q^2/2$, and we introduced the shorthand 
$Q_{12} := \<\bx_1,\bx_2\>/n$.
 Define
\begin{equation}
    Z^{(2)}_{\xi}(\bG) = \frac{1}{2^n} \sum_{\bx\in \{-1,+1\}^n} e^{\beta\oH_n^{(2)}(\bx) } \, .
\end{equation}
Then we have, for $\Lambda_n : = (2\Z-n)/n$
\begin{align*}
\E\Big[ \big(Z_{\xi}(\bG)-Z^{(2)}_{\xi}(\bG)\big)^2\Big] &= \frac{1}{4^n}
 \sum_{\bx_1,\bx_2 \in \{-1,+1\}^n}
 \E \Big[e^{\beta \oH^{(2)}_n(\bx_1)+ \beta\oH^{(2)}_n(\bx_2)}\Big]
  \E\Big[\Big(e^{\beta \oH^{(>2)}_n(\bx_1)}-1\Big)\Big(e^{ \beta\oH^{(>2)}_n(\bx_2)}-1\Big)\Big]\\
&=  \frac{1}{4^n}
 \sum_{\bx_1,\bx_2 \in \{-1,+1\}^n}
 e^{n\beta^2\xi''(0)Q_{12}^2/2}\Big(e^{n\beta^2\xi_{>2}(Q_{12})}-1\Big)\\
 &=  \frac{1}{2^n} \sum_{q \in \Lambda_n;\; |q|\le 1}
 \binom{n}{n(1+q)/2}e^{n\beta^2\xi''(0)q^2/2}\Big(e^{n\beta^2\xi_{>2}(q)}-1\Big)\,.
\end{align*}
Letting $F(q):=\xi(q)+h(q)-\log(2)$, $F^{(2)}(q):=\xi''(0)q^{2}/2+h(q)-\log(2)$, we have 
by standard bounds on binomial coefficients,
setting $q_0(n) := (C\log n)^{1/2}/ n^{1/2}$, for $C=C(\beta)$ a sufficiently large constant,
\begin{align*}
\E\Big[ \big(Z_{\xi}(\bG)-Z^{(2)}_{\xi}(\bG)\big)^2\Big]&\le 
\sum_{q \in \Lambda_n;\; q_0(n)\le |q|\le 1} e^{nF^{(2)}(q)}\Big(e^{n\beta^2\xi_{>2}(q)}-1\Big)
+\frac{C_0}{\sqrt{n}}\sum_{q \in \Lambda_n;\;  |q|\le q_0(n)}
e^{nF^{(2)}(q)}\Big(e^{n\beta^2\xi_{>2}(q)}-1\Big)\\
& \le n\, e^{nF(q_0(n))} + \frac{C_0}{\sqrt{n}}
\max_{|t|\le 1/2}\xi'''(t)\sum_{q \in \Lambda_n;\;  |q|\le q_0(n)}
e^{nF^{(2)}(q)}n  |q|^3\\
& \le \frac{1}{n^4} + \frac{C_1(\log n)^{3/2}}{n}
\sum_{q \in \Lambda_n;\;  |q|\le q_0(n)}
e^{nF^{(2)}(q)}\, .
\end{align*}
Now by the assumption $\beta<\beta_2$, we have $F''(0)<0$ and $F(q)<0$ strictly for all 
$0<|q|\le 1$. Hence, there exists $c>0$ (independent of $n$ such that $F(q)\le -c q^2$ 
for all $|q|\le 1$. Therefore,
\begin{align*}
\E\Big[ \big(Z_{\xi}(\bG)-Z^{(2)}_{\xi}(\bG)\big)^2\Big]&\le  
\frac{1}{n^4} + \frac{C_1(\log n)^{3/2}}{n}
\sum_{q \in \Lambda_n}
e^{-c q^2}\le \frac{C_1(\log n)^{3/2}}{n} \, .
\end{align*}
In particular, $\plim_{n\to\infty}(Z_{\xi}(\bG)-Z^{(2)}_{\xi}(\bG))=0$.
By \cite[Proposition 2.2]{aizenman1987some} $Z^{(2)}_{\xi}(\bG)  \xrightarrow{\rmd} e^W$,
whence the claim \eqref{eq:ConvergenceZ} follows.
Contiguity is an immediate consequence of the above and Le Cam's first lemma. 
\end{proof}
 
For the purpose of our analysis we will need a  result about the joint distributions of 
$(\bx, \bG, \by)$ under our ``random'' model and a planted model which we now introduce.  
 
Recall that $\m(\bG,\by)$ denotes the mean of the Gibbs measure $\mu_{\bG,\by}$ in
Eq.~\eqref{eq:sktilted}.
For a fixed $T \ge 0$, we define two Borel distributions $\P$ and $\Q$ on 
$(\bx, \bG,\by)\in\{-1,+1\}^n \times\sH_n\times C([0,T], \R^n)$ as follows:
\begin{align}   
\Q ~~&:~~
\begin{dcases}
\bG &\sim~~ \mu_{\rd} \, , \\
\bx & \sim~~ \mu(\,\cdot\,| \bG)\, ,\\
\by(t) &=~~ t \bx + \bB(t) \, , ~~~ t \in [0,T] \, ,
\end{dcases}
\hspace{.5cm} &\text{(random)}\label{eq:Q}
\\
~
\P ~~&:~~
\begin{dcases}
\bx &\sim~~ \onu \, , \\
\bG &\sim~~ \mu_{\pl}( \,\cdot\, |\, \bx) \, ,\\
\by(t) &=~~ t \bx + \bB(t) \, , ~~~ t \in [0,T] \, \,
\end{dcases}
\hspace{.5cm} &\text{(planted)}\label{eq:P}
\end{align}
where $(\bB(t))_{t \ge 0}$ is a standard Brownian motion in $\R^n$ independent of everything else.
Note that the marginal distributions of $(\bx,\bG)$ under
$\Q$ and $\P$ coincide, respectively, with $\mu_{\rd}$ and $\mu_{\pl}$.

Note the SDE defining the process $\by = (\by(t))_{t\in [0,T]}$ in 
Eq.~\eqref{eq:Q} is a restatement of the stochastic localization equation~\eqref{eq:GeneralSDE} 
applied to the Gibbs measure $\mu_{\bG}$. 
\begin{proposition}\label{prop:contig}
The probability distributions $\Q$ and $\P$ admit the equivalent description
\begin{align}   
\Q ~~&:~~
\begin{dcases}
\bG &\sim~~ \mu_{\rd} \, , \\
\bx & \sim~~ \mu_(\,\cdot\,| \bG)\, ,\\
\by(t) &=~~ \int_{0}^t \m(\bG,\by(s)) \, \rmd s + \bW(t) \, , ~~~ t \in [0,T] \, ,
\end{dcases}
\hspace{.5cm} &\textup{(random)}\label{eq:Q-bis}
\\
~
\P ~~&:~~
\begin{dcases}
\bx &\sim~~ \onu \, , \\
\bG &\sim~~ \mu_{\pl}( \,\cdot\, |\, \bx) \, ,\\
\by(t) &=~~\int_{0}^t \m(\bG,\by(s)) \, \rmd s + \bW(t) \,  ,
\end{dcases}
\hspace{.5cm} &\textup{(planted)}\label{eq:P-bis}
\end{align}
for $\bW$ a standard Brownian motion.

Further $\P$ is absolutely continuous with respect to $\Q$ and for all $(\bG,\by) \in  \sH_n \times  C([0,T], \R^n)$, 
\begin{align}
    \frac{\rmd \P}{\rmd \Q}(\bx,\bG, \by) = Z_{\xi}(\bG) \, .\label{eq:RN-PQ}
\end{align}
In particular, for all $\beta<\beta_2$, $\P$ and $\Q$ are mutually contiguous.
\end{proposition}
\begin{proof}
Definitions \eqref{eq:Q-bis}, \eqref{eq:P-bis}, differ from
 \eqref{eq:Q}, \eqref{eq:P}, only in the definition of the conditional distribution of 
 $\by$ given $(\bx,\bG)$. The equivalence was established in a general context in 
 \cite{el2022information}.
 
 Finally, Eq.~\eqref{eq:RN-PQ} follows because the conditional distribution of $\by$ given $(\bx,\bG)$
 is the same under the two models and therefore
 \begin{align}
    \frac{\rmd \P}{\rmd \Q}(\bx,\bG, \by) = 
      \frac{\rmd \mu_{\pl}}{\rmd \mu_{\rd}}(\bx,\bG,) = 
    Z_{\xi}(\bG) \, ,
\end{align}
  where the last equality was proven above.
\end{proof}

For the remainder of the proof of Theorem \ref{thm:main}, we work under the planted distribution $\P$. All results proven under $\P$ transfer to $\Q$ 
by Proposition~\ref{prop:contig}.

\subsection{Approximate Message Passing}
\label{sec:AMP}

In this section we analyze the AMP iteration of Algorithm~\ref{alg:Mean},
which we copy here for the reader's convenience
\begin{align}
\nonumber
    \hm^{-1} 
    &= \bz^{-1} =
    \bz^{0}= 0\, ,
    \\
\label{eq:AMPreminder}
    \hm^{k} 
    &= 
    \tanh(\bz^{k} ) \, ,
     ~~~~
    \hat{q}^k = 
    \frac{1}{n}\sum_{i=1}^n  \tanh^2(z^{k}_i) \, ,
    ~~~~
    \sb_{k} = \beta^2 (1-\hat{q}^k)\xi''\big(\hat{q}^k) \,  ~~~~~~~  \forall k\ge 0\, ,
    \\
\nonumber
    \bz^{k+1} 
    &= 
    \beta \nabla H_n
    \big(\hm^{k}\big) 
    + 
    \by - \sb_{k} \hm^{k-1}\, .
\end{align}
When needed, we will specify the dependence on $\bG,\by$ by writing 
$\hm^k = \hm^k(\bG,\by)=\AMP(\bG,\by;k)$
and  $\bz^k = \bz^k(\bG,\by)$. Throughout this section $(\bG,\by)\sim \P$ will be distributed
according to the planted model introduced above.

We will prove two results: (1)  AMP approximately computes the posterior mean $\m(\bG,\by(t))$; this is the content of Proposition~\ref{prop:amp-posterior-mean}, and (2) the posterior mean $\m(\bG,\by(t))$ has a uniform continuity property with respect to the time parameter $t$; this is the content of Lemma~\ref{lem:uniform-path}.  

 %
%
%
%
\subsubsection{State Evolution}
Our analysis will be based on the state evolution results of~\cite{ams20} for mixed tensors (see~\cite{bayati2011dynamics,javanmard2013state} for the matrix case).
They imply the following asymptotic characterization for the iterates. Set $q_0(\beta,t)=0$ and $\Lambda_{0,i}(\beta,t)=0$ for all $i$. Next, recursively define
\begin{align}
   \label{eq:AMP-sig}
        \gamma_k(\beta,t) &= \beta^2 \cdot \xi'\big(q_k(\beta,t)\big)\,,
    ~~~~~ \Sigma_{k,j}(\beta,t) = \beta^2 \cdot \xi'\big(\Lambda_{k,j}(\beta,t)\big)\, ,  
    \\
     \label{eq:AMP-nu}  
    q_{k+1}(\beta,t)
    &=
   \E\big[\tanh\left(\gamma_k(\beta,t)+t+W_k\right)\big] \, ,
    \\
    \label{eq:AMP-lam}
    \Lambda_{k+1,j+1}(\beta,t)
    &= 
    \E\big[\tanh\left(\gamma_k(\beta,t)+t+W_k\right)\cdot\tanh\left(\gamma_j(\beta,t)+t+W_j\right)\big] \, ,
\end{align}
where $\bW=(W_{j})_{0\leq j\le k}\in\bbR^{k+1}$ are jointly Gaussian, with zero mean and covariance 
$\bSigma_{\le k}+t\bfone\bfone^\top$, $\bSigma_{\le k}:=(\Sigma_{i,j})_{i,j\le k}$.

\begin{proposition}
\label{prop:state_evolution}
For $(\bx,\bG,\by)\sim \mathbb \P$ and any $k\in\mathbb Z_{\ge 0}$, the empirical distribution of the 
coordinate of the AMP iterates converges almost surely in $W_2(\mathbb R^{k+2})$ as follows:
\begin{align}
    &\frac{1}{n}\sum_{i=1}^n \delta_{(z^1_i,\cdots,z^k_i,x_i,y_i)}
    \xrightarrow[n \to \infty]{W_2}
    \cL\Big(\bgamma_{\le k}(\beta,t) X+\bW+Y\bfone,X,Y\Big) \, ,\\
    &\bgamma_{\le k}(\beta,t) = \big(\gamma_1(\beta,t),\dots,\gamma_k(\beta,t)\big)\, ,
    \;\;\;\;\; \bW \sim \normal(0,\bSigma_{\le k})\, .
\end{align}
On the right-hand side, $X$ is uniformly random in $\{-1,+1\}$,  $Y=tX+\sqrt{t}Z$ where
 $Z\sim\normal(0,1)$ and $X,\bW,Z$ are mutually independent.
\end{proposition}
\begin{proof}
We will reduce the proof to the state evolution result in~\cite[Proposition 3.1]{ams20}, the difference being that in the present case, there is a planted vector $\bx_0$ which needs to be accounted for. Proposition 3.1 in~\cite{ams20} only considers the purely Gaussian case where $\bx_0=0$. 
For ease of notation let us denote $H_{\bG}$ the Hamiltonian~\eqref{eq:def-hamiltonian} with coefficients given by the sequence of tensors $\bG \in \sH_n$. Recall that $\bG$ has a spike component $\bx_0$ in the planted model, see Eq.~\eqref{eq:G_planted}. We denote $H_{\bW}$ the same Hamiltonian where $\bW$ is Gaussian, as in Eq.~\eqref{eq:G_planted}.  
Now let us consider the surrogate AMP iteration
\begin{align}
\nonumber
    \bar{\m}^{-1} 
    &= \bar{\bz}^{-1} =
    \bar{\bz}^{0}= 0\, ,
    \\
\label{eq:AMPsurrogate}
   \bar{\m}^{k} 
    &= 
    \tanh\big(\bar{\bz}^{k} \big) \, ,
     ~~~~
    ~~~~
    \bar{\sb}_{k} = \beta^2 (1-q_k)\xi''\big(q_k) \,  ~~~~~~~  \forall k\ge 0\, ,
    \\
\nonumber
    \bar{\bz}^{k+1} 
    &= \beta\nabla H_{\bW}
    \big(\bar{\m}^{k}\big)  + \beta^2 \xi'(q_k) \bx_0
    +  
    \by - \bar{\sb}_{k} \bar{\m}^{k-1}\, .
\end{align}
In the above, $q_k$ is defined in Eq.~\eqref{eq:AMP-nu}. (Note that this is not an iteration that can be executed in practice.) 
Now, we apply the state evolution result of~\cite[Proposition 3.1]{ams20} to the above iteration for each realization of $\bx_0$ and $\by$. We see that this iteration satisfies the  conclusions of Proposition~\ref{prop:state_evolution}. It remains to show that the trajectories of Eq.~\eqref{eq:AMPreminder} and Eq.~\eqref{eq:AMPsurrogate} are close in $\ell_2$ distance.   Let 
\begin{align}
\eps_{n,k} &= \Big|\xi'\big(\langle \bx_0,\bar{\m}^{k}\rangle/n\big) - \xi'(q_k)\Big|\, ,~~~ \eta_{n,k} = \frac{1}{\sqrt{n}}\big\|\bz^{k} - \bar{\bz}^k\big\|_2\, ,~~~\mbox{and}~~~\rho_{n,k} = \frac{1}{n}\big\|\bar{\m}^k\big\|_2^2 - q_k \, .
\end{align}
Since state evolution holds for the surrogate iteration  Eq.~\eqref{eq:AMPsurrogate}, we have $\plim_{n \to \infty} \eps_{n,k}=\plim_{n \to \infty} \rho_{n,k}=0$ for all $k \ge 1$. We now argue that $\plim_{n \to \infty} \eta_{n,k} =0$ by induction over $k\ge 0$. The base cases  $k=-1$ and $k=0$ are clear, since $\eta_{n,-1} = \eta_{n,0}=0$. Next assume that $\plim_{n \to \infty} \eta_{n,k-1} = \plim_{n \to \infty} \eta_{n,k} =0$.
We derive abound on $\eta_{n,k+1}$. We have
\begin{align}
\eta_{n,k+1} = \frac{1}{\sqrt{n}}\big\|\bz^{k+1}-\bar{\bz}^{k+1}\big\|_2 &\le 
\frac{\beta}{\sqrt{n}}\big\|\nabla H_{\bG}(\hm^{k}) - \nabla H_{\bW}(\bar{\m}^k)- \beta^2 \xi'(q_k) \bx_0\big\|_2 \\
&~~~+\frac{\beta^2\|\xi''\|_{\infty}}{\sqrt{n}}\big\|\hm^{k-1} - \bar{\m}^{k-1}\big\|_2
+ |\bar{\sb}_{k}-\sb_{k}| \, .
\end{align}
Expanding $\bG$ as per Eq.~\eqref{eq:G_planted},
\[\nabla H_{\bG}(\hm^{k}) = \nabla H_{\bW}(\hm^{k}) + \xi'\big(\langle \bx_0,\hm^{k}\rangle/n\big) \bx_0\, ,\]
therefore 
\begin{align}
\frac{1}{\sqrt{n}}\big\|\nabla H_{\bG}(\hm^{k}) - \nabla H_{\bW}(\bar{\m}^k)- \beta^2 \xi'(q_k) \bx_0\big\|_2  &\le 
\sup_{\m \in [-1,1]^n} \big\|\nabla^2 H_{\bW}(\m)\big\|_{\op} \cdot \frac{1}{\sqrt{n}}\big\|\hm^{k} -  \bar{\m}^k\big\|_2\\
&~~~+\beta^2\Big|\xi'\big(\langle \bx_0,\hm^{k}\rangle/n\big) - \xi'(q_k) \Big|\\
&\le K \eta_{n,k} + \beta^2 \eps_{n,k}+ \|\xi''\|_{\infty} \eta_{n,k}  \, ,
\end{align}
where the last inequality holds for some $K= K(\xi)>0$ with probability as least $1-e^{-cn}$, $c>0$; see e.g.,~\cite[Proposition 2.3]{huang2021tight}, and by the fact that the map $x\mapsto \tanh(x)$ is 1-Lipschitz. 
Next, 
\begin{align}
|\bar{\sb}_{k}-\sb_{k}| &\le \beta^2(\|\xi'''\|_{\infty}+\|\xi''\|_{\infty})|\hat{q}^k-q_k| \\
&\le C\frac{1}{n}\sum_{i=1}^n |\tanh^2(z_i) - \tanh^2(\bar{z}_i)| + C\rho_{n,k}\\
&\le C \frac{1}{\sqrt{n}}\big\|\bz^{k}-\bar{\bz}^{k}\big\|_2 +  C\rho_{n,k}\, .
\end{align}
It follows that for some constant $C=C(\beta,\xi)>0$ we have with high probability as $n \to \infty$, 
\begin{align}
\eta_{n,k+1} \le C\big(\eta_{n,k} +\eta_{n,k-1} + \eps_{n,k} + \rho_{n,k}\big)\, ,~~~~ \forall k \ge0\, .
\end{align}
We obtain that $\plim_{n \to \infty} \eta_{n,k+1} =0$, and therefore the AMP iteration Eq.~\eqref{eq:AMPreminder} satisfies the conclusions of Proposition~\ref{prop:state_evolution}.  
\end{proof}

As in \cite[Eqs.\ (69,70)]{deshpande2017asymptotic} we argue that the state evolution equations
 \eqref{eq:AMP-nu}, \eqref{eq:AMP-sig} take a simple form thanks to our specific choice of AMP 
 non-linearity $\tanh(\cdot)$.

\begin{proposition}
\label{prop:bayes-opt-SE}
For any $t\in\mathbb R_{\geq 0}$ and $k,j\in\mathbb Z_{\geq 0}$,
\begin{align}
    \Lambda_{k,j}(\beta,t) &=q_{k\wedge j}(\beta,t)\, , ~~~\mbox{and}~~~ 
    \Sigma_{k,j}(\beta,t)=\gamma_{k\wedge j}(\beta,t)\, .
\end{align}
\end{proposition}

\begin{proof}
 It will be convenient to use the notations
\begin{align*}
    \wtgamma_{k}(\beta,t)&=\gamma_{k}(\beta,t)+t \, ,~~~\mbox{and}~~~ 
    \wtSigma_{k,j}(\beta,t)=\Sigma_{k,j}(\beta,t)+t\, .
\end{align*}

The two claims are equivalent and we proceed by induction. The base case $k=0$ holds by 
definition, so we may assume $\Lambda_{i,j}(\beta,t)=q_{i\wedge j}(\beta,t)$ for $i,j\le k-1$. Set
 \[
    V_{j}=\wtgamma_{j}X+\wtW_{j}
\]
where $\wtbW \sim \normal(0,\widetilde{\bSigma}_{\le k-1})$. By the induction hypothesis, $V_{k-1}$ is a sufficient statistic for $X$ given 
 $(V_j)_{j\le k-1}$.
 Using Bayes' rule, and writing $\wtsigma_{k-1}^2:=\wtSigma_{k-1,k-1}$, one easily computes
\begin{equation}
\label{eq:bayes-tanh}
\begin{aligned}
    \mathbb E[X|V_{k-1}]
    &=
    \frac{e^{\wtgamma_{k-1}V_{k-1}/\wtsigma_{k-1}^2}-e^{-\wtgamma_{k-1}V_{k-1}/\wtsigma_{k-1}^2}}{e^{\wtgamma_{k-1}V_{k-1}/\wtsigma_{k-1}^2}+e^{-\wtgamma_{k-1}V_{k-1}/\wtsigma_{k-1}^2}}
    \\
    &=\tanh(V_{k-1}) \, .
\end{aligned}
\end{equation}
Therefore using Eq.~\eqref{eq:AMP-nu}, the fact that $\tanh$ is an odd function and $ZX \stackrel{\rmd}{=} X$,
\begin{align*}
    \E\big[ \tanh(\wtgamma_{k-1}X+\wtsigma_{k-1} Z)\tanh(\wtgamma_{j-1}X+\wtsigma_{j-1} Z) \big]
    &= 
    \E\big[\E[X|V_{k-1}]\E[X|V_{j-1}] \big]
    \\
    & \stackrel{(a)}{=} \E\big[X\E[X|V_{j-1}] \big]
    \\
    &=\E\left[ X  \tanh(\wtgamma_{j-1}X+\wtsigma_{j-1} Z)\right]
    \\
    &=\E\left[ \tanh(\wtgamma_{j-1}+\wtsigma_{j-1} Z)\right] 
    .
\end{align*}
Step $(a)$ used the sufficient statistic property. This yields yields $\Lambda_{k,j}=q_j(\beta,t)$. Applying $\beta^2\xi'(\cdot)$ to the previous relation yields $\Sigma_{k,j}=\gamma_j(\beta,t)$. This completes the induction and proof.
\end{proof}

\subsubsection{Convergence to the State Evolution fixed point}
\label{sec:ConvergenceToFP}
It follows from Proposition~\ref{prop:bayes-opt-SE} that the state
evolution recursion \eqref{eq:AMP-nu} and \eqref{eq:AMP-lam} 
can be expressed just in terms of a scalar recursion for $q_k(\beta,t)$. 
Define
$\psi:\bbR_{\geq 0}\to [0,1)$ as in the statement of Theorem \ref{thm:main},
namely   
\begin{equation}\label{eq:psi}
  \psi(\gamma)=\bbE\big[\tanh(\gamma + \sqrt{\gamma}Z)\big] \, , ~~~~Z\sim\normal(0,1) \, .
\end{equation}
We further denote by $\phi$ the inverse of $\psi$ on $\R_{\ge 0}$. 
It is not hard to show that both $\psi$ and $\phi$ are smooth and strictly increasing with 
$\psi(0)=\phi(0)=0$. 
Moreover, $\psi$ is concave and $\phi$ is convex.

We have the recursion (c.f.\ Eq.~\eqref{eq:q_k})
\begin{align}
\label{eq:q-recusion}
  q_{k+1} & = f_t(q_k)\, ,\;\;\;\;\;\; q_0=0\, ,\\
    f_t(q)& := \psi\big(\beta^2 \xi'(q)+t\big)\, .
\end{align} 

Recall the definition of $\beta_1$ given in the statement of Theorem \ref{thm:main},
cf.  Eq.~\eqref{eq:Beta0Def}.
Equivalently,
\begin{equation}
\label{eq:beta0-equiv}
    \beta\leq\beta_1\quad\iff \quad\beta^2\xi''(q)\leq \phi'(q)
    ~~
    \forall q\in [0,1] \, .
\end{equation}
The next lemma clarifies the significance of $\beta_1$.
\begin{lemma}
\label{lem:contraction}
    If $\beta<\beta_1$, then for all $q\geq 0$, we have
    \[
        0\le f_t'(q)\leq (\beta/\beta_1)^2<1 \, .
    \]
In other words, $f_t$ is a contraction $|f_t(q_1)-f_t(q_2)|\le c(\beta) |q_1-q_2|$
for some $c(\beta)<1$.
\end{lemma}

\begin{proof}
    Recall from \eqref{eq:beta0-equiv} that $\beta_1^2\xi''(q)\le \phi'(q)$ and by integration, 
    $\beta_1^2 \xi'(q)\le \phi(q)$. Thus
    \begin{align*}
        f_t'(q)
        &=
        \beta^2 \xi''(q) \psi'(\beta^2 \xi'(q)+t)
        \\
        &\leq (\beta/\beta_1)^2
        \phi'(q) \psi'(\phi(q)+t)
        \\
         &\stackrel{(a)}{\leq} (\beta/\beta_1)^2
        \phi'(q) \psi'(\phi(q))
        \\
        &=
        (\beta/\beta_1)^2
        \frac{\rmd }{\rmd q} \psi(\phi(q))
        \\
        &=
        (\beta/\beta_1)^2\, ,
    \end{align*}
    where in $(a)$ we used the fact that $\psi$ is concave \cite{deshpande2017asymptotic}.
\end{proof}

We now prove some useful regularity properties for $q_{k}(\beta,t)$.
\begin{lemma}
\label{lem:properties}
Let $\beta<\beta_1$. Then the following properties hold, where $\{q_k(\beta,t)\}_{k\geq 0}$ is as defined in Eq.~\eqref{eq:q-recusion}.
\begin{enumerate}[label=(\alph*)]
    \item 
    \label{it:mmse-decreasing-convex} 
    $\gamma \mapsto \psi(\gamma)$ is differentiable,  strictly increasing, and strictly concave in 
    $\gamma\in\mathbb R_{\geq 0}$.
    \item
    \label{it:mmse-boundary-values} 
    $\psi(0)=0$, $\psi'(0)=1$ and $\lim_{\gamma\to\infty}\psi(\gamma)=1$.
    \item\label{it:mmseExistence} For $t\geq 0$, the fixed point equation
     \begin{equation}
        \label{eq:gamma*}
       q=\psi\big(\beta^2\xi'(q)+t \big)\, .
    \end{equation}
   has a 
     unique non-negative solution which we denote by $q_*=q_*(\beta,t)$.
   Further, the sequence $q_k$ is monotone increasing with  $\lim_{k\to \infty} q_k = q_*$.
    \item \label{it:mmseSolSmooth} The function $(\beta,t)\mapsto q_*(\beta,t)$ is 
    continuously differentiable ($C^1$) for $(\beta,t)\in [0,\beta_1)\times\R_{\ge 0}$. 
    \item \label{it:ExpConv} For all $t>0$, 
    \begin{equation}
    \label{eq:uniform-gamma-limit}
        1-(\beta/\beta_1)^{2k} \le \frac{q_{k}(\beta,t)}{q_*(\beta,t)} \le 1 \, .
    \end{equation}
    \item 
    \label{it:gamma*-near0}
    For $\T>0$, there exist constants $c(\beta,\T),C(\beta,\T)\in (0,\infty)$
    such that, for all $t\in (0,\T]$, 
    \begin{equation}
    \label{eq:gamma*-near0}
        c(\beta,\T)\leq \frac{q_*(\beta,t)}{t}\leq C(\beta,\T) \, .
    \end{equation}
    \item
    \label{it:gamma*-Lipschitz}
    For any $t_1,t_2\in (0,T)$,
    \begin{align}
    \label{eq:a*-Lipschitz}
        q_*(\beta,t_1)-q_*(\beta,t_2)&\leq 
        C(\beta,T) |t_1-t_2| \, , 
        \\
    \label{eq:gamma*-Lipschitz}    
        \gamma_*(\beta,t_1)-\gamma_*(\beta,t_2)&\leq C(\beta,T) |t_1-t_2|\, .
    \end{align}
    (Recall that we defined $\gamma_k(\beta,t) := \beta^2\xi'(q_k(\beta,t))$,
    and correspondingly we let $\gamma_*(\beta,t) := \beta^2\xi'(q_*(\beta,t))$.)
\end{enumerate}
\end{lemma}
\begin{proof}
Lemma 6.1 in \cite{deshpande2017asymptotic} proves that $\gamma\mapsto\psi(\gamma)$ 
 is differentiable,  strictly increasing, and concave in $\gamma\in\mathbb R_{\geq 0}$.
 Note that the statement of that Lemma does not claim differentiability, but this is actually proved 
 there by a simple application of dominated convergence.
  Further \cite{deshpande2017asymptotic}, only claims $\gamma\mapsto\psi(\gamma)$  is concave
  (not strictly concave). However, this function is real analytic on $\R_{>0}$ and hence it 
  must be strictly concave.
  This proves point 
 \ref{it:mmse-decreasing-convex}.
 
  Point \ref{it:mmse-boundary-values} follows by a direct calculation, c.f. \cite{deshpande2017asymptotic}.
  Indeed,  by Stein's lemma (Gaussian integration by parts), with $V=\gamma+Z\sqrt{\gamma}$,
\begin{align*}
    \psi'(\gamma)&=\frac{\rmd~}{\rmd\gamma}\E[\tanh(\gamma+Z\sqrt{\gamma})^2]\\
    &=\E[2\tanh(V)\tanh'(V)+\tanh'(V)^2 +\tanh(V)\tanh''(V)]
\end{align*}
Evaluating at $\gamma=0$ shows
\[
    \psi'(0)=1.
\]
Also, dominated convergence yields the desired limit values.

Consider next claim \ref{it:mmseExistence}. The existence and uniqueness of  
a fixed point, i.e. a solution of $q=f_t(q)$ follows from the fact that  
$f_t$ is a contraction, and $f_t(0) =\psi(t)\ge 0$.  
As for the fact that $q_k\uparrow q_*$, this is trivial for $t=0$ ($q_*=0$ in this case).
 For $t>0$, we further
have $f_t(q)>q$ for $q\in [0,q_*)$ and  $f_t(q)<q$ for $q\in (q_*,\infty)$,
By induction $q_k<q_*$ for all $k$ and $q_{k+1} = f_t(q_k)>q_k$,
whence we have that $q_k$ is an increasing sequence. Defining limit 
 $q_{\infty} := \lim_{k\to\infty} q_k\le q_*$, we have $q_{\infty} = f_t(q_{\infty})$ (continuity) and 
therefore $q_{\infty}=q_*$.

Writing for clarity $f_t(q) = f(q;\beta,t)$, claim \ref{it:mmseSolSmooth}
follows from the definition $q_*(\beta,t)-f(q_*(\beta,t);\beta,t)=0$ using the
fact that $\partial_q f(q_*(\beta,t);\beta,t)<1$  proven
in Lemma \ref{lem:contraction}, together with the implicit function theorem.
 Note that the implicit function theorem argument also gives expressions for the derivatives 
 of  of $q_*$. Namely
 \begin{align*}
    \partial_{\beta} q_* = \frac{\partial_{\beta} f(q_*;\beta,t)}{1-\partial_q f(q_*;\beta,t)}\, ,\;\;\;\;
   \partial_{t} q_* = \frac{\partial_t f(q_*;\beta,t)}{1-\partial_q f(q_*;\beta,t)}\, .
   \end{align*}
Since $f$ is $C^1$, and  $q_*$ is continuous, we deduce that
it is also $C^1$. So is $\gamma_*(\beta,t) =
\beta^2\xi'((q_*(\beta,t))$, because $(\beta,q)\mapsto \beta^2\xi'((q_*(\beta,t))$
is $C^1$.

Claim \ref{it:ExpConv} follows from the contraction property proven
in Lemma \ref{lem:contraction}, since this implies
\begin{equation}
\label{eq:gamma-linear-convergence}
    0
    \leq (q_*-q_{k+1})
=
    \big (f_t(q_*)-f_t(q_k))
    \leq 
    (\beta/\beta_1)^2(q_*-q_k) \, .
\end{equation}

Equation \eqref{eq:gamma*-near0} follows if we can prove $\partial_tq_*(\beta,t=0)\in (0,\infty)$.
This follows immediately from the above proof of differentiability, since
\begin{align}
\partial_{t}q_*(\beta,0) = \left.\frac{\partial_t f(q;\beta,t)}{1-\partial_q f(q;\beta,t)}
\right|_{q=0,t=0}\, ,
\end{align}
and it is easy to see that $\partial_t f(0;\beta,0)>0$.

Finally, Eqs.~\eqref{eq:a*-Lipschitz} and \eqref{eq:gamma*-Lipschitz} follows from the fact that
$q_*$, $\gamma_*\in C^1([0,\beta_1)\times\R_{\ge 0})$, and hence their derivative is bounded on 
$[0,T]$. 
\end{proof}

As mentioned in Remark \ref{rmk:Beta0}, 
an equivalent characterization of $\beta_1$ is given by the following result. 
\begin{lemma}\label{lem:Beta0}
We have
\begin{align}
\beta_1(\xi)&= \sup\big\{\beta:\; q = \psi(t+\tilde\beta^2\xi'(q)) 
\mbox{ has a unique solution for all }(\tilde\beta,t)\in [0,\beta)\times\R_{\ge 0} 
\big\}\, .
\end{align}
\end{lemma}
\begin{proof}
Write $\beta_1'$ for the value defined on the right-hand side. Then of course 
$\beta_1\le \beta_1'$ because by Lemma \ref{lem:properties}, the solution is unique for all 
$(\beta,t)\in [0\beta_1)\times \R_{\ge 0}$. 

To prove $\beta_1\ge \beta_1'$, we recall that $\gamma_*(\beta,t) = \beta^2\xi'(q_*(\beta,t))$.
Notice that this solves $\gamma = \beta^2 F(\gamma+t)$, where $F(x):= \xi'(\psi(x))$.
First of all $t\mapsto \gamma_*(\beta,t)$ is continuous for $\beta\in [0,\beta'_0)$.
Indeed, if it weren't at a point $t_0$
we could define $\gamma_{-}(\beta,t_0)<\gamma_{+}(\beta,t_0)$ to be the limits as $t\uparrow t_0$,
$t\downarrow t_0$ of $\gamma_*(\beta,t)$. But then by continuity of $f_t$, it
would follow that both $\gamma_{-}(\beta,t_0)$ and $q_{+}(\beta,t_0)$ are fixed points
at $t=t_0$, hence contradicting the assumption. 

Next, we claim that $\beta^2F'(\gamma_*(\beta,t)+t)<1$ for all $(\beta,t)\in [0\beta'_0)\times \R_{\ge 0}$. 
Let $\obeta\in(\beta,\beta'_0)$.
By uniqueness of solutions, we have $\obeta^2F'(\gamma_*(\obeta,t)+t)\le 1$ for all $t$. 
On the other hand,
$t\mapsto \gamma_*(\obeta,t)$ increases continuously from $0$ to $\obeta^2\xi'(1)$
as $t$ goes from $0$ to $\infty$.  Therefore $\obeta^2\sup_{x\ge}F'(x)\le 1$.
Finally
\begin{align}
\beta^2 F'(\gamma_*(\beta,t)+t) \le \beta^2\sup_{x\ge}F'(x) \le \left(\frac{\beta}{\obeta}\right)^2<1\, .
\end{align}
Hence
\begin{align}
\beta_1'\le \inf_{x\ge 0} \frac{1}{\sqrt{F'(x)}} = \inf_{x\ge 0} \frac{1}{\sqrt{\xi''(\psi(x))
\psi'(x)}} = \beta_1\,.
\end{align} 
\end{proof}

\subsubsection{Mean squared error}

 For $(\bG,\bx,\by)\sim \P$ (the planted model), define
 \begin{equation}\label{eq:MSE_AMP}
     \MSE_{\AMP}(k;\beta,t)=
     \lim_{n\to\infty}
     \frac{1}{n}
     \E \big\|\bx-\hm^k(\bG,\by(t)) \big\|_2^2 \, ,
     \;\;\;\hm^k(\bG,\by(t)):=\AMP(\bG,\by(t);k)\,,
 \end{equation}
 where the limit is guaranteed to exist by Proposition~\ref{prop:state_evolution}. 

 \begin{lemma}
 \label{lem:MSE-k}
 We have
 \begin{align*}
     \MSE_{\AMP}(k;\beta,t)
     &=1-q_{k+1}(\beta,t) \, .
 \end{align*}
 In particular,
 \begin{align*}
 \lim_{k\to\infty}\MSE_{\AMP}(k;\beta,t)&=1-q_*(\beta,t) \, .
 \end{align*}
 \end{lemma} 
 \begin{proof}
We recall the notation  $\wtgamma_{k}(\beta,t)=\gamma_{k}(\beta,t)+t$, $\wtsigma_{k}(\beta,t)=\sigma_k(\beta,t)+t$, and $\sigma_k(\beta,t)=\Sigma_{k,k}(\beta,t)$. By state evolution
 \begin{align*}
     \MSE_{\AMP}(k;\beta,t)&=\lim_{n\to\infty}\frac{1}{n}\E\big\|\hm^k(\bG,\by(t))-\bx\big\|_2^2\\
     &= \E\big[\big(\tanh(\gamma_k X+\sigma_k Z+Y)-X\big)^2\big]\\
     &= \E\big[\big(\tanh(\wtgamma_k X+\wtsigma_k Z)-X\big)^2\big]\\
     &= 1 - 2 \E[\tanh(\wtgamma_k X+\wtsigma_k Z)X]+\E[\tanh(\wtgamma_k X+\wtsigma_k Z)^2]\\
     &= 1 - 2 q_{k+1} + \Lambda_{k+1,k+1} \\
     &=1 - q_{k+1}\, ,
 \end{align*}
 where the last line follows from Proposition~\ref{prop:bayes-opt-SE}.
 \end{proof}

 We next show that,
 for any $t>0$, the mean square error achieved by AMP is the same as the Bayes optimal error,
 i.e., the mean squared error achieved by the posterior expectation $\m(\bG,\by(t))$.
 The proof of this relies on the characterization of the mutual information between 
 $\bx$ and the set of observations $(\bG,\by(t))$ as defied in the planted model 
 Eq.~\eqref{eq:P}.  Let us write $\bG=\bG(\beta)$ to emphasize the dependence of the joint distribution
 of $(\bx,\bG)$ on $\beta$.  
 Let $I(X;Y)$ denote the mutual information between random variables $X,Y$
 on the same probability space. Letting $X\sim\Unif(\{-1,+1\})$ independent of 
 $Z\sim\normal(0,1)$,  define the function 
 \begin{align}
 \info(\gamma) &:= I\big(X;\gamma X+\sqrt{\gamma} Z\big) \\
 & = \gamma -\E\log\cosh\big(\gamma +\sqrt{\gamma} Z)\, .
 \end{align}
 We also define the function
 \begin{align}\label{eq:psi}
 \Psi(q;\beta,t):= \frac{\beta^2}{2}\Big(\xi(1) - \xi(q) - (1-q)\xi'(q)\Big) + \info\big(\beta^2\xi'(q)+t\big)\, . 
 \end{align}

Then we have then following `single letter' characterization.
\begin{theorem}\label{thm:IT}
For $\beta < \min(\beta_1,\beta_2)$ and all $t$ we have
  \begin{align}\label{eq:IT}
  &\lim_{n\to\infty} \frac{1}{n}I(\bx;\bG(\beta),\by(t))=\Psi(q_*(\beta,t);\beta,t)=:\Psi_*(\beta,t)\, ,\\
 \label{eq:MMSE}
   &  \lim_{n\to\infty} \E\Big[\frac{1}{n}\big\|\bx-\m(\bG,\by(t))\big\|_2^2\Big]=
     1 - q_{*}(\beta,t) \, .
 \end{align}
\end{theorem}
\begin{proof}
This can be proved by the same argument already introduced in \cite{deshpande2017asymptotic}.
First notice that $q_*(\beta,t=0) = 0$, whence $\Psi_*(\beta,t=0) = \beta^2\xi(1)/2$.
On the other hand, denoting by $\frac{\de\mu_{\pl}}{\de\mu_{\rd}}(\bG(\beta)|\bx)$
the Radon-Nikodym derivative of $\mu_{\pl}(\bG(\beta)|\bx)$ with respect to 
$\mu_{\rd}(\de\bG)$,  we have
\begin{align*}
I(\bx;\bG(\beta),\by(t=0)) &= -\E_{\pl}\log \frac{\de\mu_{\pl}}{\de\mu_{\rd}}(\bG(\beta))+
\E_{\pl}\log\frac{\de\mu_{\pl}}{\de\mu_{\rd}}(\bG(\beta)|\bx)\\
& = -\E_{\pl} \log Z_{\xi}(\bG) +\frac{n}{2} \beta^2\xi(1)\\
& = -\E_{\rd} Z_{\xi}(\bG)\log Z_{\xi}(\bG) +\frac{n}{2} \beta^2\xi(1)\\
& = \frac{n}{2} \beta^2\xi(1) +O(1)\, ,
\end{align*}
where the last equality follows from Theorem~\ref{thm:fluctuations_Z}. Summarizing,
we have
\begin{align}
\lim_{n\to\infty}\frac{1}{n}I(\bx;\bG(\beta),\by(t=0)) = \Psi_*(\beta,t=0)\, .
\label{eq:Psi0}
\end{align}

On the other hand, by the data-processing inequality
\begin{align*}
\log 2 \ge \frac{1}{n}I(\bx;\bG(\beta),\by(t))\ge  \frac{1}{n}I(\bx;\by(t)) = \info(t)\, .
\end{align*}
By dominated convergence, we obtain $\info(t)\to \log 2$ as $t\to\infty$, and therefore 
\begin{align}
\lim_{T\to\infty}\lim_{n\to\infty}\frac{1}{n}I(\bx;\bG(\beta),\by(T)) = 
\Psi_*(\beta,t=\infty)=\log 2\, .\label{eq:PsiInfty}
\end{align}

By de Brujin's identity (also known as I-MMSE relation \cite{Guo05mutualinformation})  we have
 \begin{align}
 \frac{\de\phantom{t}}{\de t}I(\bx;\bG(\beta),\by(t)) &=
 \frac{1}{2}\E\left[\big\|\bx-\m(\bG,\by(t))\big\|^2\right]   \, , \label{eq:DerivativeIBeta}
 \end{align}
and therefore, using Eqs.~\eqref{eq:Psi0}, \eqref{eq:PsiInfty}, we get
\begin{align}
\lim_{T\to\infty}\lim_{n\to\infty} \frac{1}{2n}\int_{0}^{T}\E\left[\big\|\bx-\m(\bG,\by(t))\big\|^2\right]  
\, \de t = \Psi_*(\beta,\infty) - \Psi_*(\beta,0)
  \, , \label{eq:Integral}
 \end{align}

On the other hand, using the the differentiability of $q_*$ with respect to $t$
(see Lemma \ref{lem:properties})
and the fact that $\partial_q\Psi(q;\beta,t)|_{q=q_*}=0$ (which follows from the fixed 
point condition to define $q_*$), we obtain
\begin{align}
 \frac{\de\phantom{t}}{\de t} \Psi_*(\beta,t) = 1-q_*(\beta,t)=   
 \lim_{k\to\infty}\MSE_{\AMP}(k;\beta,t)\, .
\end{align}
Therefore, by optimality of the conditional expectation,
\begin{align}
\lim\sup_{n\to\infty} \frac{1}{n}\E\left[\big\|\bx-\m(\bG,\by(t))\big\|^2\right] \le 
 \frac{\de\phantom{t}}{\de t} \Psi_*(\beta,t) = 1-q_*(\beta,t) \, . \label{eq:BoundMSE}
\end{align}
Using Eq.~\eqref{eq:Integral}, we get
\begin{align*}
\Psi_*(\beta,\infty) - \Psi_*(\beta,0) &=  
\lim_{T\to\infty}\lim_{n\to\infty} \frac{1}{2n}\int_{0}^{T}\E\left[\big\|\bx-\m(\bG,\by(t))\big\|^2\right]  
\, \de t
\\
&\stackrel{(a)}\le \lim_{T\to\infty} \frac{1}{2}\int_{0}^{T} \big( 1-q_*(\beta,t) \big) \, \de t\\
&=\lim_{T\to\infty} \frac{1}{2}\int_{0}^{T} \frac{\de\phantom{t}}{\de t} \Psi_*(\beta,t)\, \de t\\
&=\Psi_*(\beta,\infty) - \Psi_*(\beta,0) \,,
\end{align*}
where in $(a)$ we used dominated convergence to exchange limit and integral.
Since the first tem of the chain of inequalities coincides with the last, all
the inequalities must hold with equality.
Because of Eq.~\eqref{eq:BoundMSE}, this implies that for almost every $t$,
\begin{align}
\lim_{n\to\infty} \frac{1}{n}\E\left[\big\|\bx-\m(\bG,\by(t))\big\|^2\right] = 1-q_*(\beta,t)\, .
\end{align}
Since  $t\mapsto q_*(\beta,t)$ is continuous by Lemma \ref{lem:properties} and
the left-hand side is monotone (by monotonicity of the mean squared error), the above must 
hold for every $t$. This proves Eq.~\eqref{eq:MMSE}. Equation \eqref{eq:IT}
follows by applying once more the I-MMSE relation.
\end{proof}

 It follows that AMP approximately computes the posterior mean 
 $\m(\bG,\by(t))$ 
 in the following sense.
 \begin{proposition}
 \label{prop:amp-posterior-mean}
 Fix $\beta<\beta_1$, $\T>0$ and let $t \in (0,\T]$. Recalling that 
 $\hm^k(\bG,\by(t)):=\AMP(\bG,\by(t);k)$ denotes the AMP estimate after $k$ iterations, 
 and that $\bz^k$ is defined by Eq.~\eqref{eq:AMPreminder}, we have for all $\eps>0$
  \begin{align}
 \label{eq:m-converge}
     &\limsup_{n\to\infty}\, \P\left(
     \frac{\|\m(\bG,\by(t))-\hm^k(\bG,\by(t))\|_2}{\|\m(\bG,\by(t))\|_2} \ge \eps \right)\le \frac{8}{\eps^2}(\beta/\beta_1)^{2k+2}\, .
     \end{align}
     Moreover
     \begin{align}
 \label{eq:z-converge}
     &\lim_{k\to\infty}\sup_{t\in (0,\T)}\plim_{n\to\infty} \frac{\|\bz^{k+1}-\bz^k\|_2}{\|\bz^k\|_2}=0 \, .
 \end{align}
 \end{proposition}
 %
 %
 \begin{proof}
 Throughout this proof we write $\by, q_*, q_k$ instead of $\by(t), q_*(\beta,t), q_k(\beta,t)$ respectively for ease of notation. 
 To show Eq.~\eqref{eq:m-converge}, 
 observe that the bias-variance decomposition yields
 (recalling the definition $\MSE_{\AMP}(\;\cdot\;)$ in Eq.~\eqref{eq:MSE_AMP})
 \begin{align*}
     \MSE_{\AMP}(k;\beta,t)
     &=
     \lim_{n\to\infty}\left\{
     \frac{1}{n}
     \E
     \Big[
     \big\|\hm^k(\bG,\by)-\m(\bG,\by)\big\|_2^2\Big]+
     \frac{1}{n}
     \E
     \Big[\big\|\bx_0- \m(\bG,\by)\big\|_2^2
     \Big]\right\}.
 \end{align*}
 Using Lemma \ref{lem:MSE-k} for the left-hand side and
 Theorem~\ref{thm:IT} for the second term on the right-hand side, 
 we get
 \begin{equation}\label{eq:limgamma}
     \lim_{n\to\infty}
     \frac{1}{n}
     \E\left[
     \big\|\hm^k(\bG,\by)-\m(\bG,\by) \big\|_2^2
     \right] = q_*-q_{k+1} \, .
 \end{equation}
Now by the triangle inequality,
  \begin{align*}
  \Big|\frac{1}{\sqrt{n}}\big\|\m(\bG,\by) \big\|_2 - \sqrt{q_{*}}\Big| &\le \frac{1}{\sqrt{n}}\big\|\hm^k(\bG,\by)-\m(\bG,\by) \big\|_2+ \Big|\frac{1}{\sqrt{n}}\big\|\hm^k(\bG,\by)\big\|_2 -\sqrt{q_{k+1}}\Big| \\
 &~~~+  \big|\sqrt{q_{k+1}}- \sqrt{q_{*}}\big| := A+B+C\, .
 \end{align*}
We have $\lim_{k \to \infty} C = 0$ by Lemma \ref{lem:properties}. We also have $\plim_{n \to \infty} B = 0$ by Proposition~\ref{prop:state_evolution}. Next we have $\lim_{k \to \infty} \lim_{n \to \infty}\E[A^2] = 0$ by Eq.~\eqref{eq:limgamma} and Lemma \ref{lem:properties}. Combined together we have 
\begin{equation}\label{eq:plim_m}
\plim_{n \to \infty} \frac{1}{\sqrt{n}}\big\|\m(\bG,\by) \big\|_2 = \sqrt{q_{*}} \, .
\end{equation}

Using Eq.~\eqref{eq:limgamma}, there exists $n_0(k)$ such that for all $\eps>0$ and all $n \ge n_0(k)$ we have
\begin{align}
\P\Big(\frac{1}{\sqrt{n}}\|\m(\bG,\by)-\hm^k(\bG,\by)\|_2 \ge \frac{\eps\sqrt{q_*}}{2}\Big) &\le \frac{4}{\eps^2q_*} \cdot 2(q_* - q_{k+1}) \nonumber\\
&\le \frac{8}{\eps^2} (\beta/\beta_1)^{2k+2} \label{eq:lim_D}\, ,
\end{align}
where the last line was obtained by Lemma \ref{lem:properties} (c).

Now, letting $E$ be the event that $\big|\big\|\m(\bG,\by) \big\|_2/\sqrt{n} -\sqrt{q_*}\big| \le \sqrt{q_*}/2$, it follows that
\begin{align}
\P\Big(\frac{\|\m(\bG,\by)-\hm^k(\bG,\by)\|_2}{\|\m(\bG,\by)\|_2} \ge \eps\Big) &\le \P\Big(\Big\{\frac{1}{\sqrt{n}}\|\m(\bG,\by)-\hm^k(\bG,\by)\|_2 \ge \frac{\eps\sqrt{q_*}}{2}\Big\} \cap E \Big)   + \P\big(E^c\big)\, .
\end{align}
Taking $n \to \infty$ and using Eq.~\eqref{eq:plim_m} and Eq.~\eqref{eq:lim_D} completes to the proof of the first claim~\eqref{eq:m-converge}.

 Finally, Eq.~\eqref{eq:z-converge}
 is an immediate consequence of Proposition~\ref{prop:state_evolution}
 and Proposition \ref{prop:bayes-opt-SE}. Indeed, by  Proposition~\ref{prop:state_evolution},
 we have
 \begin{align}
     \plim_{n\to\infty} \frac{1}{n}\big\|\bz^k\big\|_2^2 
     &=  \E\big[(\gamma_k X + W_k+ Y)^2\big] =  (\gamma_k+t)^2 + \gamma_k+t \, ,\\
       \plim_{n\to\infty} \frac{1}{n}\big\|\bz^{k+1}-\bz^k\big\|_2^2 
     &=  \E\big[\big((\gamma_{k+1}-\gamma_k) X +W_{k+1}-W_k)^2\big]\\
     & =  (\gamma_{k+1}-\gamma_k)^2 + (\Sigma_{k+1,k+1}-2\Sigma_{k,k+1}+\Sigma_{k,k}) \\
     & =  (\gamma_{k+1}-\gamma_k)^2 + (\gamma_{k+1}-\gamma_k)\, ,
 \end{align}
 where in the last step we used  Proposition \ref{prop:bayes-opt-SE}.
 We therefore obtained
 we have
 \begin{align}
     \plim_{n\to\infty} \frac{\|\bz^{k+1}-\bz^k\big\|_2^2}{\|\bz^k\|_2^2} =
      \frac{(\gamma_{k+1}-\gamma_k)^2 + (\gamma_{k+1}-\gamma_k)}{(\gamma_k+t)^2 + \gamma_k+t}\, .
      \end{align} 
 Hence Eq.~\eqref{eq:z-converge} also follows from Eq.~\eqref{eq:uniform-gamma-limit}.
 \end{proof}

 We conclude this subsection with a lemma controlling the regularity of the 
 posterior  path $t\mapsto \m(\bG,\by(t))$, which will be useful later.
 \begin{lemma}
 \label{lem:uniform-path}
 Fix $\beta<\beta_1$ and $0\leq t_1<t_2\leq \T$. Then 
 \begin{align}
     \lim_{n\to\infty}\sup_{t\in [t_1,t_2]} \frac{1}{n} \big\|\m(\bG,\by(t))-\m(\bG,\by(t_1))\big\|_2^2&= 
     \lim_{n\to\infty} \frac{1}{n} \big\|\m(\bG,\by(t_2))-\m(\bG,\by(t_1)) \big\|_2^2\\
     &=q_*(\beta,t_2)-q_*(\beta,t_1)\, . 
     \label{eq:uniform-path}
 \end{align}
 \end{lemma}
 \begin{proof}
 We will exploit the fact that $(\m(\bG,\by(t)))_{t\ge 0}$ is a martingale,
 as a consequence of Lemma~\ref{prop:stochloc1} (with $\varphi:\mathbb R^n\to\mathbb R^n$ given by
  $\varphi(\bx)=\bx$).
 
 Using Theorem~\ref{thm:IT}, we obtain, for any $t_1<t_2$
 \begin{align*}
     \lim_{n\to\infty} 
      \frac{1}{n}\E\big[\big\|\m(\bG,\by(t_2))-\m(\bG,\by(t_1))\big\|_2^2 \big]
     &=  \lim_{n\to\infty} 
      \frac{1}{n}\Big\{\E\big[\big\|\bx-\m(\bG,\by(t_1))\big\|_2^2 \big]
      -\E\big[\big\|\bx-\m(\bG,\by(t_1))\big\|_2^2 \big]\Big\}\\
      & = 
     q_*(\beta,t_2)-q_*(\beta,t_1)\, ,
 \end{align*}
 where the first equality uses the fact that $\E[\m(\bG,\by(t_2))|\bG,\by(t_1)]=
 \m(\bG,\by(t_1))$. By Proposition \ref{prop:amp-posterior-mean}, we have, with high probability,
 $\|\m(\bG,\by(t_i))-\hm^k(\bG,\by(t))\|_2^2/n\le \eps_k$, for some deterministic constants
 $\eps_k$ so that $\eps_k\to 0$ as $k\to\infty$. As a consequence
 \begin{align}
     \plim_{n\to\infty} 
      \frac{1}{n}\big\|\m(\bG,\by(t_2))-\m(\bG,\by(t_1))\big\|_2^2
     &=  q_*(\beta,t_2)-q_*(\beta,t_1)\, .\label{eq:MM-gamma}
 \end{align}
    
 Now, since $t\to \m(\bG,\by(t))$ is a bounded martingale, it follows that,
 for any fixed constant $c$, the process 
 \begin{equation}
  Y_{n,t} := \max\big\{M_{n,t} - c, 0\big\} \, , ~~~\mbox{where}~~~ M_{n,t} := 
  \frac{1}{\sqrt{n}} \big\|\m(\bG,\by(t))-\m(\bG,\by(t_1))\big\|_2 \, ,
 \end{equation}
 is a positive bounded submartingale for $t\ge t_1$.
  Therefore by Doob's maximal inequality~\cite{durrett2019probability},
 \begin{equation}
     \P\Big(\sup_{t\in [t_1,t_2]}Y_{n,t} \geq a \Big)
     \leq 
     \frac{1}{a} \E\big[Y_{n,t_2} \big] \le \frac{1}{a} \E\big[Y_{n,t_2}^2 \big]^{1/2} \, ,
 \end{equation}
 for any $a>0$. We choose $c =  \sqrt{q_*(\beta,t_2)-q_*(\beta,t_1)}$. 
 By \eqref{eq:MM-gamma}, we have
 \[
     \plim_{n\to\infty} M_{n,t_2}^2 = q_*(\beta,t_2)-q_*(\beta,t_1) =c^2\, ,
 \]
 and therefore, since $M_{n,t}$ is bounded, for any fixed $a>0$
 \begin{align*}
 \lim_{n\to\infty}\P\Big(\sup_{t\in [t_1,t_2]}M_{n,t} \geq c + a \Big) &\le 
 \lim_{n\to\infty}\P\Big(\sup_{t\in [t_1,t_2]}Y_{n,t} \geq a \Big) \\
 &\le \frac{1}{a}\lim_{n\to\infty}\E\big[(M_{n,t_2}-c)_+^2 \big]^{1/2} = 0\, .
 \end{align*}
 Together with Eq.~\eqref{eq:MM-gamma}, this yields 
 \[
     \plim_{n\to\infty}\sup_{t\in [t_1,t_2]}M_{n,t}^2=q_*(t_2)-q_*(t_1) \, ,
 \]
 which coincides with the claim \eqref{eq:uniform-path}.
 \end{proof}

\subsection{Natural Gradient Descent}
\label{sec:ngd}

\begin{algorithm}
\label{alg:NGD}
\DontPrintSemicolon 
\KwIn{Initialization $\bu^0\in \R^n$, disorder $\bG$, $\hby\in \mathbb R^n$, step size
 $\eta>0$, $q\in (0,1)$, integer $K>0$.}
$\hm^{+,0} = \tanh(\bu^0)$. \\
\For{$k = 0,\cdots,K-1$} { 
$\bu^{k+1} \leftarrow \bu^k - \eta \cdot\nabla \widehat{\cuF}_{\sTAP}(\hm^{+,k};\by,q)$,  \label{line:explicit-NGD-step} \\
$\hm^{+,k+1} = \tanh(\bu^{+,k+1})$,
}
\Return{$\hm^{+,K}$}\;
\caption{{\sc Natural Gradient Descent on $\widehat{\cuF}_{\sTAP}(\;\cdot\; ;\by,q)$}}
\end{algorithm}

The main objective of this section is to show that
$\widehat{\cuF}_{\sTAP}(\m;\by,q)$ behaves well for $q=q_*(\beta,t)$
and for $\m$ in a neighborhood of $\hm^{K_{\sAMP}}$. Namely it has a unique local minimum 
$\m_* = \m_*(\bG,\by)$ in such a neighborhood, and
NGD approximates $\m_*$ well for large number of iterations $K$. 
 Crucially, the map 
 $\by \mapsto \m_*(\bG,\by)$ will be Lipschitz.
For reference, we recall that the modified TAP free energy functional was defined as
follows
\begin{align}\label{eq:TAP_reg2}
 \widehat{\cuF}_{\sTAP}(\m ; \by, q) :=  -\beta H_n(\m) - \langle \by , \m \rangle - 
 \sum_{i=1}^n h(m_i) - \ons(q) - \ons'(q)(Q(\m)-q)+\frac{n\Treg\beta}{8} (Q(\m)-q)^2\, ,
  \end{align}
where
\begin{align}
& \ons\big(Q\big) = \frac{\beta^2n}{2}\Big( \xi(1) - \xi(Q) - (1-Q)\xi'(Q \Big)\, ,\\
 & Q(\m) = 
  \frac{1}{n} \|\m\|^2 \, ,
~~~\mbox{and}~~~~~ h(m) = -\frac{1+m}{2}\log \left(\frac{1+m}{2}\right) - \frac{1-m}{2}\log \left(\frac{1-m}{2}\right)  \, .
\end{align}
We also reproduce the NGD algorithm as Algorithm \ref{alg:NGD}.
This corresponds to lines \ref{alg:NGD-begin}-\ref{alg:NGD-end} of Algorithm \ref{alg:Mean}.

In the next statement we omit mentioning the dependence of various constants on $\xi$.
 \begin{lemma}
 \label{lem:local-landscape}
 Let $\beta_1,\beta_3$ be defined by Eqs.~\eqref{eq:Beta0Def}, \eqref{eq:Beta3Def}.
Then,  for any $\beta<\min(\beta_1,\beta_3)$ and any $\T>0$,
there exists $\eps_0 = \eps_0(\beta,\T)$, $\Treg=\Treg(\beta)$ such that, for all $\eps\in (0,\eps_0)$
 there exists $K_{\sAMP} = K_{\sAMP}(\beta,\T,\eps)$ and 
  $\rho_0 =\rho_0(\beta,\T,\eps)$ such that 
  for all $\rho\in (0,\rho_0)$  there exists  $K_{\sNGD} = K_{\sNGD}(\beta,\T,\eps,\rho)$,
  such that the 
  following holds.
 
  Let $\hm^{\sAMP}= \AMP(\bG,\by(t);K_{\sAMP})$ be the output of the AMP after 
  $K_{\sAMP}$ iterations, when applied to $\by(t)$. Fix $K\ge K_{\sAMP}$.
 With probability $1-o_n(1)$ over $(\bG,\by)\sim\P$, 
 for all $t\in (0,\T]$ and all $\hby\in \Ball^n\left(\by(t),c\sqrt{\eps tn} / 4\right)$,
 setting $q_{*} := q_{*}(\beta,t)$:
 \begin{enumerate}
     \item 
     \label{it:landscape-basic}
     The function 
     \[
         \m\mapsto \widehat{\cuF}_{\sTAP}(\m ; \hby, q_{*})
     \]
     restricted to $\Ball^n\left(\hm^{\sAMP},\sqrt{\eps tn}\right)\cap (-1,1)^n$ has a unique stationary point 
     \[
         \m_*(\bG,\hby)\in  \Ball^n\left(\hm^{\sAMP},\sqrt{\eps tn} / 2\right)\cap (-1,1)^n
     \]
     which is also a local minimum. In the case $\hby=\by(t)$, $\m_*(\bG,\by(t))$ 
     also satisfies 
     \[
         \m_*(\bG,\by)\in  \Ball^n\left(\hm^{k'},\sqrt{\eps tn} / 2\right)\cap (-1,1)^n
     \]
     for all $k'\in [K_{\sAMP},K]$, where $\hm^{k'} = \AMP(\bG,\by(t);k')$.
     \item
     \label{it:landscape-stationary-point-good}
     The stationary point $\m_*(\bG,\hby)$ satisfies
     (recall that $\m(\bG,\hby)$ denotes the mean of the Gibbs measure)
     \[
         \big\|\m(\bG,\hby)-\m_*(\bG,\hby) \big\|_2\leq \rho\sqrt{tn} \, .
     \]
     \item The stationary point $\m_*$ satisfies the following Lipschitz property
     for all  $\hby,\hby'\in \Ball^n\left(\by(t),c\sqrt{\eps tn} / 4\right)$:
     \label{it:landscape-lipschitz}
     \begin{equation}
     \label{eq:stationary-point-lipschitz}
         \big\|\m_*(\bG,\hby)-\m_*(\bG,\hby')\big\| \leq c^{-1} \|\hby-\hby'\| \, .
     \end{equation}
     \item 
     \label{it:landscape-NGD}
     There exists a learning rate $\eta=\eta(\beta,\T,\eps)$ such that the following holds.
     Let $\hm^{\sNGD}(\bG,\hby)$ be the output of NGD (Algorithm
     \ref{alg:NGD}), when run for $K_{\sNGD}$ iterations with parameter $q_{*}$,
     $\hby$, $\eta$. Assume that the initialization $\bu^0$ satisfies 
     \begin{align}\label{ass:landscape-NGD}
         \big\|\bu^0-\atanh(\hm^{\sAMP})\big\|\leq \frac{c\sqrt{\eps t n}}{200}  \, .
     \end{align}
     Then the algorithm output satisfies 
     \begin{equation}
     \label{eq:landscape-NGD-convergence}
         \big\| \hm^{\sNGD}(\bG,\hby)-\m_*(\bG,\hby) \big\| \leq \rho\sqrt{tn} \, .
     \end{equation}
 \end{enumerate}
 \end{lemma}
 The proof of this lemma is deferred to the appendix.

 Here we will prove the two
key elements: first that  $\hm^{\sAMP}$ is an approximate stationary point of
 $\widehat{\cuF}_{\sTAP}(\;\cdot\;;\by(t),q_*)$
  (Lemma~\ref{lem:TAP-stationary}), and second that
  $\widehat{\cuF}_{\sTAP}(\;\cdot\;;\hby,q_*)$ is strongly convex in a neighborhood of 
  $\hm^{\sAMP}$ (Lemma~\ref{lem:local-convex}). Let us point out that, in the local convexity 
  guarantee, it is important that the neighborhood has radius 
  $\Theta(\sqrt{tn})$ as $t\to 0$.

We recall below the expressions for the gradient and Hessian of  $\widehat{\cuF}_{\sTAP}(\;\cdot\;;\by,q)$
at
$\m\in (-1,1)^n$:
\begin{align}
\label{eq:grad-F}
    \nabla\widehat{\cuF}_{\sTAP}(\m;\by,q)    
    &=
    -\beta \nabla H_n(\m)-\by+\atanh(\m)
    +
    \beta^2(1-q)\xi''(q)\m +\frac{\Treg\beta}{2} (Q(\m)-q)\m\, ,
    \\
\label{eq:hess-F}
    \nabla^2\widehat{\cuF}_{\sTAP}(\m;\by,q)&=
    -\beta \nabla^2 H_n(\m)+\D(\m)
    +
   \Big( \beta^2(1-q)\xi''(q)+ \frac{\Treg\beta}{2}(Q(\m)-q)\Big)\bI_n 
   +\frac{\Treg\beta}{n}\m\m^{\sT}\,, \;\; 
    \\
\nonumber
    \bD(\m)&:={\rm diag}\big(\{(1-m_i^2)^{-1}\}_{i\le n}\big).
\end{align}
In \eqref{eq:grad-F}, $\atanh$ is applied coordinate-wise to $\m\in (-1,1)^n$.

Recall that, by  Lemma~\ref{lem:MSE-k}, we have
\begin{align}
\label{eq:qk*-lim}
    q_k(\beta,t) = \plim_{n\to\infty}\frac{\big\|\hm^k\big\|^2}{n} 
\, ,\;\;\;\;\;
  q_*(\beta,t) =\lim_{k\to\infty}q_k(\beta,t) \, .
\end{align}
 We will use the bounds~\eqref{eq:uniform-gamma-limit}, \eqref{eq:gamma*-near0} in Lemma~\ref{lem:properties} several times below, 
 which ensures that $(q_k(\beta,t)/t)\in [c,C]$ holds for constants $c,C>0$ 
 independent of $t\in (0,\T]$ and $k\geq 1$.

\begin{lemma}
\label{lem:TAP-stationary}
Let $\hm^k = \hm^k(\bG,\by(t))$ denote the AMP iterates on input $\bG,\by(t)$. 
Then for any $ \T>0$,
\[
    \lim_{k\to\infty}\sup_{t\in (0,\T]}\sup_{q\in [q_k(\beta,t),q_*(\beta,t)]}\plim_{n\to\infty} \frac{\big\|\nabla \widehat{\cuF}_{\sTAP}(\hm^k;\by(t),q)\big\|}{\sqrt{tn}}=0 \, .
\]
\end{lemma}

\begin{proof}
As in Algorithm \ref{alg:Mean}, let
\[
  \bz^{k+1}=\atanh(\hm^{k+1})=\beta \nabla H_n(\hm^k)+\by-\beta^2(1- \hat{q}^k) \xi''(\hat{q}^k)\, \hm^{k-1} \, , ~~~\hat{q}^k = \frac{1}{n}\big\|\hm^{k}\big\|^2 \, .
\]
Let $q\in [q_k(\beta,t),q_*(\beta,t)]$. Combining the above with 
Eqs.~\eqref{eq:grad-F} and \eqref{eq:qk*-lim} yields
\begin{align*}
    \frac{1}{\sqrt n}\|\nabla \widehat{\cuF}_{\sTAP}(\hm^k;\by,q)\|
    &=
    \frac{1}{\sqrt n}\left\|
    -
    \beta \nabla H_n(\hm^k)-\by
    +
    \atanh(\hm^k)
    +
    \beta^2(1-q)\xi''(q)\hm^k\right\| 
    \\
    &=
    \frac{1}{\sqrt n}\left\|
    \bz^k-\beta \nabla H_n(\hm^k)-\by
    +
    \beta^2(1-q)\xi''(q)\hm^k
    \right\| +\frac{\Treg\beta}{2}\big|\hat{q}^k-q\big|
    \\
    &\leq 
    \frac{1}{\sqrt n}\|\bz^{k+1}-\bz^{k}\|
    +
    \frac{1}{\sqrt n}\left\|\bz^{k+1}
    -
    \beta \nabla H_n(\hm^k)-\by+\beta^2(1-q)\xi''(q)\hm^k\right\|+\frac{\Treg\beta}{2}\big|\hat{q}^k-q\big|
    \\
    &=\frac{1}{\sqrt n}\big\|\bz^{k+1}-\bz^{k} \big\|
    +\frac{\beta^2}{\sqrt n}\left\| (1- \hat{q}^k) \xi''(\hat{q}^k)\hm^{k-1} - (1-q)\xi''(q)\hm^k \right\|
    +\frac{\Treg\beta}{2}\big|\hat{q}^k-q\big|
    \\
    &\leq 
    \frac{1}{\sqrt n}\big\|\bz^{k+1}-\bz^{k}\big\|
    +
    \frac{\beta^2 \xi''(1)}{\sqrt n} \, \big\|\hm^{k-1}-\hm^k\big\| \\
    &~~~+ 
    \sup_{q_k \le q \le q_*} \beta^2 \big| (1-q)\xi''(q)-(1-q_k)\xi''(q_k)\big|
    +\frac{\Treg\beta}{2}\big|q^k-q\big|+
    o_{n,\mathbb P}(1).
\end{align*}
Here $o_{n,\mathbb P}(1)$ denotes terms which converge to $0$ in probability as $n\to\infty$, and \ $\|\cdot\|_{\infty}$ denotes the supremum norm on $[0,1]$. Further, for the last term we have
\[ \sup_{q_k \le q \le q_*} \beta^2 \big| (1-q)\xi''(q)-(1-q_k)\xi''(q_k)\big| \le \beta^2(\|\xi''\|_{\infty} + \|\xi'''\|_{\infty})(q_* - q_k) \, .\] 
By \eqref{eq:z-converge}, \eqref{eq:qk*-lim} and the bound $(q_k(\beta,t)/t)\in [c,C]$ 
\[
    \lim_{k\to\infty}\sup_{t\in (0,T)}\plim_{n\to\infty}\frac{\|\bz^{k+1}-\bz^{k}\|}{\sqrt{tn}}=0\, .
\]
Moreover, $\|\hm^{k-1}-\hm^k\|\leq \|\bz^{k-1}-\bz^k\|$ since the function $x\mapsto \tanh(x)$ 
is $1$-Lipschitz. Finally \eqref{eq:uniform-gamma-limit} and \eqref{eq:gamma*-near0} of Lemma~\ref{lem:properties} imply
\[
    \lim_{k\to\infty}\sup_{t\in (0,\T]} \frac{q_*(\beta,t)-q_k(\beta,t)}{\sqrt{t}}=0 \, .
\]
Combining the above statements concludes the proof.
\end{proof}

We next control the Hessian $\nabla^2\widehat{\cuF}_{\sTAP}(\;\cdot\;;\by,q)$. We begin by 
a bound on the Hessian of the Hamiltonian, whose proof is deferred to Appendix 
\ref{app:HessianHamiltonian}.
\begin{lemma}\label{lemma:HessianHamiltonian}
There exist constants $\Treg_0(\xi)$, $C_0(\xi)>0$, depending uniquely on $\xi$
and a universal constant $C_*$, such that, defining
$\oxi^{(\ell)}(1)  :=\sum_{p=2}^Pc_p^2p^\ell$ and, for $\delta>0$,
\begin{equation}\label{eq:Kxi}
\begin{aligned}
K_{\delta}(\xi)&:= 
\begin{dcases}
(2 +\delta)\sqrt{\xi''(1)} & \mbox{ if  } \xi(t) = c_2^2t^2\, ,\\
C_*\sqrt{\xi''(1)\log\oxi^{(8)}(1)}& \mbox{ otherwise}\, ,
\end{dcases}
\end{aligned}
\end{equation}
the following holds. For any $\delta>0$, and any $\Treg>\Treg_0(\xi)$,
 the following happens with probability at least $1-2\exp(-n\delta^2/2)$ (for  $\xi(t) = c_2^2t^2$) 
 or $1-2\exp(-nC_0(\xi))$ (otherwise). For all  $\m\in (-1,1)^n$, we have
\begin{align}
 -K_{\delta}(\xi)\, \bI_n \preceq \D(\m)^{-1/2} \big(-\nabla^2 H_n(\m) + \frac{\Treg}{n}\m\m^{\sT}\big)\D(\m)^{-1/2} \preceq
 (K_{\delta}(\xi)+\Treg)\, \bI_n\, .\label{eq:lemma:HessianHamiltonian}
\end{align}
\end{lemma}
We next pass to the Hessian of the TAP free energy.
\begin{lemma}
\label{lem:local-convex}
Let  $K_{\delta}(\xi)$ be defined as in Eq.~\eqref{eq:Kxi} and, for 
 $\beta >0$, $\by \in \R^n$,
let $\widehat{\cuF}_{\sTAP}(\m;\by,q)$ be defined as per Eq.~\eqref{eq:TAP_reg2}
with $\Treg =  \Treg_0(\xi)$ chosen as in Lemma \ref{lemma:HessianHamiltonian}.
Then there exists $C_0 = C_0(\xi,\beta,\delta)$ and $\eps_1 = \eps_1(\xi)$ such that
the following holds with probability at least $1-2\exp(-C_0n)$. For all $\m\in (-1,1)^n$,
and $q$ such that $|Q(\m)-q|\le \eps_1$, we have 
 \begin{equation}
  \label{eq:global-smooth}
    \big(1-\beta K_{\delta}(\xi)\big) \D(\m) \preceq  \nabla^2\widehat{\cuF}_{\sTAP}(\m;\by,q)\preceq 
     \big(1 + \beta K_{\delta}(\xi) + \beta \Treg_0
     +2\beta^2\xi''(1)\big) \D(\m) \, .
 \end{equation}
 In particular, letting $\beta_3:=1/K_0(\xi) >0$, 
 as in \eqref{eq:Beta3Def} if $\beta<\beta_3$ 
  for all $\m\in (-1,1)^n$ and $q$ such that $|Q(\m)-q|\le \eps_1$, we have
 \begin{equation}  \label{eq:local-convex}
c \, \D(\m) \preceq \nabla^2\cuF_{\sTAP}(\m;\by,q) \preceq C \D(\m) \, ,
 \end{equation}
for some constants $c=c(\beta,\xi)$, $C= C(\beta,\xi)$.
\end{lemma}

\begin{proof}
Using the expression~\eqref{eq:hess-F} for $\nabla^2\widehat{\cuF}_{\sTAP}(\m;\by,q)$
and  Lemma \ref{lemma:HessianHamiltonian}, we obtain that, with probability 
at least $1-2\exp(-C_0n)$
\begin{align}
 \big(1-\beta K_{\delta}(\xi)\big) \D(\m) 
  +
  b(\m)\bI_n
\preceq
\nabla^2\widehat{\cuF}_{\sTAP}(\m;\by,q) 
  \preceq
  \big(1+\beta (K_{\delta}(\xi)+\Treg_0)\big) \D(\m) 
  + b(\m)\bI_n \,,
\end{align}
where $b(\m):=  ( \beta^2(1-q)\xi''(q)+ (\Treg_0\beta/2)(Q(\m)-q))$.
By choosing $\eps_1 =  \beta(1-q)\xi''(q)/\Treg_0$,  
we obtain $0\le b(\m)\le 2\beta^2\xi''(1)$. Using $\id_n\preceq \D(\m)$,
we obtain the claim \eqref{eq:global-smooth}, whence \eqref{eq:local-convex}
follows.
\end{proof}

As mentioned above, our convergence analysis of NGD, and proof of Lemma~\ref{lem:local-landscape}
are given in Appendix~\ref{app:NGD}. 
The key insight is that the main iterative step in line \ref{line:explicit-NGD-step} of 
Algorithm~\ref{alg:NGD} can be expressed as a version of mirror descent. 
Define the concave function $h(\m)=\sum_{i=1}^n h(m_i)$ for $\m\in (-1,1)^n$
(recall that $h(m) := -((1+m)/2)\log((1+m)/2)-((1-m)/2)\log((1-m)/2)$). 
Following \cite{lu2018relatively}, we define for $\m,\n\in (-1,1)^n$ the Bregman divergence
\begin{equation}
\begin{aligned}
\label{eq:bregman}
    \Dh(\m,\n)&=-\bh(\m)+\bh(\n)+\langle \nabla \bh(\n),\m-\n\rangle\ .
\end{aligned}
\end{equation}
Then with $L=1/\eta$, the update in line~\ref{line:explicit-NGD-step} admits the alternate description
\begin{equation}
\label{eq:NGD}
    \hm^{+,k+1}=\argmin_{\bx\in (-1,1)^n} \big\langle \nabla \widehat{\cuF}_{\sTAP}(\hm^{+,k};\by,q),\bx-\hm^{+,k} \big\rangle + L\cdot \Dh(\bx,\hm^{+,k})\, .
\end{equation} 
We will use this description to prove convergence.

\begin{remark}
If the Hessian $\nabla^2 \widehat{\cuF}_{\sTAP}$ were bounded above and below by constant multiples of the
 \textbf{identity} matrix instead of $\D(\m)$, then we could use simple gradient descent instead
 of NGD in Algorithm \ref{alg:Mean}. This would also simplify the proof.
 However, $\nabla^2 \widehat{\cuF}_{\sTAP}$ is not bounded above near the boundaries of $(-1,+1)^n$.
 The use of NGD to minimize TAP free energy was introduced in \cite{celentano2021local},
 which however considered a different regime in the planted model. 
\end{remark}

\subsection{Continuous time limit and proof of Theorem~\ref{thm:main}}
\label{sec:proofMain}

We fix $(\beta,\T)$ and assume, without loss of generality, $\eps\in (0,\eps_0(\beta,T))$.
We can then choose $K_{\sAMP}=K_{\sAMP}(\beta,\T,\eps)$,
 $\rho_0=\rho_0(\beta,\T, \eps, K_{\sAMP})$, $\rho\in (0,\rho_0)$
 and $K_{\sNGD} = K_{\sNGD}(\beta,\T,\eps,\rho)$ so that Lemma~\ref{lem:local-landscape} holds. 
 
 We couple the discretized process $(\hby_{\ell})_{\ell\ge 0}$ defined in Eq.~\eqref{eq:approx} 
(line \ref{step:DiscreteSDE} of Algorithm \ref{alg:Sampling})
 to the continuous time process $(\by(t))_{t\in \R_{\geq 0}}$ of Eqs.~\eqref{eq:Q},
 \eqref{eq:Q-bis}
 via the driving noise, as follows: 
\begin{equation}
\label{eq:BM-couple}
    \bw_{\ell+1}=\frac{1}{\sqrt{\delta}} \int_{\ell\delta}^{(\ell+1)\delta}\rmd \bW(t) \, .
\end{equation}
We denote by $\hm(\bA,\by)$ the output of the mean estimation procedure
(Algorithm 
\ref{alg:Mean}) on input $\bA,\by$.  By Lemma~\ref{lem:local-landscape} (combined with the contiguity in Theorem~\ref{thm:fluctuations_Z}), we know that, for any $t\in (0,\T]$, 
with probability $1-o_n(1)$,
\begin{equation}
\label{eq:TAP-near-optimizer}
    \big\|\hm(\bG,\by(t))-\m_*(\bG,\by(t);q_*(\beta,t))\big\|\leq 
     \rho \sqrt{t n} \, .
\end{equation}
Here and below we note explicitly the dependence of $\m_*$ on $t$ via $q_*$.

The next lemma is an imediate consequence of Lemma \ref{lemma:RoughBoundThirdDerivative}.
\begin{lemma}\label{lemma:LipGrad}
Define the event
\begin{align}
\cuG(J) :=\Big\{\bG:\; \|\nabla H(\m_1)-\nabla H(\m_2)\|\le J\|\m_1-\m_2\|\;\;\forall \m_1,\m_2\in[-1,1]^n\Big\}\, .
\end{align}
Then there exists $J_*,C_0>0$ depending uniquely on $\xi$ such that $\P(\cuG(J_*))\ge 1-2\exp(-nC_0)$.
\end{lemma}
\begin{proof}
At $\m = 0$, $\nabla^2 H(\bfzero) = \sqrt{\xi''(0)/n}\cdot\bW$ where $\bW$ is a GOE
matrix. Hence $\|\nabla^2 H(\bfzero)\|_{\op}\le 3\sqrt{\xi''(0)}$ with probability
at least $1-2\exp(-n/C)$. Since
 on the high event of Lemma \ref{lemma:RoughBoundThirdDerivative}
 $\m\mapsto \nabla^2 H(\m)$ is Lipschitz continuous in operator norm,
 we have $\max_{\m\in\Ball^n(\sqrt{n})} \|\nabla^2 H(\m)\|_{\op}\le C$ for some new constant $C$
 and therefore $\m\mapsto \nabla H(\m)$ is Lipschitz as well.
 \end{proof}
The above lemma implies Lipschitz continuity of AMP with respect to its 
input, as stated below.
\begin{lemma}
\label{lem:AMP-lip}
Recall that $\AMP(\bG,\by;k)\in \R^n$ denotes the output of the AMP 
algorithm on input $(\bG,\by)$, after $k$ iterations, cf. Eq.~\eqref{eq:AMP}.
On the high probability event $\cuG(J_*)$ of Lemma  \ref{lemma:LipGrad}
there exists a constant $C_{\#}=C_{\#}(\beta\xi)<\infty$ such that for any $\by,\hby\in \R^n$, 
\begin{equation}
\label{eq:AMP-lip}
    \big\|\atanh\big(\AMP(\bG,\by;k)\big) - \atanh\big(\AMP(\bG,\hby;k)\big)\big\|_2\leq C_{\#}^k \,\|\by-\hby\|_2\, .
\end{equation}
\end{lemma}
\begin{proof}
For $0\leq j\leq k$, set:
\begin{align*}
 \m^j&=\AMP(\bG,\by;j),
 &&\bz^j=\atanh(\m^j),
 &&\sb_j=\beta^2 (1-\hat{q}^k)\xi''\big(\hat{q}^k) ,\;\;\;\;\hat{q}^k = 
    \frac{1}{n}\sum_{i=1}^n  \tanh^2(z^{k}_i)\, ,
 \\
\hm^j&=\AMP(\bG,\hby;j),
&&\hbz^j=\atanh(\hm^j),
&&\hat \sb_j=\beta^2 (1-\doublehat{q}^k)\xi''\big(\doublehat{q}^k) ,\;\;\;\;\doublehat{q}^k = 
    \frac{1}{n}\sum_{i=1}^n  \tanh^2(\hat z^{k}_i)\, .
\end{align*}
Using the AMP update equation (line \ref{eq:AMP-main-step} of Algorithm \ref{alg:Mean})
 and the fact that $\tanh(\, \cdot\, )$ is $1$-Lipschitz, we obtain
\begin{align*}
    \|\bz^{j+1}-\hbz^{j+1}\|
    &\leq \|\beta (\nabla H(\m^j)-\nabla H(\hm^j))\| + 
    \|\by-\hby\| + 
    \|\sb_j \m^{j-1}-\sb_j \hm^{j-1}\| + 
    \|\sb_j \hm^{j-1}-\hat \sb_j \hm^{j-1}\| 
    \\
     &\leq \beta J_* \|\bz^{j}-\hbz^{j}\| + \|\by-\hby\| + \sb_j \|\bz^{j-1}-\hbz^{j-1}\| + |\sb_j-\hat \sb_j|\sqrt{n} \, .
\end{align*}
Note that $\hq^k, \doublehat{q}^k\in [0,1]$ and therefore $|\sb_j|,
|\hat \sb_j|\leq \beta^2\xi''(1)$. 
Further $q\mapsto \beta^2(1-q)\xi''(q)$ is Lipschitz for $q\in [0,1]$. Denoting by 
$\Lip_{\sb}$ its Lipschitz constant, we have
\begin{align*}
 |\sb_j-\hat \sb_j|&\le \frac{\Lip_{\sb}}{n}\sum_{i=1}^n\big|  \tanh^2(z^{j}_i)
 - \tanh^2(\hat z^{j}_i)\big|\\
 & \le \frac{2\Lip_{\sb}}{n}\sum_{i=1}^n\big|  z^{j}_i
 - \hat z^{j}_i\big|\\
 & \le \frac{2\Lip_{\sb}}{\sqrt{n}}\cdot\|\bz-\hbz^k\|_2
 \end{align*}
 Setting $E_j=\max_{i\leq j}\|\bz^{i}-\hbz^{i}\|$, we  thus find
 event $\cuG(J_*)$ of Lemma  \ref{lemma:LipGrad}, 
 there exists a constant  $C(\beta,\xi)$ such that , for all $j$, 
\begin{align*}
    E_{j+1} &\le C(\beta,\xi)\cdot E_j + \|\by-\hby\|  \, .
\end{align*}
The proof is concluded by noting that, by induction 
\[
    E_j \leq \big(C(\beta,\xi)+1\big)^{j+1}\|\by-\hby\| \, .
\]
\end{proof}

In the next lemma we bound the random approximation errors defined below:
\begin{align}
    A_{\ell} &:= \frac{1}{\sqrt{n}} \big\|\hby_{\ell}-\by({\ell}\delta)\big\| \, ,\\
    B_{\ell} &:= \frac{1}{\sqrt{n}} \big\|\hm(\bG,\hby_{\ell})-\m(\bG,\by({\ell}\delta))\big\| \, .
\end{align} 
\begin{lemma}
\label{lem:discretization-works}
For $\beta<\bar\beta$ and $\T>0$, there exists a constant $C = C(\beta) <\infty$,
and a deterministic non-negative sequence $\eta(n)$ with $\lim_{n\to\infty}\eta(n)= 0$  
such that the following holds with probability $1-o_n(1)$. 
For every $\ell \geq 0$, $\delta \in (0,1)$ such that $\ell\delta \le \T$,
\begin{align}
    \label{eq:A-bound}
    A_{\ell}\leq Ce^{C\ell\delta}\ell\delta \big(\rho\sqrt{\ell\delta}+\sqrt{\delta}\big) + \eta(n) \, ,\\
    \label{eq:B-bound}
     B_{\ell}\leq Ce^{C\ell\delta}\ell\delta \big(\rho\sqrt{\ell\delta}+\sqrt{\delta}\big)+ C \rho\sqrt{\ell\delta}+ \eta(n) \, .
\end{align}

\end{lemma}

\begin{proof}
Throughout the proof, we denote by $\eta(n)$ a deterministic 
non-negative sequence $\eta(n)$ with $\lim_{n\to\infty}\eta(n)= 0$, which can change from 
line to line. Also, $C$ will denote a generic constant that may depend on $\beta,\T,K_{\sAMP}$.

The proof proceeds by induction on $\ell$. The base case is trivial,
since $A_0=B_0=0$.
We assume the result holds for all $ j \le \ell$ and we prove it for $\ell+1$.
 We first claim that with probability $1-o_n(1)$,
\begin{equation}
\label{eq:A-recursion}
    A_{{\ell}+1}\leq A_{\ell}+\delta B_{\ell}+C\delta^{3/2}.
\end{equation}
Indeed, using Eq.~\eqref{eq:BM-couple} we find
\begin{align*}
    A_{{\ell}+1}-A_{\ell}&\leq n^{-1/2}\int_{{\ell}\delta}^{({\ell}+1)\delta} \big\|\hm(\bA,\hby_{\ell})-\m(\bA,\by(t)) \big\|\, \rmd t
    \\
    &\leq 
    \delta n^{-1/2}\Big(\big\|\hm(\bA,\hby_{\ell})-\m(\bA,\by({\ell}\delta))\big\| 
    +
     \sup_{t\in [{\ell}\delta,({\ell}+1)\delta]} \big\|\m(\bA,\by(t))-\m(\bA,\by({\ell}\delta))\big\|\Big)
    \\
    &\leq 
    \delta B_{\ell} 
    + 
    \delta n^{-1/2}\cdot\sup_{t\in [{\ell}\delta,({\ell}+1)\delta]} \big\|\m(\bA,\by(t))-\m(\bA,\by({\ell}\delta))\big\| 
    \\
    &\leq 
    \delta B_{\ell} + C(\beta)\delta^{3/2} + \eta(n)\, ,
\end{align*}
where the last line holds with high probability by Lemma~\ref{lem:uniform-path} and
 Eq.~\eqref{eq:gamma*-Lipschitz} of Lemma~\ref{lem:properties}.
 Using this bound together with the inductive hypothesis on
 $A_{\ell}$ and $B_{\ell}$, we obtain
\begin{align*}
    A_{{\ell}+1}&\leq Ce^{C({\ell}+1)\delta}\ell\delta (\rho\sqrt{\ell\delta} + \sqrt{\delta})+ C\rho\delta\sqrt{\ell\delta} + C\delta^{3/2} + \eta(n)\\
    &\leq Ce^{C(\ell+1)\delta}(\ell+1)\delta (\rho +\sqrt{\delta}) + \eta(n) \, .
\end{align*}
This implies Eq.~\eqref{eq:A-bound} for ${\ell}+1$.

We next show that Eq.~\eqref{eq:B-bound} holds with $\ell$ replaced by  $\ell+1$.
By the bound \eqref{eq:A-bound} for ${\ell}+1$,
taking $\delta \le \delta(\beta,\eps,K_{\sAMP},\T)$ and $\rho\in (0,\rho_0)$
 $\rho=  \rho(\beta,\eps,K_{\sAMP},\T)$ ensures that
\[
    A_{\ell+1}\leq \frac{c\sqrt{\eps \ell \delta}}{200 C_{\#}^{K_{\sAMP}}} \, ,
\]
where $C_{\#}$ is the constant of Lemma \ref{lem:AMP-lip} and $\eps$ can be chosen an arbitrarily small constant.
So by Lemma~\ref{lem:AMP-lip}, we have with probability $1-o_n(1)$,
\begin{align*}
    \big\|\atanh(\AMP(\bA,\by((\ell+1)\delta);K_{\sAMP}))-\atanh(\AMP(\bA,\hby_{\ell+1};K_{\sAMP})) \big\|_2
    &\leq 
    K_{\sAMP} C_{\#}^{K_{\sAMP}}  A_{\ell+1} \sqrt{n}\\
    &\leq \frac{c\sqrt{\eps \ell \delta n}}{200}\, .
\end{align*}
By choosing $\eps\le \eps_0(\beta,\T)$, we obtain that Lemma~\ref{lem:local-landscape}, 
part \ref{it:landscape-NGD} applies. We thus find
\[
    \|\hm(\bA,\hby_{\ell+1})-\m_*(\bA,\hby_{\ell+1})\|\leq \rho\sqrt{\ell\delta n}\, .
\]
Using parts 3 and 2 respectively of Lemma~\ref{lem:local-landscape} on the other terms below, 
by triangle inequality  we obtain
(writing for simplicity $q_{\ell}:=q_{*}(\beta,\ell\delta)$)
\begin{equation}
\begin{aligned}
\label{eq:low-error-TAP}
    \|\hm(\bG,\hby_{{\ell}+1})-\m(\bG,\by(({\ell}+1)\delta))\|
    &\leq 
    \|\hm(\bG,\hby_{\ell+1})-\m_*(\bG,\hby_{\ell+1};q_{\ell+1})\| \\
    &\quad\quad+ \|\m_*(\bG,\hby_{\ell+1};q_{\ell+1})-\m_*(\bG,\by((\ell+1)\delta);q_{\ell+1})\| \\
    &\quad\quad + \|\m_*(\bG,\by((\ell+1)\delta);q_{\ell+1})-\m(\bG,\by(({\ell}+1)\delta))\|\\
    &\leq 
    \big(\rho\sqrt{\ell\delta} + c^{-1}A_{{\ell}+1}+\rho\sqrt{\ell\delta} + \eta(n)\big)\sqrt{n} \, .
\end{aligned}
\end{equation}
In other words with probability $1-o_n(1)$,
\[
    B_{{\ell}+1}\leq c^{-1}A_{{\ell}+1} +2\rho\sqrt{\ell\delta} + \eta(n) \, .
\]
Using this together with the bound \eqref{eq:A-bound} for $\ell+1$
  verifies the inductive step for \eqref{eq:B-bound} and concludes the proof. 
\end{proof}

Finally we show that standard randomized  rounding is continuous in $W_{2,n}$. 
\begin{lemma}
\label{lem:rounding-safe}
Suppose probability distributions $\mu_1,\mu_2$ on $[-1,1]^n$ are given. 
Sample $\m_1\sim \mu_1$ and $\m_2\sim\mu_2$ and let $\bx_1,\bx_2\in \{-1,+1\}^n$ be standard 
randomized roundings, respectively of $\m_1$ and $\m_2$. (Namely,
the coordinates of $\bx_i$ are conditionally independent given $\m_i$,
with $\E[\bx_i|\m_i]=\m_i$.) Then
\[
    W_{2,n}(\mathcal L(\bx_1),\mathcal L(\bx_2))\leq 2\sqrt{W_{2,n}(\mu_1,\mu_2)}\, .
\] 
\end{lemma}

\begin{proof}
Let $(\m_1,\m_2)$ be distributed according to a $W_{2,n}$-optimal coupling
 between $\mu_1,\mu_2$. Couple the roundings $\bx_1,\bx_2$ by choosing i.i.d.\ uniform 
 random variables $u_i\sim\Unif ([0,1])$ for $i\in [n]$, and for $(i,j)\in [n]\times \{1,2\}$ setting
\begin{align*}
 (\bx_j)_i &= 
\begin{cases}
 +1, & \mbox{if}~ u\leq \frac{1+(\m_j)_i}{2} \, ,\\
-1, & \mbox{else.} 
    \end{cases}
\end{align*}
Then it is not difficult to see that
\begin{align*}
    \frac{1}{n} \E\big[\|\bx_1-\bx_2\|^2~|(\m_1,\m_2)\big] &= \frac{2}{n}\sum_{i=1}^n |(\m_1)_i-(\m_2)_i|\\
    & \leq 2\sqrt{\frac{1}{n}\|\m_1-\m_2\|^2 }.
\end{align*}
Averaging over the choice of $(\m_1,\m_2)$ implies the result.
\end{proof}

\begin{proof}[Proof of Theorem~\ref{thm:main}]
Set $\ell=L=\T/\delta$ and $\rho=\sqrt{\delta}$ in Eq.~\eqref{eq:B-bound}. With all laws 
$\cL(\, \cdot\,)$ conditional on $\bA$ below, we find
\begin{align*}
    \E W_{2,n}(\mu_{\bG},\cL(\hm(\bG,\hby_L)))
    &\leq 
    \E W_{2,n}(\mu_{\bG},\cL(\m(\bG,\by(\T))))
    +
    \E W_{2,n}(\cL(\m(\bG,\by(\T)))),\cL(\hm(\bG,\hby_L)))
    \\
    &\leq 
    \T^{-1/2}+C(\beta,\T)\sqrt{\delta} +o_n(1).
\end{align*}
Here the first term was bounded by Eq.~\eqref{eq:W2bound} in Section~\ref{sec:stochloc} and the second by Eq.~\eqref{eq:B-bound}. Taking $\T$ sufficiently large, $\delta$ sufficiently small, and $n$ sufficiently large, we may obtain 
\[
   \E  W_{2,n}\big(\mu_{\bG},\cL(\hat{\m}_{\NGD}(\bG,\hby_L))\big)\leq \frac{\eps^2}{4}
\]
for any desired $\eps>0$. Applying Lemma~\ref{lem:rounding-safe} shows that
\[
    \E W_{2,n}(\mu_{\bG},\bx^{\salg})\leq \eps \, .
\]
The Markov inequality now implies that \eqref{eq:main} holds with probability $1-o_n(1)$ as desired.
\end{proof}

\section{Algorithmic stability and disorder chaos}
\label{sec:stable}

In this section we
prove Theorem \ref{thm:stable} establishing that our sampling algorithm is stable. 
Next, we prove that the Gibbs measure $\mu_{\bG,\beta}$ exhibits $W_2$-disorder chaos for
 $\beta>\beta_c$ and prove Theorem \ref{thm:disorder_chaos_sk}. Finally,
 we deduce Theorem \ref{thm:disorder-stable-LB} establishing hardness for 
 stable algorithms, either under RSB or in the shattering phase.

\subsection{Algorithmic stability: Proof of Theorem \ref{thm:stable}}
\label{sec:algo_stable}
Recall Definition \ref{def:Stable}, defining  sampling algorithms
as measurable functions $\ALG_n:(\bG,\beta,\omega)\mapsto \ALG_n(\bG,\beta,\omega)\in [-1,1]^n$ 
where $\beta\geq 0$ and $\omega$ is an independent random variable taking values in some probability 
space. 

\begin{remark}
In light of Lemma~\ref{lem:rounding-safe}, we can always turn a stable sampling algorithm 
$\ALG$ with codomain $[-1,1]^n$ into a stable sampling algorithm with binary output:
\[
    \widetilde{\ALG}_n(\bG,\beta,\widetilde\omega) \in \{-1,+1\}^n\, .
\]  
Indeed this is achieved by standard randomized rounding, i.e., drawing a
(conditionally independent)  random binary value with mean 
$\big(\widetilde{\ALG}(\bG,\beta,\widetilde\omega)\big)_i$ for each 
coordinate $1 \le i \le n$.
\end{remark}

Recall the definition of the interpolating family $(\bG_s)_{s\in [0,1]}$ in which $\bG_{0},\bG_{1}$ are i.i.d.\ and 
\begin{equation}
\label{eq:interpolation-path}
    \bG_{s}^{(p)} = \sqrt{1-s^2}\, \bG_{0}^{(p)} + s\, \bG_{1}^{(p)} \, ,\quad s\in [0,1],~2\leq p\leq P.
\end{equation}
We let  $H^{(s)}_n(\bx)$ denote the Hamiltonian with disorder  $\bG_{s}$ 
and $\mu_{\bG_s,\beta}(\bx)\propto e^{\beta H^{(s)}_n(\bx)}$ be the corresponding Gibbs measure on
 $\{-1,+1\}^n$. We begin with the following simple estimate.
\begin{lemma}
\label{lem:grad-stable}
There exists a universal constant $C >0$ depending only on $\xi$ such that
\begin{equation}
\label{eq:standard-alg-stable}
    \inf_{s\in (0,1)}
    \mathbb P\Big(
    \big\| \nabla H_n^{(0)}(\bu)- \nabla H_n^{(s)}(\bv) \big\| \leq C(\|\bu-\bv\|+s\sqrt{n}) \, ,~~\forall~\bu,\bv\in [-1,1]^n \Big)=1-o_n(1) \, .
\end{equation}
\end{lemma}
\begin{proof}
We write 
\begin{align*} 
    \big\| \nabla H_n^{(0)}(\bu)- \nabla H_n^{(s)}(\bv)\big\|  
    &\le 
    \big\| \nabla H_n^{(0)}(\bu)- \nabla H_n^{(0)}(\bv) \big\|  
    + 
    \big\| \nabla H_n^{(0)}(\bv)- \nabla H_n^{(s)}(\bv) \big\| 
    \\
    &\le 
    \sup_{\bx\in [-1,1]^n}\|\nabla^2 H_n^{(0)}(\bx)\|_{\text{op}} \, \big\|\bu - \bv\big\| 
    +
    \big\| (1-\sqrt{1-s^2})\nabla H_n^{(0)}(\bv) - s \nabla H_n^{(1)}(\bv)\big\| \, 
    .
\end{align*}
Because $1-\sqrt{1-s^2}\leq s$, the result follows from the fact that 
\begin{align*}
\sup_{\bx\in [-1,1]^n}\|\nabla H_n(\bx)\| \le C\sqrt{n}, \;\;\;
\sup_{\bx\in [-1,1]^n}\|\nabla^2 H_n(\bx)\|_{\text{op}} \le C\, ,
\end{align*}
 are both exponentially likely 
(recall Lemma \ref{lemma:LipGrad}).
\end{proof}

\begin{proposition}
\label{prop:standard-alg-stable}
Suppose an algorithm $\ALG$ is given by an iterative procedure
\begin{align*}
    &\bz^{k+1}
    = 
    F_k\left((\bz^j,\beta \nabla H_n(\m^j),\nabla H_n(\m^j),\beta^2\m^j,\bw^j)_{0\leq j\leq k}\right),
    \quad 0\leq k\leq K-1,\\
    &\m^k =\rho_k(\bz^k),\quad 0\leq k\leq K-1,\\
    &\ALG_n(\bG,\beta,\omega) := \m^K \, ,
\end{align*}
where the sequence $\omega=(\bw^0,\dots,\bw^{K-1})\in (\mathbb R^n)^K$, the initialization 
$\bz^0\in\mathbb R^n$, and $\bG$ are mutually independent, and the functions 
$F_k:(\mathbb R^n)^{5k+5}\to \mathbb R^n$ and $\rho_k:\mathbb R^n\to [-1,1]^n$ are $L_0$-Lipschitz for 
$L_0\geq 0$ an $n$-independent constant. Then $\ALG$ is both disorder-stable and temperature-stable.
\end{proposition}

\begin{proof}

Let us generate iterates $\bz^k=\bz^{k}(\bG_0,\beta)$ and $\wtbz^k=\bz^{k}(\bG_s,\wtbeta)$ for $0\leq k\leq K$ using the same initialization $\bz^0=\wtbz^0$ and external randomness $\omega=(\bw^0,\dots,\bw^{K-1})$, but with different Hamiltonians and inverse temperatures. Similarly let $\m^k=\rho_k(\bz^k)$ and $\wtm^k=\rho_k(\wtbz^k)$. We will allow $C$ to vary from line to line in the proof below.

First by Lemma~\ref{lem:grad-stable}, with probability $1-o_n(1)$,
\begin{align*}
    \|\beta \nabla H_n^{(0)}(\m^k)- \wtbeta\nabla H_n^{(s)}(\wtm^k)\|
    &\leq
    \|\beta  \nabla H_n^{(0)}(\m^k) - \beta \nabla H_n^{(s)}(\wtm^k)\|  
    +
    \|\beta \nabla H_n^{(s)}(\wtm^k) - \wtbeta \nabla H_n^{(s)}(\wtm^k)\|
    \\
    &\leq
    C\beta \|\m^k-\wtm^k\| + C\beta s\sqrt{n}
    + 
    |\beta-\wtbeta|\cdot \| \nabla H_n^{(s)}(\wtm^k)\|\\
    &\leq
    C(\|\m^k-\wtm^k\|+s\sqrt{n}+|\beta-\wtbeta|\sqrt{n})\, .
\end{align*}
 Similarly as long as $\wtbeta\leq 2\beta$ so that $|\beta^2-\wtbeta^2|\leq 3\beta|\beta-\wtbeta|$, we have
\begin{align*}
    \|\beta^2\m^k-\wtbeta^2\wtm^k\|
    &\leq
    \|\beta^2\m^k-\beta^2\wtm^k\|
    +
    \|\beta^2\wtm^k-\wtbeta^2\wtm^k\|\\
    &\leq
    \beta^2 \|\m^k-\wtm^k\| 
    +
    3\beta|\beta-\wtbeta|\sqrt{n}.
\end{align*}
It follows that the error sequence 
\[
    A_k=\frac{1}{\sqrt{n}}\max_{j\leq k}\|\bz^{j+1}(\bG_0,\beta)-\bz^{j+1}(\bG_s,\wtbeta)\|
\] 
satisfies with probability $1-o_n(1)$ the recursion
\begin{align*}
    A_{k+1} & \leq L_0k^{1/2}C(A_k + s + |\beta-\wtbeta|) \, ,\\
    A_0 &= 0 \, ,
\end{align*}
for a suitable $C=C(\beta)$. It follows that with probability $1-o_n(1)$,
\begin{equation}\label{eq:A_k}
    A_K\leq \sum_{k=1}^K (L_0k^{1/2}C)^k (s+ |\beta-\wtbeta|) \leq K(L_0KC)^K (s+|\beta-\wtbeta|) \, .
\end{equation}
Since $\|\m^K(\bG_0)-\m^K(\bG_s)\|\leq 2\sqrt{n}$ almost surely, we obtain for any $\eta>0$
\[
    n^{-1} \E\left[\big\|\m^K(\bG_0)-\m^K(\bG_s)\big\|^2\right] \leq \big(L_0K(L_0KC)^K (s+|\beta-\wtbeta|)\big)^2 + \eta
\]
if $n\geq n_0(\eta)$ is large enough so that Eq.~\eqref{eq:A_k} holds with probability at least $1-\frac{\eta}{4}$. The stability of the algorithm follows.   
\end{proof}

\begin{proof}[Proof of Theorem~\ref{thm:stable}]

We show that Algorithm \ref{alg:Sampling} with 
$n$-independent parameters $(\beta,\eta,K_{\sAMP},K_{\sNGD},L,\delta)$
 is of the form in Proposition~\ref{prop:standard-alg-stable} for a constant 
 $L_0=L_0(\beta,\eta,K_{\sAMP},K_{\sNGD},L,\delta)$.
 Indeed note that the algorithm goes through $L$ iterations, indexed by
 $\ell\in\{0,\dots,L-1\}$. 
 
 During each of these iterations, two loops 
 are run (here we modify the notation introduced in Algorithm \ref{alg:Mean} and Algorithm \ref{alg:Sampling},
 to account for the dependence on $\ell$, and to get closer to the notation of Proposition~\ref{prop:standard-alg-stable}):
\begin{enumerate}
    \item\label{it:stable-case-1}
    The AMP loop, whereby, for $k = 0,\cdots,K_{\sAMP}-1$,
    \begin{align}
    \hm^{\ell,k} &= \tanh(\bz^{\ell,k} ) , ~~~~~~~ \sb_{k,\ell} = \beta^2 (1-\hat{q}_{\ell}^k) \xi''( \hat{q}_{\ell}^k)\, ,\\
\bz^{\ell,k+1} &= \beta \nabla H_n(\hm^{\ell,k}) + \hby_{\ell} - \sb_{k,\ell}\,\hm^{\ell,k-1}\, .
    \end{align}
    Here $\tanh'(x)$ denotes the first derivative of $\tanh(x)$, and $q_{\ell} = q_{K_{\sAMP}}(\beta,t=\ell\delta)$.
    \item\label{it:stable-case-2}
    The NGD loop, whereby, for $k = K_{\sAMP},\cdots,K_{\sAMP}+K_{\sNGD}-1$:
    \begin{align}
    \hm^{\ell,k} & = \tanh(\bz^{\ell,k} ) \, ,\\
    \bz^{\ell,k+1} &= \bz^{\ell,k} + \eta
     \big[\beta \nabla H_n(\hm^{\ell,k})+\by_{\ell}-\bz^{\ell,k}
    - \beta^2\left(1-q_{\ell}\right)\m^{\ell,k} \big]\, . 
    \end{align}
\end{enumerate}
Further, recalling line  \ref{step:DiscreteSDE} of Algorithm \ref{alg:Sampling}, $\hby_{\ell}$
 is updated via
\begin{align}
\hby_{\ell+1} = \hby_{\ell} + \hm^{\ell,K_{\sAMP}+K_{\sNGD}} \, \delta + \sqrt{\delta} \, \bw_{\ell+1}\, .
\label{eq:Yell}
\end{align}

These updates take the same form as in  Proposition~\ref{prop:standard-alg-stable},
with iterations indexed by $(\ell,k)$, $\omega =(\bw_{\ell})_{\ell\le L}$, 
$\rho_{\ell,k}(\bz)=\tanh(\bz)$ for all $\ell,k$,  and 
\begin{align}
 F_{\ell,k}
 \left((\bz^{\ell',j},\beta\nabla H_n(\hm^{\ell',j}),\nabla H_n(\hm^{\ell',j}),\beta^2\hm^{\ell',j},\bw_{\ell'})_{\ell',j}\right)
 &=\beta \nabla H_n(\hm^{\ell,k}) + \hby_{\ell} - \sb_{k,\ell} \hm^{\ell,k-1}\, ,\;\;\;\;\;
 0\le k\le K_{\sAMP}-1\, ,\label{eq:G-AMP}\\
  F_{\ell,k}\left(
  (\bz^{\ell',j},\beta \nabla H_n(\hm^{\ell',j}),\nabla H_n(\hm^{\ell',j}),\beta^2\hm^{\ell',j},\bw_{\ell'})_{\ell',j}
 \right)& \nonumber\\
  =\bz^{\ell,k} + \eta
     \big[\beta \nabla H_n(\hm^{\ell,k})+\by_{\ell}-&\bz^{\ell,k}
    - \beta^2\left(1-q_{\ell}\right)\m^{\ell,k} \big]\, ,\;\;\;\;\;
K_{\sAMP} \le k\le K_{\sAMP}+K_{\sNGD}-1\, .\label{eq:G-NGD}
\end{align}
 Notice that these functions depend on previous iterates both explicitly, as 
 noted, and implicitly through $\hby_{\ell}$. By summing up Eq.~\eqref{eq:Yell},
 we obtain
\begin{align}
\hby_{\ell} = \sum_{j=0}^{\ell-1} \hm^{j,K_{\sAMP}+K_{\sNGD}} \, \delta + 
\sqrt{\delta} \sum_{j=1}^{\ell}  \bw_{\ell+1}\, ,
\end{align}
which is Lipschitz in the previous iterates  $(\m^{j,k})_{j\le \ell-1,k<K_{\sAMP}+K_{\sNGD}}$. 
Since both \eqref{eq:G-AMP} and \eqref{eq:G-NGD} depend linearly on $\hby_{\ell}$
(with $n$-independent coefficients), it is sufficient to consider 
the explicit dependence on previous iterates of $F_{\ell,k}$.
Namely, it is sufficient to control the Lipschitz modulus of the following functions
\begin{align}
 \tF_{\ell,k}\left(\bz^{\ell,k},\beta\nabla H_n(\hm^{\ell,k}),\hm^{\ell,k-1}\right)
 &=\beta \nabla H_n(\hm^{\ell,k}) - \sb_{k,\ell} \hm^{\ell,k-1}
 \, ,\;\;\;\;\;
 k< K_{\sAMP}\label{eq:TG1}\\
  \tF_{\ell,k}\left(\bz^{\ell,k},\beta \nabla H_n(\m^{\ell,k}),\beta^2\m^{\ell,k}\right)
  &=\bz^{\ell,k} + \eta
     \big[\beta \nabla H_n(\hm^{\ell,k})-\bz^{\ell,k}
    - \beta^2\left(1-q_{\ell}\right)\hm^{\ell,k} \big]  
     \, ,\;\;\;\;\;
 k> K_{\sAMP}\, .\label{eq:TG2}
\end{align}

Consider first Eq.~\eqref{eq:TG1}. Since $|\tanh''(x)|\leq 2$ for all $x\in \R$, it follows that
\[
    |\sb(\bz)-\sb(\wtbz)|\leq \frac{2\beta^2}{n} \sum_{i=1}^n  |z_i-\tilde{z}_i|\leq \frac{2\beta^2}{\sqrt{n}} \|\bz-\wtbz\|_2 .
\]
Therefore, that for any $(\bu,\bv,\beta,\wtbu,\wtbv,\wtbeta)$
(noting explicitly the dependence of $\sb$ upon $\beta$):
\begin{align*}
    \|\sb_{\beta}(\bu)\tanh(\bv)-\sb_{\wtbeta}(\wtbu)\tanh(\wtbv)\|
    &\leq
    \|\sb_{\beta}(\bu)\tanh(\bv)-\sb_{\beta}(\wtbu)\tanh(\bv)\|
    +
    \|\sb_{\beta}(\wtbu)\tanh(\bv)-\sb_{\wtbeta}(\wtbu)\tanh(\wtbv)\|
    \\
    &\leq \frac{2\beta^2}{\sqrt{n}} \|\bu-\wtbu\| \cdot \|\tanh(\bv)\| + \Big(\frac{1}{n} \sum_{i=1}^n \tanh'(\tilde{u}_i)\Big)  \|\beta^2\tanh(\bv)-\wtbeta^2\tanh(\wtbv)\| 
    \\
    &\leq
    2\beta^2\|\bu-\wtbu\| + \|\beta^2\tanh(\bv)-\wtbeta^2\tanh(\wtbv)\| .
\end{align*}
Using this bound implies that the function $\tG$ of Eq.~\eqref{eq:TG1}
satisfies the Lipschitz assumption of Proposition~\ref{prop:standard-alg-stable}. 

Consider next Eq.~\eqref{eq:TG2}. Since this function is linear in its arguments, 
with coefficients independent of $n$, it follows that it satisfies  Lipschitz assumption of 
Proposition~\ref{prop:standard-alg-stable}. This completes the proof.
\end{proof}

\subsection{Hardness for stable algorithms: Proof of Theorems \ref{thm:disorder_chaos_sk} and \ref{thm:disorder-stable-LB}}
\label{sec:disorder}

Before proving Theorem \ref{thm:disorder_chaos_sk}, we recall a known result about disorder chaos, already stated in Eq.~\eqref{eq:FirstDisorderChaos}. 
Draw $\bx^0\sim \mu_{\bG_0,\beta}$ independently of  
$\bx^s\sim \mu_{\bG_s,\beta}$, and denote by $\mu^{(0,s)}_{\bG,\beta}:=\mu_{\bG_0,\beta}\otimes
\mu_{\bG^s,\beta}$ their joint distribution.
Then \cite[Theorem 10.5]{chatterjee2014superconcentration} implies that, for all $\beta\in (0,\infty)$,
\begin{align}
\label{eq:FirstDisorderChaos-B}
\lim_{s\to 0}\lim_{n\to\infty}\E_{\mu^{(0,s)}_{\bG,\beta}} \lt[\lt(\frac{\<\bx^0,\bx^s\>}{n}\rt)^2\rt]= 0
\, .
\end{align}
Indeed we have assumed $\xi$ is even in Theorem \ref{thm:disorder_chaos_sk}, and \cite[Theorem 10.5]{chatterjee2014superconcentration} shows that for all $\xi$, 
\[
\lim_{s\to 0}\lim_{n\to\infty}
\E_{\mu^{(0,s)}_{\bG,\beta}} \lt[\xi\lt(\frac{\<\bx^0,\bx^s\>}{n}\rt)\rt]= 0.
\]

The following simple estimate will be used in our proof.
\begin{lemma}\label{lem:continuity}
Recall that $\cuP(\{-1,+1\}^n)$ denotes the space of probability distributions over 
$\{-1,+1\}^n$, and 
let the function $f : \cuP(\{-1,+1\}^n)^2 \to \R$ be defined as
\[
f(\mu,\mu')=  \E_{(\bx,\bx')\sim \mu\otimes\mu'} \lt[\frac{|\langle\bx,\bx'\rangle|}{n} \rt]\, .
\]
Then, for all $\mu_1,\mu_2, \nu_1,\nu_2 \in \cuP(\{-1,+1\}^n)$, we have 
\[
\big| f(\mu_1,\nu_1) - f(\mu_2,\nu_2) \big| \le W_{2,n}(\mu_1,\mu_2) + W_{2,n}(\nu_1,\nu_2) \, .
\]
\end{lemma} 
\begin{proof}
Let the vector pairs $(\bx^{\mu_1},\bx^{\mu_2})$ and $(\bx^{\nu_1},\bx^{\nu_2})$ be 
independently drawn from the optimal $W_{2,n}$-couplings of the pairs $(\mu_1,\mu_2)$ 
and $(\nu_1,\nu_2)$, respectively. Then we have:
\begin{align*}
 \Big| \E\big\{|\< \bx^{\mu_1} , \bx^{\nu_1}\>|\big\} -  
 \E\big\{|\< \bx^{\mu_2} , \bx^{\nu_2}\>|\big\} \Big| 
 &\le 
 \Big| \E\big\{|\< \bx^{\mu_1} , \bx^{\nu_1}\>| -  
 |\< \bx^{\mu_2} , \bx^{\nu_1}\>|\big\} \Big| +
 \Big| \E\big\{|\< \bx^{\mu_2} , \bx^{\nu_1}\>| -  
 |\< \bx^{\mu_2} , \bx^{\nu_2}\>|\big\} \Big| \\
&\le \sqrt{n} \Big(\E\big\|\bx^{\mu_1} -  \bx^{\mu_2}\big\| +  \E\big\|\bx^{\nu_1} - \bx^{\nu_2}\big\|\Big) \\
&\le \sqrt{n} \Big(\E\Big[\big\|\bx^{\mu_1} -  \bx^{\mu_2}\big\|^2\Big]^{1/2} +  \E\Big[\big\|\bx^{\nu_1} - \bx^{\nu_2}\big\|^2\Big]^{1/2}\Big)  \, ,
\end{align*}
where the second inequality follows from the fact that $\bx \mapsto |\<\bv,\bx\>|$ is Lipschitz
continuous with Lipschitz constant $\|\bv\|_2$. 
\end{proof}

We are now in position to prove Theorem~\ref{thm:disorder_chaos_sk}.
\begin{proof}[Proof of Theorem~\ref{thm:disorder_chaos_sk}]
Using the notations of the last lemma Eq.~\eqref{eq:FirstDisorderChaos-B}
 implies that for all $s \in (0,1]$,
\begin{equation}\label{eq:mu0mus} 
\lim_{n \to \infty} \E f( \mu_{\bG_s,\beta}, \mu_{\bG_0,\beta}) = 0 \, . 
\end{equation}
Therefore, Theorem~\ref{thm:disorder_chaos_sk} follows from Lemma \ref{lem:continuity} if we can 
show that 
$ f( \mu_{\bG_0,\beta}, \mu_{\bG_0,\beta})$ remains bounded away from zero in the double limit of \eqref{eq:FirstDisorderChaos-B}. This is in turn a well-known consequence of the Parisi formula, as we recall below. 
 
Recall from \eqref{eq:free_energy} that the free energy density of the mixed $p$-spin model is
\[
F_n(\beta) = \frac{1}{n} \E\, \log \Big\{\sum_{\bx \in \{-1,+1\}^n} e^{\beta H_n(\bx)}
\Big\} \, .
\]
$F_n$ is almost surely convex in $\beta$ and one obtains by Gaussian integration parts that
\begin{equation} 
\label{eq:derivative_free_energy}
\frac{\rmd ~}{\rmd \beta} F_n(\beta)
= 
\beta\lt(\xi(1) - 
\E_{\bx_1,\bx_2\sim \mu_{\bG_0,\beta}}
\lt[\xi\Big(\frac{\<\bx_1, \bx_2\>}{n}\Big)\rt]\rt) \, .
\end{equation}  
The convexity of $F_n$ implies that for almost all $\beta>0$, 
\[
    \plim_{n\to\infty}F_n'(\beta) = \frac{\rmd ~}{\rmd \beta} \Par_{\beta}(\zeta_{\beta}^*).
\]
Moreover as shown in e.g. \cite[Theorem 3.7]{panchenko2013sherrington} or \cite[Theorem 1.2]{talagrand2006parisi-b}, the map $\beta \mapsto \Par_{\beta}(\zeta_{\beta}^*)$ is convex and differentiable at all $\beta > 0$, and
    \begin{equation} 
    \label{eq:derivative_parisi}
    \frac{\rmd ~}{\rmd \beta} \Par_{\beta}(\zeta_{\beta}^*) 
    = 
    \beta\Big( \xi(1) - \int \xi(q) \zeta_{\beta}^*(\rmd q) \Big) \, .
    \end{equation} 

Using Eq.~\eqref{eq:derivative_free_energy} and  Eq.~\eqref{eq:derivative_parisi} we obtain 
\begin{equation}
\label{eq:lim_f}
    \lim_{n \to \infty} 
    \beta
    \Big( 
    \xi(1) 
    - 
    \E_{\bx_1,\bx_2\sim \mu_{\bG_0,\beta}}
    \lt[\Big(\frac{\<\bx_1, \bx_2\>}{n}\Big)^2\rt]\Big)
     = 
     \beta\Big( 1 - \int \xi(q)  \zeta_{\beta}^*(\rmd q) \Big)< \frac{\beta}{2}-\eps(\beta)\, ,
\end{equation}
where the last inequality holds for almost all $\beta>\beta_c$ by Property~\ref{it:lowtemp}
above. Since the both sides are non-decreasing and the right hand side is 
continuous, the inequality holds for all $\beta$. This is equivalent to
\begin{equation}
\label{eq:mu0mu0}
    \lim_{n \to \infty} \E f(\mu_{\bG_0,\beta}, \mu_{\bG_0,\beta}) >0 \, . 
\end{equation}
Now, using Eq.~\eqref{eq:mu0mus} and Eq.~\eqref{eq:mu0mu0}, together with the continuity of $f$ 
(Lemma~\ref{lem:continuity}) implies the claim of the theorem.
\end{proof}

We next prove that Theorem~\ref{thm:disorder-stable-LB} is a consequence
of \eqref{eq:transport_disorder_chaos}.
\begin{proof}[Proof of Theorem \ref{thm:disorder-stable-LB}]
Fix $s\in (0,1)$ and $\mu^{\salg}_{\bG_s,\beta}$ be the law of  $\ALG_n(\bG_s,\beta,\omega)$ 
conditional on $\bG_s$. By the triangle inequality,
\begin{align*}
W_{2,n}(\mu_{\bG_s,\beta}, \mu_{\bG_0,\beta}) 
\le W_{2,n}(\mu_{\bG_s,\beta}, \mu^{\salg}_{\bG_s,\beta}) + 
W_{2,n}(\mu^{\salg}_{\bG_s,\beta}, \mu^{\salg}_{\bG_0,\beta}) + 
W_{2,n}(\mu^{\salg}_{\bG_0,\beta}, \mu_{\bG_0,\beta,0}) \, .
\end{align*}
 Taking expectations over $\bG$ and $\bG_s$, we have 
 $\E\big[W_{2,n}(\mu_{\bG_s,\beta}, \mu^{\salg}_{\bG_s,\beta})\big] = \E\big[W_{2,n}(\mu^{\salg}_{\bG_0,\beta}, 
 \mu_{\bG_0,\beta})\big]$. Further, stability of the algorithm implies
 \begin{equation}
 \label{eq:stability-implies-W2-small}
    \lim_{s\to 0}
    \lim_{n\to 0}
    \E\big[W_{2,n}(\mu^{\salg}_{\bG_s,\beta}, \mu^{\salg}_{\bG_0,\beta})\big] = 0.
\end{equation}
Therefore, using \eqref{eq:transport_disorder_chaos} and choosing $s$ sufficiently small, we obtain
 \[\liminf_{n\to\infty} \E \big[W_{2,n}(\mu^{\salg}_{\bG_0,\beta},~\mu_{\bG_0,\beta}) \big] \ge W_*>0\, .
 \qedhere
 \]
 \end{proof}

\begin{proof}[Proof of Corollary~\ref{cor:concrete-shattering-hardness}]
    Case~\ref{it:shattering-Ising} follows directly from the shattered case of Theorem~\ref{thm:disorder-stable-LB} and \cite{gamarnik2023shattering}, see \cite[Remark 5.2]{alaoui2023shattering}.
    To obtain Case~\ref{it:shattering-spherical}, we note that \cite[Theorem 5.1]{alaoui2023shattering} shows \eqref{eq:transport_disorder_chaos} for $p>C$ and $\beta\in (C,\beta_c)$.
    Finally, Remark~\ref{rem:no-degeneration-near-criticality} is justified by the fact that the positive constant obtained in \eqref{eq:transport_disorder_chaos} via \cite[Theorem 5.1]{alaoui2023shattering} depends only on $(p,r,s)$ for $\beta\in (C,\beta_c)$ (in particular not on the constant $c$ in Definition~\ref{def:Shattering}). 
    Moreover the values of $r,s$ do not degenerate as $\beta\uparrow \beta_c$ (see e.g. \cite[Remark 2.2]{alaoui2023shattering}), from which it easily follows that \eqref{eq:transport_disorder_chaos} remains uniformly positive in the limit $\lim_{\beta\uparrow\beta_c}\lim_{n\to\infty}$. 
\end{proof}

\section*{Acknowledgments}

AM was supported by the NSF through award DMS-2031883, the Simons Foundation through
Award 814639 for the Collaboration on the Theoretical Foundations of Deep Learning, the NSF
grant CCF-2006489 and the ONR grant N00014-18-1-2729. Part of this work was carried out while
Andrea Montanari was on partial leave from Stanford and a Chief Scientist at Ndata Inc dba
Project N. The present research is unrelated to AM’s activity while on leave.

\bibliographystyle{amsalpha}

\newcommand{\etalchar}[1]{$^{#1}$}
\providecommand{\bysame}{\leavevmode\hbox to3em{\hrulefill}\thinspace}
\providecommand{\MR}{\relax\ifhmode\unskip\space\fi MR }
\providecommand{\MRhref}[2]{%
  \href{http://www.ams.org/mathscinet-getitem?mr=#1}{#2}
}
\providecommand{\href}[2]{#2}

 \newpage

 \appendix

\section{Replica symmetry breaking and dynamical phase transitions}
\label{app:prelim}

\subsection{Parisi formula and the RSB phase transition}

With $\xi$ fixed throughout, define the quenched free energy density of the mixed $p$-spin model as
\begin{equation} 
\label{eq:free_energy}
  F_n(\beta) = \frac{1}{n} \E\, \log \Big(\frac{1}{2^n}\sum_{\bx \in \{-1,+1\}^n} e^{\beta H_n(\bx)}
  \Big) \, .
\end{equation}
The limit of $F_n(\beta)$ for large $n$ is known to exist for all $\beta>0$ 
and its value is given by the Parisi formula \cite{talagrand2006parisi}:
\begin{equation} \label{eq:parisi}
\lim_{n\to\infty} F_n(\beta) = \inf_{\zeta \in \cuP([0,1])} \Par_{\beta}(\zeta) \, .
\end{equation}
Here $\cuP([0,1])$ denotes the set of Borel probability measures supported on $[0,1]$, and
 $\Par_{\beta}=\Par_{\xi,\beta}$ is the Parisi functional at inverse temperature $\beta$
  defined as follows. For $\zeta\in \cuP([0,1])$, define $\Phi_{\zeta}:[0,1]\times \bbR\to\bbR$ 
  to be the solution of the following `Parisi PDE'
  (with an abuse of notation, we write $\zeta(t)$ for the distribution function $\zeta([0,t])$)
\begin{align}
    \label{eq:intro-1dParisiPDEdefn}
    \partial_t \Phi_{\zeta}(t,x)+\frac{\beta^2}{2}\xi''(t)\left(\partial_{xx}\Phi_{\zeta}(t,x)+\zeta(t)(\partial_x \Phi_{\zeta}(t,x))^2\right)=0 &\, ,\\
    \Phi_{\zeta}(1,x)=\log\cosh(x) &\, .
\end{align}

Existence and uniqueness properties for this PDE are established in \cite{auffinger2015parisi,jagannath2016dynamic}.
The Parisi functional $\Par_{\beta}: \cuP([0,1]) \to \bbR$ is given by
\begin{equation}
    \label{eq:def-parisi-functional-is}
    \Par_{\beta}(\zeta) 
    = 
    \Phi_{\zeta}(0,0) - \frac{\beta^2}{2}\int_{0}^1 t\xi''(t)\zeta(t) \rmd t\, .
\end{equation}
The SDE 
%
\begin{equation}
\label{eq:1dParisiSDE}
  \de X_t=\beta^2\xi''(t)\zeta(t)\partial_x \Phi_{\zeta}(t,X_t)\de t+\sqrt{\beta^2\xi''(t)}\de B_t\, , \quad X_0=0\, .
\end{equation}
is intimately connected to the PDE \eqref{eq:intro-1dParisiPDEdefn}.

The following properties are known:
\begin{enumerate}
    \item A unique minimizer $\zeta_{\beta}^*\in\cuP([0,1])$ of $\Par_{\beta}$ exists for all $\beta$~\cite{auffinger2015parisi}.
    \item 
    \label{it:lowtemp}
    There exists a critical inverse temperature $\beta_c\in(0,\infty)$ such that $\beta>\beta_c$ if and only if $\zeta_{\beta}^* \neq \delta_0$, i.e., $\zeta_{\beta}^*$ is not an atom at $0$. 
    This follows from e.g. \cite[Footnote 2]{subag2021tap}.
\end{enumerate}

In the replica-symmetric case $\zeta^*_{\beta}=\delta_0$, the Parisi solution simplifies. 
Then we have the following standard result (which can be proved simply by substituting the claimed solutions).
\begin{proposition}
\label{prop:Parisi-RS-simple}
Assume   $\zeta=\delta_0$. Then the unique solution \eqref{eq:intro-1dParisiPDEdefn} is given by
\begin{align}
 \Phi_{\zeta}(t,x)=\log\cosh(x)+\frac{\beta^2}{2}\big(\xi'(1)-\xi'(t)\big)\, .
 \end{align}
Further, the solution of the SDE \eqref{eq:1dParisiSDE} has distribution
  \begin{equation}
  X_t\ed  \gamma X+\sqrt{\gamma}Z\, ,
\end{equation}
  for $\gamma=\xi'(t)$, and  $X\in \{-1,+1\}$  uniform and independent of $Z\sim\normal(0,1)$. 
 Finally, 
$\partial_x\Phi_x(t,x)=\tanh(X_t)$.
\end{proposition}

Recall the definitions of $\psi$, Eq.~\eqref{eq:psi}:
\begin{equation}
  \psi(\gamma)=\bbE[\tanh(\gamma + \sqrt{\gamma}Z)] \, , ~~~~Z\sim\normal(0,1) \, ,
\end{equation}
and $\phi$ the inverse function of $\psi$. It is not hard to show that both are smooth and strictly increasing with $\psi(0)=\phi(0)=0$. Moreover, $\psi$ is concave and $\phi$ is convex.

\begin{proposition}
\label{prop:RS-condition}
  $\beta\leq \beta_c$ holds if and only if
  \begin{align}\label{eq:RS-Condition}
  \RS(t):= \int_0^t \xi''(s)\big(\psi(\beta^2\xi'(s)) - s\big) \rmd s\leq 0\quad\forall t\in [0,1]\, .
  \end{align}
\end{proposition}

\begin{proof}
   By strict convexity of the Parisi functional \cite{auffinger2015parisi}, $\beta\leq \beta_c$ is equivalent to first-order optimality of $\delta_0$ in $\Par_{\beta}$. The result now follows from \cite[Theorem 2]{chen2017variational} and Proposition~\ref{prop:Parisi-RS-simple}.
\end{proof}

\begin{lemma}\label{lemma:Beta1Betac}
It holds that $ \xi''(1)^{-1/2}\leq \beta_1 \le \beta_c$.
\end{lemma}
\begin{proof}
Since $\xi''(q)\leq \xi''(1)$ and $\phi'(q)\geq \phi'(0)=1/\psi'(0)=1$, we obtain the lower bound $\xi''(1)^{-1/2}\leq \beta_1$.

To see that $\beta_1\leq \beta_c$ as defined in the previous subsection, recall from Eq.~\eqref{eq:beta0-equiv} that $\beta\leq \beta_1$ is equivalent to 
\[\beta^2\xi''(q)\leq \phi'(q) ~~~\forall q\in [0,1] \, .\]
Integrating the above yields $\beta^2 \xi'(q)\leq \phi(q)$. Therefore
\[\psi(\beta^2 \xi'(s))\leq \psi(\phi(s))\leq s\quad\forall s\in [0,1]\, ,\]
which immediately implies the condition of Proposition~\ref{prop:RS-condition}. 
\end{proof}

In the SK model, $\beta_1=\beta_c=  1/\sqrt{\xi''(0)}$ so the two thresholds are equal. 
In fact this remains true for small perturbations around the SK model.
\begin{proposition}
\label{prop:beta0=betac}
We always have $\beta_c\le 1/\sqrt{\xi''(0)}$. 

Further assume  $\xi''(0) =c_2^2/2> 0$ strictly  and $c_p \leq a\cdot c_2 $ for all $3\leq p\leq P$ and some small 
$a=a(P)>0$. Then 
  \[
    \beta_1(\xi)=\beta_c(\xi)= \frac{1}{\sqrt{\xi''(0)}}\, .
  \]
\end{proposition}
\begin{proof}
In order to prove $\beta_c\le 1/\sqrt{\xi''(0)}$,  recall the definition
of $\RS(t)$ given in Eq.~\eqref{eq:RS-Condition}. Using $\psi(t) = t+o(t)$ for $t\to 0$,
we get 
\[
 \RS(t)= \frac{1}{2}\xi''(0)\big(\beta^2\xi''(0)-1\big) t^2+o(t^2)\, ,
 \]
 and hence  $\RS(t)>0$ for $t$ small enough unless  $\beta\le 1/\sqrt{\xi''(0)}$.

To prove $\beta_1(\xi)=\beta_c(\xi)= 1/\sqrt{\xi''(0)}$ when
$c_p \leq a\cdot c_2 $, $c_2>0$, we will assume, without loss of generality, $\xi''(0)=1$,
and therefore $\xi(t)=t^2/2 +\sum_{p=3}^Pc_p^2t^p$.
Let $\bc =  (c_3,\dots,c_P)$.

  First we show $\beta_1=1$.  Let $f(q)=f(q;\bc)=\phi'(q)/\xi''(q)$. Since $\phi'(0)=1$ and therefore 
  $f(0)=1$, it suffices to show that $f(q)\geq 1$ on $q\in (0,1)$.
  
  Consider the case $\bc=\bfzero$, $\xi(t)=t^2/2$. Then $f(q)= \phi'(q)$. Since $\gamma\mapsto \psi(\gamma)$
  is concave by Lemma \ref{lem:properties}.$(a)$, it follows thar $q\mapsto \phi(q)$
  is convex, and therefore $f(q)$ is non-decreasing, which proves our claim. 
  Further by a direct calculation $f'(0)=\phi''(0)>0$ in this case, and therefore there
  exists $q_0, A_1, A_2>0$ such that $f'(q;\bfzero)\ge A_1$  on $(0,q_0]$, $f(q;\bfzero)>A_2$ on $[q_0,1)$.
  
  Next consider the general case $\bc\neq \bfzero$.
Note that $f(q;\bc)/\phi(q)\to f(q;\bfzero)/\phi(q)$  uniformly on $[0,1]$
as $\bc\to \bfzero$, and therefore $f(q;\bc)>A_2/2$ on $[q_0,1)$ if  $\|\bc\|_{\infty}<a$
for a small enough $a$.
  Further compute the derivative with respect to $q$:
  \[
    f'(q;\bc)
    =
    \frac{\phi''(q)}{\xi''(q)}
    -
    \frac{\phi'(q)\xi'''(q)}{\xi''(q)^2}.
  \]
  Therefore $f'(q;\bc)\to f'(q;\bfzero)$  uniformly on $[0,q_0]$
as $\bc\to \bfzero$, whence  $f(q;\bc)>A_1/2$ on $[0,q_0]$ if  $\|\bc\|_{\infty}<a$
for small enough $a$. This completes the proof that $f(q;\bc)\ge 1$ on $[0,1]$
and therefore $\beta_1(\xi)=1$.

  Since we showed $\beta_1(\xi)\leq \beta_c(\xi)$ in Lemma \ref{lemma:Beta1Betac}, 
  and we proved already  that $\beta_c(\xi)\le 1/\sqrt{\xi''(0)}=1$, the claim 
  follows.
\end{proof}

\subsection{The dynamical phase transition}

In this appendix we report, for the reader's convenience, the statistical physics
prediction for the location of the dynamical phase transition,
and derive its large-$p$ asymptotics for the `pure' model $\xi(t) = t^p$.
We refer to \cite{montanari2003nature,ferrari2012two} for pointers to the physics literature
and to \cite{montanari2022short} for further background.
Define the `effective potential' $\Sigma_{\beta}:[0,1]\to \R$ by
\begin{align}
\Sigma_{\beta}(q) &:= -\frac{\beta^2}{2}\big[\xi(q)-q\xi'(q)\big]- 
\E\big\{\cosh(\lambda G)\log \cosh(\lambda G)\big\}\, ,
\;\;\;\;\;\lambda := \beta\sqrt{\xi'(q)}\, .
\end{align}
(Here expectation is taken with respect to $G\sim\normal(0,1)$.)
The point $q=0$ is always stationary for $\Sigma$. The dynamical phase transition 
is predicted to be located at the infimum $\beta$ such that a second
stationary point appears:
\begin{align}
\beta_{\sdyn}(\xi)= \inf\Big\{\beta>0: \; \exists q\in (0,1] \mbox{ such that }
\Sigma'_{\beta}(q)=0\Big\}\, .
\end{align}
By differentiating $\Sigma$, $\beta_{\sdyn}(\xi)$ is equivalently given by the infimum 
of all $\beta$'s such that the following equations admit a solution
$q\in (0,1]$:
\begin{align}
q &= \frac{\E\{\cosh (\lambda G)\tanh( \lambda G)^2 \}}{\E\{\cosh ( \lambda G) \}}
=:H(\lambda)\, ,\label{eq:Dyn1}\\
\lambda & = \beta\sqrt{\xi'(q)}\, .\label{eq:Dyn2}
\end{align}
We next consider the special case $\xi(t) = t^p$ and derive the large $p$
asymptotics. To this end, we note that, as $\lambda\to \infty$
\begin{align}
H(\lambda)&= 1- \frac{\E\{\cosh ( \lambda G)^{-1} \}}{\E\{\cosh ( \lambda G) \}}\nonumber\\
& = 1- e^{-\lambda^2/2}\E\{\cosh ( \lambda G)^{-1} \}\nonumber\\
& = 1- \frac{K}{\lambda}e^{-\lambda^2/2}\big(1+O(\lambda^{-1})\big)\, ,\label{eq:AsympH}
\end{align}
where $K:= (2\pi)^{-1/2}\int (\cosh t)^{-1} \de t$. We next set $\lambda= \sqrt{2z\log p}$,
$\beta^2= (2\alpha\log p)/p$
and rewrite the above equations as
\begin{align}
z &= F(z;\alpha,p)\, ,\\
F(z;\alpha,p) &:= \alpha\, H(\sqrt{2z\log p})^{p-1}\, .
\end{align}
Using the asymptotics \eqref{eq:AsympH}, we get, as $p\to\infty$ and $z>0$ fixed:
\begin{align*}
F(z;\alpha,p) &= \alpha\, \Big(1- \frac{Kp^{-z}}{\sqrt{2 z\log p}}\big(1+o(1)\big) \Big)^{p-1}\\
& = \alpha\, \exp\Big\{- \frac{Kp^{1-z}}{\sqrt{2 z\log p}}\big(1+o(1)\big)\Big\}\, .
\end{align*}
Using this expression (and a similar one for $F'(z;\alpha,p)$,
it is easy to check that
\begin{enumerate}
\item[$(a)$] $z\mapsto F(z;\alpha,p)$ is non-decreasing, with $F(0;\alpha,p)=0$.
\item[$(b)$] For $z\in [0,1)$ fixed, we have 
$F(z;\alpha,p),F'(z;\alpha,p)\to 0$ as $p\to\infty$.
\item[$(c)$] For $z\in (1,\infty)$ fixed, we have 
$F(z;\alpha,p)\to \alpha$ as $p\to\infty$.
\end{enumerate}
From these it follows that:
$(i)$~If $\alpha<1$, Eqs.~\eqref{eq:Dyn2}, \eqref{eq:Dyn2} admit the unique solution $q=0$;
$(ii)$~If $\alpha>1$, Eqs.~\eqref{eq:Dyn2}, \eqref{eq:Dyn2} admit a solution $q\in (0,1)$.
This in turns yields
\begin{align}
\beta_{\sdyn}(\xi_p) = \sqrt{\frac{2\log p}{p}}\cdot\big(1+o_p(1)\big)\, ,
\end{align}
as claimed in the main text.

 \section{Convergence analysis of Natural Gradient Descent}

 \label{app:NGD}

The main objective of this appendix is to prove Lemma \ref{lem:local-landscape},
which we will do in Section \ref{app:local-landscape}, after some technical preparations in Section \ref{app:Preliminaries}.

\subsection{Technical preliminaries}
\label{app:Preliminaries}

\begin{definition}
\label{defn:c-strongly-convex}
Let $Q\subseteq (-1,1)^n $ be a convex set. We say that a twice differentiable function 
$\F:Q\to \mathbb R$ is relatively 
$c$-strongly convex if it satisfies 
\begin{equation}
\label{eq:c-strongly-convex}
    \nabla^2 \F(\m) \succeq  c \Dm \;\;\;\forall\;\m\in Q \, .
\end{equation}
We say it is relatively $C$-smooth if 
it satisfies 
\begin{equation}
\label{eq:C-smooth}
    \nabla^2 \F(\m) \preceq   C\Dm \;\;\;\forall\;\m\in Q \, .
\end{equation}
\end{definition}

As $\Dm=\nabla^2 (-\bh(\m))\succeq \bI_n$, it follows that \eqref{eq:c-strongly-convex} 
implies ordinary $c$-strong convexity in Euclidean norm. The next proposition connects 
relative strong convexity with the Bregman divergence introduced in Eq.~\eqref{eq:bregman}.
\begin{proposition}[Proposition 1.1 in \cite{lu2018relatively}]
A twice differentiable function 
$\F:Q\to \mathbb R$ is relatively 
$c$-strongly convex if  and only if
\begin{equation}
\label{eq:alt-strongly-convex}
    \F(\m)\geq \F(\n)+\langle \nabla\F(\n),\m-\n\rangle + c \Dh(\m,\n),\quad\quad \forall\m,\n\in Q \, .
\end{equation}

\end{proposition}

\begin{lemma}
\label{lem:Bregman}
For $\m,\n\in (-1,1)^n$, 
\begin{align}
    \label{eq:Bregman-bigger}
    \Dh(\m,\n) &\geq \frac{\|\m-\n\|_2^2}{2} \, ,\\
    \label{eq:Bregman-bound}
    \Dh(\m,\n)&\leq 10 n\left(1+\frac{\|\atanh(\n)\|_2}{\sqrt{n}}\right)\, ,\\
\label{eq:Bregman-Lip}
    \Dh(\m,\n)&\leq \|\atanh(\m)-\atanh(\n)\|_2^2 \, .
\end{align}

\end{lemma}

\begin{proof}

Observe that $h''(x)=-1/(1-x^2)\leq -1$ for all $x\in (-1,1)$ with equality if and only if $x=0$. 
Therefore
\begin{align*}
    \Dh(\m,\n)&=
    \sum_{i=1}^n \int_{m_i}^{n_i}  (x-m_i) (-h''(x))\, \rmd x \\
    & = \sum_{i=1}^n \frac{(n_i-m_i)^2}{2} \, .
\end{align*}
This proves Eq.~\eqref{eq:Bregman-bigger}.

 Next, Eq.~\eqref{eq:Bregman-bound} follows from Eq.~\eqref{eq:bregman} and the fact that the binary 
 entropy $h:\mathbb R\to\mathbb R$ is uniformly bounded. 
 
 Finally Eq.~\eqref{eq:Bregman-Lip} follows from
\begin{align*}
    \Dh(\m,\n)&\le \<\nabla h(\n)-\nabla h(\m), \m-\n\>\\
    & = \big\langle \atanh(\m)-\atanh(\n),\m-\n\big\rangle \\
    & \leq \big\|\atanh(\m)-\atanh(\n)\big\|_2^2\, .
\end{align*}
Here in the last step we used that $\tanh(\cdot)$ is $1$-Lipschitz.
\end{proof}

\begin{lemma}
If $\F:Q\to \mathbb R$ is relatively $c$-strongly convex for some convex set $Q\subseteq (-1,1)^n$,
 and $\nabla \mathcal F(\m_*)=0$ for $\m_*\in Q$, it follows that 
\[
    \F(\m)-\F(\m_*)\geq \frac{c\|\m-\m_*\|_2^2}{2} \, .
\]
for all $\m\in Q$.
\end{lemma}

\begin{proof}

Using \eqref{eq:alt-strongly-convex} and \eqref{eq:Bregman-bigger}, and observing that $\nabla \mathcal F(\m_*)=0$, we obtain
\[
    \frac{\mathcal F(\m)-\mathcal F(\m_*)}{\|\m-\m_*\|_2^2}
    \geq
    \frac{\mathcal F(\m)-\mathcal F(\m_*)}{2\cdot\Dh(\m,\m_*)}
    \geq 
    \frac{c}{2} \, .
\]
\end{proof}

\begin{lemma}
\label{lem:Boundary}
Suppose $\F:Q_*\to \mathbb R$ is $c$-strongly convex in the convex set $Q_* :=B(\m_*,\rho)\cap (-1,1)^n$.
If $\bx_*\in \partial Q_*$, $x_{*,k} =+1$ (respectively, $x_{*,k}=-1$) and
 $|x_j|<1$ for all $j\in[n]\setminus\{k\}$,
then $\lim_{t\to 0+}\partial_{x_k}\F(\bx_*-t\bfe_k) = +\infty$
(respectively $\lim_{t\to 0+}\partial_{x_k}\F(\bx_*+t\bfe_k) = -\infty$.)
\end{lemma}
\begin{proof}
Consider the case $x_k=+1$ (as the case $x_k=-1$ follows by symmetry.)
Then there exists $t_0>0$ such that $\bx_*-t\bfe_k\in Q_*$ for all $t\in (0,t_0]$.
 Let $\bx(s) := \bx_*-(t_0-s)\bfe_k$, $s\in [0,t_0)$.
 Then 
 \begin{align*}
 \partial_{x_k}F(\bx(s))&= \partial_{x_k}F(\bx(0))+\int_0^s \partial_{x_k}^2F(\bx(u))\, \de u\\
 & = \partial_{x_k}F(\bx(0))+\int_0^s \<\bfe_k,\nabla^2F(\bx(u))\bfe_k\>\, \de u\\
 & \ge \partial_{x_k}F(\bx(0))+c\int_0^s (1-x_k(u)^2)^{-1}\, \de u\\
 & \ge \partial_{x_k}F(\bx(0))+c \int_0^s (1-(1-t_0+u)^2)^{-1}\, \de u, .
 \end{align*}
 The last integral diverges as $s\uparrow t_0$, thus proving the claim.
\end{proof}

\begin{lemma}
\label{lem:approx-stationary-pt-to-minimizer}
Suppose $\F:Q\to \mathbb R$ is $c$-strongly convex for a convex set $Q\subseteq (-1,1)^n$. Moreover suppose that 
\[
    \|\nabla \mathcal F(\m)\|\leq c\sqrt{\eps n}
\]
for some $\m\in Q$ with 
\[
    B\left(\m,2\sqrt{\eps n}\right)\cap (-1,1)^n\subseteq Q \, .
\]
Then there exists a unique
$\m_*\in \Ball^n\left(\m,2\sqrt{\eps n}\right)\cap (-1,1)^n$ satisfying $\nabla \F(\m_*)=0$, which 
is in fact a global minimizer of $\F$ on $Q$. Moreover
\begin{equation}
\label{eq:approx-stationary-approx-opt}
    \F(\m)-\F(\m_*)\leq 2c \eps n \, .
\end{equation}

\end{lemma}

\begin{proof}
Let $Q_\le :=\{\bx\in Q: \F(\bx)\le \F(\m) \}$. Then, for any $\bx\in Q_0$, we have
\begin{align*}
0&\ge \F(\bx) -\F(\m)\\
& \ge -c\sqrt{\eps n}\|\bx-\m\|_2+c  \Dh(\bx;\m)\\
& \ge  -c\sqrt{\eps n}\|\bx-\m\|_2+\frac{c}{2}\|\bx-\m\|_2^2\, .
\end{align*}
Hence $Q_\le \subseteq Q_* := B\left(\m,\sqrt{\eps n}\right)\cap (-1,1)^n$, $Q_*\subseteq Q$.
By  continuity three cases are possible: $(i)$~The minimum of $\F$ is
achieved in the interior of $Q_{\le}$; $(ii)$~The minimum is achieved along a sequence
$(\bx_i)_{i\ge 0}$, $\|\bx_i\|_{\infty}\to 1$; $(iii)$ the minimum is achieved at $\m_*\neq \m$
such that $\F(\m_*) = \F(\m)$. Case $(iii)$ cannot hold by strong convexity,
and case $(ii)$ cannot hold by Lemma \ref{lem:Boundary}.

Uniqueness of $\m_*$ follows by strong convexity, and 
$\nabla \F(\m_*)=0$ by differentiability.
Finally
\[
    \F(\m)-\F(\m_*)\leq \|\nabla \F(\m)\|\cdot \|\m-\m_*\|\leq 2c\eps n \, .
\]

\end{proof}

\begin{lemma}

Suppose $\F:Q\to \mathbb R$ is relatively $c$-strongly convex. Let $\m_*$ be a local minimum of $\F$ belonging to the interior of $Q$, and suppose that $B\left(\m_*,2\sqrt{\eps n}\right)\cap (-1,1)^n\subseteq Q$. Consider for $\by\in\mathbb R^n$ the function
\[
    \F_{\by}(\m)=\F(\m)-\langle\by,\m\rangle.
\] 
Then $\F_{\by}$ is relatively $c$-strongly convex on $Q$ for any $\by\in \mathbb R^n$. 
If $\|\by\|\leq (c/2)\sqrt{\eps n}$, then $\F_{\by}$ has a unique stationary point and minimizer $\m_*(\by) \in Q$. Moreover if $\|\by\|,\|\hby\|\leq \frac{c\sqrt{\eps n}}{2}$ then
\begin{equation}
\label{eq:local-lip-minimizer}
    \|\m_*(\by)-\m_*(\hby)\| \leq \frac{\|\by-\hby\|}{c}.
\end{equation}
\end{lemma}

\begin{proof}

The relative $c$-strong convexity of $\F_{\by}$ is clear as the Hessian of $\F_{\by}$ 
does not depend on $\by$. For $\|\by\|\leq (c/2)\sqrt{\eps n}$, because
\[
    \|\nabla \F_{\by}(\m_*)\|= \|\by\| \le \frac{c\sqrt{\eps n}}{2} \quad\text{ and }\quad B\left(\m_*,\sqrt{\eps n}\right)\cap (-1,1)^n\subseteq Q \, ,
\]
Lemma~\ref{lem:approx-stationary-pt-to-minimizer} implies the existence of a unique minimizer
\[
    \m_*(\by)\in B\left(\m_*,\sqrt{\eps n}\right)\cap (-1,1)^n\subseteq Q \, .
\]
If $\|\hby\|\leq (c/2)\sqrt{\eps n}$ also holds, $\F_{\hby}$ is $c$-strongly convex on 
\[
    B\left(\m_*(\hby),\sqrt{\eps n}\right)\cap (-1,1)^n\subseteq \Ball^n\left(\m_*,2\sqrt{\eps n}\right)\cap (-1,1)^n \subseteq Q.
\]
Moreover since $\|\by-\hby\|\leq c\sqrt{\eps n}$, we obtain
\begin{align*}
    \|\nabla \F_{\hby}(\m_*(\by))\| 
    & = 
    \|\by-\hby\|
    = 
    c\sqrt{\eps' n} \, ,
\end{align*}
for $\eps' = \frac{\|\by-\hby\|^2}{c^2 n}\leq \eps$. Therefore the conditions of Lemma~\ref{lem:approx-stationary-pt-to-minimizer} are satisfied with $(\F_{\hby},\m_*(\by),\eps')$ in place of $(\F,\m,\eps)$. Equation~\eqref{eq:local-lip-minimizer} now follows since 
\[
    \|\m_*(\by)-\m_*(\hby)\|\leq \sqrt{\eps' n} = \frac{\|\by-\hby\|}{c} \, .
\]
\end{proof}

We now analyze the convergence of Algorithm~\ref{alg:NGD} from a good initialization.
\begin{lemma}
\label{lem:NGD-convergence-general}
Assume $\F$ has a local minimum at $\m_*$, it
is relatively $c$-strongly-convex on $B(\m_*,\sqrt{\eps n})\cap (-1,1)^n$ and $C$-relatively
 smooth on $(-1,1)^n$, and  
$\max_{\m\in\Ball^n(\sqrt{n})}\|\nabla\F(\m)+\nabla \bh(m)\|_2\le C\sqrt{n}$, for a constant $C$.
  Suppose 
\begin{align}
\label{eq:NGD-Tech}
    \hm^0\in B\left(\m_*,\sqrt{\eps n}\right)\cap (-1,1)^n
\end{align}
satisfies
\begin{equation}
\label{eq:good-init-NGD}
\F(\hm^0)<\F(\m_*)+\frac{c\eps n}{8}.
\end{equation}
Then there exist constants $\eta_0,C'>0$ depending only on $(C,c,\eps)$ such that the following holds.
 If Algorithm~\ref{alg:NGD} is initialized at $\hm^0$ with learning rate $\eta=1/L\in (0,\eta_0)$, 
 then, for every $K\geq 1$ 
\begin{align}
\label{eq:NGD-good-value}
    \F(\hm^K) 
    & \leq 
    \F(\m_*)+C'n\left(1+\frac{\|\atanh(\hm^0)\|_2}{\sqrt{n}}\right)(1-c\eta)^K,\\
\label{eq:NGD-good-approx}
    \|\hm^K-\m_*\|_2 
    & \leq 
    C'\sqrt{n}\left(1+\frac{\|\atanh(\hm^0)\|_2}{\sqrt{n}}\right)(1-c\eta)^{K/2}.
\end{align}
\end{lemma}

\begin{proof}
Recall Eq.~\eqref{eq:NGD}, which we copy here for the reader's convenience:
\begin{equation}
\label{eq:NGD-App}
    \hm^{i+1}=\argmin_{\bx\in (-1,1)^n} \big\langle \nabla \F(\hm^i),\bx-\hm^{i} \big\rangle + L\cdot \Dh(\bx,\hm^{i}).
\end{equation} 

If $\eta_0\leq \frac{1}{2C}$ then \cite[Lemma 3.1]{lu2018relatively} applied to the linear
 (hence convex) function $\langle \nabla\F(\hm^i),\,\cdot\,\rangle$ states that for all
  $\m\in (-1,1)^n$,
\begin{equation}
\label{eq:3-point}
    \langle \nabla\F(\hm^i),\hm^{i+1}\rangle + L\Dh(\hm^{i+1},\hm^i)+L\Dh(\m,\hm^{i+1})\leq \langle  \nabla \F(\hm^i),\m\rangle + L\Dh(\m,\hm^i).
\end{equation}
Moreover the global relative smoothness shown in \eqref{eq:global-smooth} implies that for $\m,\m'\in (-1,1)^n$,
\begin{equation}
\label{eq:global-smooth-bregman}
\F(\m)\leq \F(\m')+\langle\nabla\F(\m'),\m-\m'\rangle+ C\cdot\Dh(\m,\m').
\end{equation}
Combining  Eqs.~\eqref{eq:3-point} and \eqref{eq:global-smooth-bregman} yields
\begin{equation}
\begin{aligned}
\label{eq:weird-ineq}
\F(\hm^{i+1}) & \leq \F(\hm^i)+\langle \nabla \F(\hm^i),\hm^{i+1}-\hm^i\rangle + L\Dh(\hm^{i+1},\hm^i)\\
&\leq \F(\hm^i)+\langle \nabla \F(\hm^i),\m-\hm^i\rangle + L\Dh(\m,\hm^i)-L\Dh(\m,\hm^{i+1}).
\end{aligned}
\end{equation}
Setting $\m=\hm^i$, we find
\[
    \F(\hm^{i+1})\leq \F(\hm^i),\quad \forall ~i\in [K].
\] 
We next prove by induction that for each $i\geq 1$, 
\begin{equation}
\label{eq:continue-good-NGD}
\F(\hm^i)<\F(\m_*)+\frac{c\eps n}{8},
\quad\quad 
\|\hm^i-\m_*\|<\sqrt{\eps n}.
\end{equation}
The base case $i=0$ holds by assumption. Suppose \eqref{eq:continue-good-NGD} holds for $i$. It follows that
\[
    \F(\hm^{i+1})\leq \F(\hm^i)\leq \F(\m_*)+\frac{c\eps n}{8}.
\] 
In fact, local $c$-strong convexity
\[
    \nabla^2 \F(\m) \succeq c\Dm \succeq c\bI_n,\quad \m\in B(\m_*,\sqrt{\eps n})\cap (-1,1)^n
\]
implies $\hm^i$ is even closer to $\m_*$ than required by \eqref{eq:continue-good-NGD}:
\[
    \|\hm^i - \m_*\|_2 \leq \sqrt{\frac{\F(\hm^i)-\F(\m_*)}{c}}\leq \frac{\sqrt{\eps n}}{2}.
\]
Next we bound the movement from a single NGD step. Comparing values of \eqref{eq:NGD-App} at $\hm^i$ and the minimizer $\hm^{i+1}$ implies
\begin{equation}
\label{eq:another-ineq}
    \langle \nabla \F(\hm^i),\hm^{i+1}-\hm^i\rangle + L\Dh(\hm^{i+1},\hm^i)\leq 0.
\end{equation}
From definition of Bregman divergence and the fact that 
$\max_{\m\in\Ball^n(\sqrt{n})}\|\nabla\F(\m)+\nabla \bh(m)\|_2\le C\sqrt{n}$,
\begin{align*}
    |\langle \nabla \F(\hm^i),\hm^{i+1}-\hm^i\rangle + \Dh(\hm^{i+1},\hm^i)| 
    & = 
    |\langle \nabla \F(\hm^i)+\nabla \bh(\hm^i),\hm^{i+1}-\hm^i\rangle - \bh(\hm^{i+1})+\bh(\hm^i)|\\
    & \leq C_1 n\left(1+\frac{\|\hm^{i+1}-\hm^i\|}{\sqrt{n}}\right).  
\end{align*}
Moreover assuming $L>1$, \eqref{eq:Bregman-bigger} implies
\[
    (L-1) \Dh(\hm^{i+1},\hm^i)\geq \frac{L-1}{2} \|\hm^{i+1}-\hm^i\|^2.
\]
Substituting the previous two displays into \eqref{eq:another-ineq} yields
\[
    0\geq \frac{L-1}{2} \|\hm^{i+1}-\hm^i\|^2 - C_2\sqrt{n} \|\hm^{i+1}-\hm^i\|_2 - C_2n
\]
and so
\[
    \|\hm^{i+1}-\hm^i\|_2\leq \frac{C_3\sqrt{n}}{\sqrt{L-1}}.
\]
Taking $L$ large enough, it follows that 
\[
    \|\hm^{i+1}-\m_*\|\leq \|\hm^{i+1}-\hm^i\|_2 + \|\hm^{i}-\m_*\|_2\leq \sqrt{\eps n}.
\]
This completes the inductive proof of Eq.~\eqref{eq:continue-good-NGD}, which we now use to 
analyze convergence of Algorithm~\ref{alg:NGD}. Indeed from the first part of \eqref{eq:continue-good-NGD}, the local relative strong convexity of $\F$ implies
\[
    \F(\hm^i)+\langle \nabla\F(\hm^i),\m_*-\hm^i\rangle\leq \F(\m_*)-c\Dh(\m_*,\hm^i),\quad\quad \forall~i\in [K].
\]
Setting $\m=\m_*$ in \eqref{eq:weird-ineq} and combining yields
\[
    \F(\hm^{i+1})\leq \F(\m_*)+(L-c)\Dh(\m_*,\hm^i)-L\Dh(\m_*,\hm^{i+1}).
\]
Multiplying by $\left(\frac{L}{L-c}\right)^{i+1}$ and summing over $i$ gives 
\[
    \sum_{i=0}^{K-1} \left(\frac{L}{L-c}\right)^{i+1}\F(\hm^{i+1})\leq \sum_{i=0}^{K-1} \left(\frac{L}{L-c}\right)^{i+1} \F(\m_*)+L\Dh(\m_*,\hm^0).
\]
Since the values $\F(\hm^i)$ are decreasing, we find
\begin{align*}
    \F(\hm^K) 
    & \leq 
    \F(\m_*)+L\left(\sum_{i=0}^{K-1}\left(\frac{L}{L-c}\right)^{i+1}\right)^{-1} \Dh(\m_*,\hm^0)\\
    & \leq 
    \F(\m_*)+L\left(1-c\eta\right)^{K} \Dh(\m_*,\hm^0).
\end{align*}
Using Eq.~\eqref{eq:Bregman-bound} together with the last display
proves Eq.~\eqref{eq:NGD-good-value}.

 It was shown above by induction that $\hm^K$ is in a $c$-strongly convex neighborhood of $\m_*$. 
 Using strong convexity in Euclidean norm yields
\[
    \|\hm^k-\m_*\|\leq \sqrt{\frac{\F(\hm^K)-\F(\m_*)}{c}}
\]
and so \eqref{eq:NGD-good-approx} follows as well. 
\end{proof}

\begin{lemma}
There exist constants $C_0$, $C_1$ depending uniquely on $\beta$, $L$, $\xi$
such that the following holds with probability at least $1-C_0\exp(-n/C_0)$.
 For any $\m_1,\m_2\in (-1,1)^n$, and $\by_1,\by_2\in\mathbb R^n$, 
and $q\in [0,1]$
\begin{align}
    \label{eq:grad-F-atanh}
    &\|\nabla \hcuF_{\sTAP}(\m_1;\by_1,q)-\nabla \hcuF_{\sTAP}(\m_2;\by_2,q)\| 
    \leq 
    C_1\|\atanh(\m_1)-\atanh(\m_2)\| + \|\by_1-\by_2\|,\\
 \label{eq:grad-F-Bis}
     &\max_{\m\in\Ball^n(\sqrt{n})}\|\nabla\hcuF_{\sTAP}(\m_1;\by_1,q)+\nabla \bh(\m_1)\|_2\le C\sqrt{n}\, .
\end{align}
\end{lemma}


\begin{proof}
Lemma \ref{lemma:LipGrad} yields that $\m\mapsto \nabla H(\m)$ is Lipschitz continuous.
The claim follows using 
\eqref{eq:grad-F} and the fact that 
 $\m\mapsto \tanh(\m)$ is $1$-Lipschitz. 
 
Equation \eqref{eq:grad-F-Bis} also follows from the Lipschitz property 
of $\nabla H(\m)$, implying that $\|\nabla H(\m)\|/\sqrt{n}$ is bounded.
\end{proof}

\subsection{Proof of Lemma~\ref{lem:local-landscape}}
\label{app:local-landscape}

We split the proof into four parts.

\begin{proof}[Proof of Lemma~\ref{lem:local-landscape}, Part~\ref{it:landscape-basic}]
Fix $c = (1/4)-(\beta/2)>0$.
Lemma~\ref{lem:TAP-stationary} implies that for $K_{\sAMP}=K_{\sAMP}(\beta,\T,\eps)$ sufficiently 
large, we have with probability $1-o_n(1)$ 
\begin{align}
    \|\nabla \hcuF_{\sTAP}(\hm^{\sAMP};\by,q_*)\|\leq \frac{c\sqrt{\eps t n}}{4}\, ,
    \label{eq:again-small-grad}\\
    \hm^{\sAMP}:=\AMP(\bG,\by(t);K_{\sAMP}),\;\; q_*:=q_*(\beta,t)\, .\nonumber
\end{align}
Therefore, if $\|\by(t)-\hby\|\leq (c\sqrt{\eps tn})/4$ then
\begin{align*}
    \|\nabla \hcuF_{\sTAP}(\hm^{\sAMP};\hby,q_*)\|\leq 
    \|\nabla \hcuF_{\sTAP}(\hm^{\sAMP};\by(t),q_*)\|+\|\by-\hby\|\leq \frac{c}{2}\sqrt{\eps t}\, 
\end{align*}
Moreover Lemma~\ref{lem:local-convex} implies that there exist $\eps_0$, $c>0$
such that for all $\eps\in (0,\eps_0)$, 
\[
    \nabla^2 \hcuF_{\sTAP}(\m;\hby,q_*)
    =
    \nabla^2 \hcuF_{\sTAP}(\m;\by(t),q_*)
    \succeq 
    c\Dm,\quad\quad \forall~\m\in B\left(\hm^{\sAMP},\sqrt{\eps tn}\right)\cap (-1,1)^n.
\]
Using $\eps t/4$ in place of $\eps$ in Lemma~\ref{lem:approx-stationary-pt-to-minimizer}, 
it follows that there exists a local minimum
\[
    \m_*(\bG,\hby;q_*)\in \Ball^n\left(\hm^{\sAMP},\frac{\sqrt{\eps tn}}{2}\right)\cap (-1,1)^n
\]
of $\hcuF_{\sTAP}(\,\cdot\, ,\hby;q_*)$ which is 
also the unique stationary point in $\Ball^n\left(\hm^{\sAMP},(1/2)\sqrt{\eps tn}\right)\cap (-1,1)^n$.

We next claim that, for any $K>K_{\sAMP}$, with probability $1-o_n(1)$, this local minimum 
is also the unique stationary point in 
$\Ball^n\left(\AMP(\bG,\by(t);k),(1/2)\sqrt{\eps tn}\right)\cap (-1,1)^n$.
Indeed for $K_{\sAMP}$ sufficiently large (writing for simplicity $\by=\by(t)$):
\begin{align*}
    \plim_{n\to\infty}\sup_{k_1,k_2\in [K_{\sAMP},K]} \|\AMP(\bG,\by;k_1)-\AMP(\bG,\by;k_2)\|^2 
    &=
    \sup_{k_1,k_2\in [\kalg,K]} \plim_{n\to\infty} \|\AMP_{\beta}(\bG,\by;k_1)-\AMP_{\beta}(\bG,\by;k_2)\|^2\\
    & \leq n\cdot \sup_{k_1,k_2\geq K_{\sAMP}} |q_{k_1}(\beta,t) - q_{k_2}(\beta,t)|.
\end{align*}
From Eq.~\eqref{eq:uniform-gamma-limit}, by eventually increasing $K_{\sAMP}$, we have 
\[
    \sup_{k_1,k_2\geq K_{\sAMP}}|q_{k_1}(\beta,t) - q_{k_2}(\beta,t)|\leq \frac{\eps t}{16}.
\]
For such $K_{\sAMP}$, with probability $1-o_n(1)$, all $k\in [K_{\sAMP},K]$ satisfy
\begin{align*}
    \|\m_*(\bG,\by;q_{K_{\sAMP}}) - \AMP(\bG,\by;k)\|
    & \leq 
    \|\m_*(\bG,\by;q_{\sAMP}) - \AMP(\bG,\by;K_{\sAMP})\|
    \\
    &
    \quad\quad
    +
    \|\AMP(\bG,\by;k) - \AMP(\bG,\by;K_{\sAMP})\|\\
    & \leq
    \frac{\sqrt {\eps t n}}{2}+\sqrt{\frac{\eps t n}{4}}\\
    & \le \frac{3}{4}\sqrt{\eps t n}.
\end{align*}
Let 
\[
    S(k,\rho) : = \Ball^n\left(\AMP_{\beta}(\bG,\by;k),\rho\right)\cap (-1,1)^n\, ,\;\;\;
    \rho_{n,t}:=\sqrt{\eps nt}
\]
Recall that $\m_*(\bG,\by;q_*)$ is the unique stationary point of 
$\hcuF_{\sTAP}(\, \cdot\, ;\by,q_*)$  in  $S(K_{\sAMP},\rho_{n,t})$.
By the above, it is also a stationary point in  $S(k,\rho_{n,t})$, for $k\in [K_{\sAMP},K]$.
Repeating the same argument as before, there is only one stationary point inside 
$S(k,\rho_{n,t})$, hence this must coincide with  $\m_*(\bG,\by;q_*)$.
\end{proof}

\begin{proof}[Proof of Lemma~\ref{lem:local-landscape}, Part~\ref{it:landscape-stationary-point-good}]

Because $K_{\sAMP}$ is large depending on $\delta_0$, Lemma~\ref{lem:TAP-stationary} implies that with 
probability $1-o_n(1)$,
\[
    \|\nabla \hcuF_{\sTAP}(\AMP(\bG,\by;K_{\sAMP}),\by;q_*)\|\leq \frac{c\delta_0\sqrt{ t n}}{4}.
\]
Using $\frac{\delta_0\sqrt{t}}{4} $ in place of $\eps$ in 
Lemma~\ref{lem:approx-stationary-pt-to-minimizer}, it follows that the local minimizer
$\m_*(\bG,\by;q_*)$
of $\hcuF_{\sTAP}(\,\cdot\,;\by,q_*)$ satisfies
\[
    \|\AMP(\bG,\by;K_{\sAMP})-\m_*(\bG,\by;q_*)\|\leq \frac{\delta_0\sqrt{tn}}{2}. 
\]
Since $K$ is sufficiently large depending on $\delta_0$, Lemma 
implies that with probability $1-o_n(1)$,
\[
    \|\m(\bG,\by)-\AMP(\bG,\by;K_{\sAMP})\|\leq \frac{\delta_0\sqrt{tn}}{2}.
\]
Combining, we obtain that with probability $1-o_n(1)$,
\begin{align*}
    \|\m(\bG,\by)-\m_*(\bG,\by;q_*)\| 
    & \leq \|\m(\bG,\by)-\AMP(\bG,\by;K_{\sAMP})\| + \|\AMP(\bG,\by;K_{\sAMP})-
    \m_*(\bG,\by;q_*)\|
    \\
    &\leq \delta_0\sqrt{tn}.
\end{align*}
\end{proof}

\begin{proof}[Proof of Lemma~\ref{lem:local-landscape}, Part~\ref{it:landscape-lipschitz}]
The result is immediate from \eqref{eq:local-lip-minimizer}.
\end{proof}

\begin{proof}[Proof of Lemma~\ref{lem:local-landscape}, Part~\ref{it:landscape-NGD}]
We apply Lemma~\ref{lem:NGD-convergence-general} with $\F(\,\cdot\,)=\hcuF_{\sTAP}(\,\cdot\, ;\hby,q_*)$ 
and $\m_*=\m_*(\bG,\hby;q_*))$ (with $q_* = q_*(\beta,t)$).
We need to check that assumptions  \eqref{eq:NGD-Tech}, \eqref{eq:good-init-NGD} of
Lemma~\ref{lem:NGD-convergence-general} hold for $\hm^0=\tanh(\bu^0)$ with $\bu^0$ satisfying 
Eq.~\eqref{ass:landscape-NGD}. 

To check assumption  \eqref{eq:NGD-Tech}, we take $K_{\sAMP}$ sufficiently large
and $\delta_0$ sufficiently small, obtaining
\begin{align*}
    \|\hm^0-\m_*(\bG,\hby;q_*)\|
    &\leq 
    \|\hm^0-\AMP(\bG,\by;K_{\sAMP})\|
    + \|\AMP(\bG,\by;K_{\sAMP}) -\m(\bG,\by) \|\\
    &\quad\quad
    + \|\m(\bG,\by) - \m_*(\bG,\by;q_*)\|
    + \|\m_*(\bG,\by;q_*) - \m_*(\bG,\hby;q_*)\|\\
    &\stackrel{(a)}{\leq} \frac{c\sqrt{\eps tn}}{96(\beta^2+1)} + \frac{1}{100}\sqrt{\eps tn} + 
    \delta_0\sqrt{tn} + \frac{\|\by-\hby\|}{c}\\
    &\leq
    \frac{\sqrt{\eps tn}}{3}
\end{align*}
 where  inequality $(a)$ holds
 with probability $1-o_n(1)$. In the last step we used $c\leq 1$.

To check Eq.~\eqref{eq:good-init-NGD}, 
we use \eqref{eq:grad-F-atanh} we find 
that with probability $1-o_n(1)$,
\begin{align*}
    \|\nabla\hcuF_{\sTAP}(\hm^0;\hby,q_*)\|
    &\leq 
    \|\nabla\hcuF_{\sTAP}(\AMP(\bG,\by;K_{\sAMP});\by,q_*)\| + \|\by-\hby\|\\
    &\quad\quad+(4\beta^2+4)\|\atanh(\hm^0)-\atanh(\AMP(\bG,\hby;K_{\sAMP}))\| \\
    &\leq 
    \|\nabla\hcuF_{\sTAP}(\AMP(\bG,\by;K_{\sAMP});\by,q_*)\| 
    + \frac{c\sqrt{\eps t n}}{24}
    + \frac{c\sqrt{\eps t n}}{4}.
\end{align*}
Combining with Eq.~\eqref{eq:again-small-grad}, we find that with probability $1-o_n(1)$,
\[
     \|\nabla\hcuF_{\sTAP}(\hm^0;\hby,q_*)\|\leq \frac{c\sqrt{\eps t n}}{6}.
\]
Finally, we apply 
 Lemma~\ref{lem:approx-stationary-pt-to-minimizer} with $\frac{\eps t}{9}$ in place of
 $\eps$, to get
\[
    \hcuF_{\sTAP}(\hm^0;\hby,q_*)\leq \hcuF_{\sTAP}(\m_*(\bG,\hby;q_*);\hby,q_*) + \frac{nc\eps t}{9}.
\]

Lemma~\ref{lem:NGD-convergence-general} now applies for $\eta_0$ sufficiently small. 
Moreover, with probability $1-o_n(1)$ the initialization $\bx^0$ satisfies
\begin{align*}
    \|\atanh(\hm^0)\| 
    &\leq  
    \|\atanh(\hm^0)-\atanh(\AMP(\bG,\by;K_{\sAMP}))\|+\|\atanh(\AMP(\bG,\by;K_{\sAMP}))\|
    \\
    &\leq
    \frac{c\sqrt{\eps tn}}{96(\beta^2+1)}+  \sqrt{3(\gamma_*(\beta,t)+t)}\sqrt{n}\\
    &\leq
    C(\beta,c,\T)\sqrt{tn}.
\end{align*}
Thus, \eqref{eq:NGD-good-approx} implies \eqref{eq:landscape-NGD-convergence} for a sufficiently 
large number $K_{\sNGD}$ of natural gradient iterations.
\end{proof}

\section{Proof of Lemma \ref{lemma:HessianHamiltonian}}
\label{app:HessianHamiltonian}

We will separately consider the simpler case $\xi(t) = c_2^2t^2$, and the case of a general $\xi$.
Before proceeding, we note that it is immediate to compute the   distribution of
the Hessian at a point $\m\in\Ball^n(\sqrt{n})$. Letting $Q=Q(\m):= \|\m\|^2/n$, we have
\begin{align}
-\nabla^2 H(\m) &\stackrel{d}{=} \sqrt{\frac{\xi^{(4)}(Q)}{n^3}} \cdot g_0\cdot\m\m^{\sT}
+\sqrt{ \frac{\xi^{(3)}(Q)}{n^2}}\Big(\m\g^{\sT}+
\g\m^{\sT}\Big)+ \sqrt{\frac{\xi^{(2)}(Q)}{n}}
\bW\, ,\label{eq:HessianAtAPoint}
\end{align}
where $(g_0,\g,\bW)\sim \normal(0,1)\otimes\normal(\bfzero,\id_n)\otimes\GOE(n)$.

We also introduce the matrix of interest 
\begin{align}
\bX(\m;\Treg) := \D(\m)^{-1/2} &\big(-\nabla^2 H_n(\m) + \frac{\Treg}{n}\m\m^{\sT}\big)\D(\m)^{-1/2}\, .
\end{align}

\subsection{The case $\xi(t) = c_2^2t^2$}
By Eq.~\eqref{eq:HessianAtAPoint}, we have
\begin{align}
-\nabla^2 H(\m) &= \sqrt{\frac{\xi''(1)}{n}}
\bW\, ,
\end{align}
Recall that $\|\bW\|_{\op}\le (2+\delta)\sqrt{n}$ with probability at least $1-2\exp(-Cn\delta^2)$
\cite{Guionnet}. Using $\|\m\|_2\le \sqrt{n}$, 
we have with the same probability 
\begin{align}
-K_{\delta}(\xi)\,\id_n
\preceq -\nabla^2 H_n(\m) + \frac{\Treg}{n}\m\m^{\sT} \preceq \big(K_{\delta}(\xi)+\Treg\big)\id_n
\;\;\;\;\;\;\;\forall\m\in (-1,1)^n\, ,
\end{align}
and the claim follows since $\|\D(\m)^{-1/2}\|_{\op}\le 1$.

\subsection{The case of a general $\xi$}

 Throughout this section, we use the notation
$\xi^{(\ell)}(q)$ for $\ell$-th derivative of $\xi$, and define
\begin{align}
\oxi^{(\ell)}(q) = \Big(q\frac{\rmd \phantom{q}}{\rmd q}\Big)^{\ell}\xi(q) =\sum_{p=2}^Pc_p^2p^\ell
q^p\, .
\end{align}
We further denote by denote by $\hcuF$ the modified TAP free energy of Eq.~\eqref{eq:TAP_reg2}, dropping the subscript
for simplicity.
Recall that $\Ball^n(r)$ denotes the Euclidean
ball of radius $r$ in $\R^n$ centered at the origin.

Before proving the general case of Lemma \ref{lemma:HessianHamiltonian}, we establish a crude
bound on the third derivative of the Hamiltonian  $H$.
\begin{lemma}\label{lemma:RoughBoundThirdDerivative}
For $H$ defined as in Eq.~\eqref{eq:def-hamiltonian}, there exists a universal constant 
$C$, and a constant $C_0=C_0(\xi)$ such that, with probability at least 
$1-2\exp(-C_0n)$, we have
\begin{align}
\sup_{\bx\in\Ball^n(\sqrt{n})}\|\nabla^3H(\bx)\|_{\op}\le C\sqrt{
\frac{\oxi^{(8)}(1)}{n}}\, .
\end{align}
\end{lemma}
\begin{proof}
Let $\obG^{(p)}:= (p!)^{-1}\sum_{\pi\in \fS_p} (\bG^{(p)})^{\pi}$ where $\fS_p$ is the group of
permutation over $p$ objects and  $(\bG^{(p)})^{\pi}$ is the tensor obtained by permuting
 the indices of $\bG^{(p)}$. We then have
 \begin{align}
\<\nabla^3H(\bx),\bv^{\otimes 3}\> =\sum_{p\ge 2}\frac{c_p}{n^{(p-1)/2}}p(p-1)(p-2)\<\obG^{(p)},\bx^{\otimes (p-3)}
\otimes \bv^{\otimes 3}\>\,.
 \end{align}
 Therefore
 \begin{align}
\sup_{\bx\in\Ball^n(\sqrt{n})}\|\nabla^3H(\bx)\|_{\op} \le\frac{1}{n}
\sum_{p\ge 2}c_pp^3\|\bG^{(p)}\|_{\op}\, .
\end{align}
Using, e.g. \cite[Proposition A.1]{montanari2023solving} and Gaussian concentration, we get 
$\|\bG^{(p)}\|_{\op}\le C\sqrt{n\log p}$ for all $p\le P$ with probability at least 
$1-2\exp(-C_0n)$. Therefore, with the same probability
 \begin{align}
\sup_{\bx\in\Ball^n(\sqrt{n})}\|\nabla^3H(\bx)\|_{\op}& \le\frac{C}{\sqrt{n}}
\sum_{p\ge 2}c_pp^3\sqrt{\log p}\\
&\le\frac{C}{\sqrt{n}}
\left(\sum_{p\ge 2}c_p^2p^8\right)^{1/2}\left(\sum_{p\ge 2}\frac{\log p}{p^2}\right)^{1/2}\\
&\le \frac{C''}{\sqrt{n}} \sqrt{\oxi^{(8)}(1)}\, .
\end{align}
\end{proof}

We first control $\bX(\m;\Treg)$  on any finite collection of points
at a single point $\m$.
\begin{lemma}\label{lemma:HessianHamiltonian-OnePoint}
For any $\Delta>0$, exist a constant $\Treg_1(\Delta,\xi)$ depending uniquely on $\xi$
and a universal constant $C_*$, such that
the following holds for any fixed $\m\in(-1,1)^n$.

For any $\Treg\ge \Treg_1(\Delta,\xi)$, with probability at least
$1-2\exp(-n\Delta^2)$, we have (letting $Q=\|\m\|^2/n$):
\begin{align}
 -3(1+\Delta)\sqrt{\xi''(Q)Q^{-a(\xi)}}\, \bI_n \preceq 
  -\nabla^2 H_n(\m) + \frac{\Treg}{n}\m\m^{\sT}  \preceq
 3((1+\Delta)\sqrt{\xi''(Q)Q^{-a(\xi)}}+\Treg)\, \bI_n\, ,
\end{align}
where $a(\xi) = 1$ if $\xi^{(3)}(0)>0$, $\xi^{(2)}(0) = 0$, and $a(\xi) =0$
otherwise.
\end{lemma}
\begin{proof}
Recall the decomposition \eqref{eq:HessianAtAPoint} for the Hessian at a point $\m$.
As a consequence, defining   $\bY(\m,\Treg):= -\nabla^2 H_n(\m) + \frac{\Treg}{n}\m\m^{\sT}$,
we have
\begin{align}
 \bY(\m;\Treg)=\frac{1}{n}\Big(\Treg+\sqrt{\xi^{(4)}(Q)} \cdot \frac{g_0}{n^{1/2}}\Big)\,\m\m^{\sT}
+\sqrt{ \frac{\xi^{(3)}(Q)}{n^2}}\Big(\m\g^{\sT}+
\g\m^{\sT}\Big)+ \sqrt{\frac{\xi^{(2)}(Q)}{n}} \bW\, .
\end{align}
Note that, for any $s>0$, 
\begin{align}
\m\g^{\sT}+
\g\m^{\sT} \succeq -s\m\m^{\sT}-\frac{1}{s}\g\g^{\sT}\, .\label{eq:SquareDecomp}
\end{align}
On the event 
\begin{align}
\cuG:=\Big\{|g_0|\le \Delta\sqrt{n}, \;\; \|\g\|\le \sqrt{(1+\Delta)n}, \;\; \|\bW\|\le 
2(1+\Delta)\sqrt{n}\Big\}\, ,
\end{align}
we thus have 
\begin{align*}
 \bY(\m;\Treg)&\succeq\frac{1}{n}\Big(\Treg+\sqrt{\xi^{(4)}(Q)} \cdot \frac{g_0}{n^{1/2}}
 -s\sqrt{\xi^{(3)}(Q)} \Big)\,\m\m^{\sT}-\frac{1}{s}
 \sqrt{ \frac{\xi^{(3)}(Q)}{n^2}} \g\g^{\sT}
+ \sqrt{\frac{\xi^{(2)}(Q)}{n}}
 \bW\\
&\succeq\frac{1}{n}\Big(
\Treg-\Delta\sqrt{\xi^{(4)}(Q)} 
 -s\sqrt{\xi^{(3)}(Q)} 
 \Big)\,\m\m^{\sT}-\frac{1}{s}
  (1+\Delta)\sqrt{\xi^{(3)}(Q)} \id_n
-2(1+\Delta)\sqrt{\xi^{(2)}(Q)}
 \id_n \, .
\end{align*}
We then choose $s =  \sqrt{Q^{a}\xi^{(3)}(Q)/\xi^{(2)}(Q)}$, where $a = a(\xi)$, to get
\begin{align*}
\bY(\m;\Treg)&\succeq 
\frac{1}{n}\Big(\Treg-\Delta\sqrt{\xi^{(4)}(Q)} 
 -Q^{a/2}\frac{\xi^{(3)}(Q)}{\sqrt{\xi^{(2)}(Q)}}\Big)\,\m\m^{\sT}-
 3(1+\Delta)\sqrt{Q^{-a}\xi^{(2)}(Q)}
 \, \id_n\, .
\end{align*}
We finally choose $\Treg_1(\Delta,\xi) =\sup_{q\in [0,1]}\big\{\Delta\sqrt{\xi^{(4)}(q)} 
 +q^{a/2}\xi^{(3)}(q)/\sqrt{\xi^{(2)}(q)}\big\}$, whence
\begin{align*}
\bY(\m;\Treg)&\succeq -3(1+\Delta)\sqrt{Q^{-a}\xi^{(2)}(Q)}
 \, \id_n\, .
\end{align*}

Proceeding in the same way for the upper bounds, with Eq.~\eqref{eq:SquareDecomp}
replaced by $\m\g^{\sT}+
\g\m^{\sT} \preceq -s\m\m^{\sT}-s^{-1}\g\g^{\sT}$, we get
\begin{align*}
\bY(\m;\Treg)&\preceq 
\frac{1}{n}\Big(\Treg+\Delta\sqrt{\xi^{(4)}(Q)} 
 +Q^{a/2}\frac{\xi^{(3)}(Q)}{\sqrt{\xi^{(2)}(Q)}}\Big)\,\m\m^{\sT}+3(1+\Delta)\sqrt{Q^{-a}\xi^{(2)}(Q)}
 \, \id_n\\
 & \preceq \Big(2\Treg+3(1+\Delta)\sqrt{Q^{-a}\xi^{(2)}(Q)}\Big)\, \id_n\, .
\end{align*}
Finally the claim follows by noting that $\P(\cuG)\ge 1-C\exp(-n\Delta^2/C)$
by Gaussian concentration arguments \cite{Guionnet}.
\end{proof}

We are now in position to prove Lemma \ref{lemma:HessianHamiltonian}.
We let $N^n(\eta)$, $\eta<1$ be an $\eta\sqrt{n}$-net on $(-1,1)^n$ of cardinality
$|N^n(\eta)|\le (10/\eta)^n$.
Note that $\sup_{q\in(0,1]}q^{-a}\xi^{(2)}(q) = \xi^{(2)}(1)$. 
 Using Lemma \ref{lemma:RoughBoundThirdDerivative}
and Lemma \ref{lemma:HessianHamiltonian-OnePoint} (with $\Treg\ge \Treg_1(\Delta;\xi)$ as in the latter) 
and keeping using the notation $\bY(\m,\Treg):= -\nabla^2 H_n(\m) + \frac{\Treg}{n}\m\m^{\sT}$,
\begin{align*}
\P\Big(\min_{\m\in (-1,1)^n}&\lambda_{\min}\big(\bY(\m;\Treg)\big)\le -3(1+\Delta)\sqrt{\xi^{(2)}(1)}-
C\eta\sqrt{\oxi^{(8)}(1)}
\Big)\\ 
&\le |N^n(\eta)|\max_{\m\in N^n(\eta)} \P\Big(\lambda_{\min}\big(\bY(\m;\Treg)\big)\le 
-3(1+\Delta)\sqrt{\xi^{(2)}(1)}\Big)+ C_0\, e^{-n/C_0}\\
& \le \Big(\frac{10}{\eta}\Big)^n \,2\, e^{-n\Delta^2/C_1}+  C_0\, e^{-n/C_0}\, .
\end{align*}
We take $\eta = 1/\sqrt{C\oxi^{(8)}(1)}$ and $\Delta=C'\sqrt{\log(\oxi^{(8)}(1))}=:\Delta_{\xi}$,
$\Treg\ge \Treg_1(\Delta_{\xi};\xi) =:\Treg_0(\xi)$,
and obtain
\begin{align*}
\P\Big(\min_{\m\in (-1,1)^n}\lambda_{\min}\big(\bY(\m;\Treg)\big)\le -C''
\sqrt{\xi''(1)\log(\oxi^{(8)}(1))}\Big)\le Ce^{-n/C}\, .
\end{align*}
Proceeding analogously for the lower bound, we obtain
\begin{align}
 -K_{\delta}(\xi)\, \bI_n \preceq \big(-\nabla^2 H_n(\m) + \frac{\Treg}{n}\m\m^{\sT}\big)\preceq
 (K_{\delta}(\xi)+\Treg)\, \bI_n\, ,
\end{align}
and the desired claim \eqref{eq:lemma:HessianHamiltonian} holds because
$\|\D(\m)^{-1/2}\|_{\op}\le 1$

\end{document}